\documentclass[aos]{imsart}

\RequirePackage{amsthm,amsmath,amsfonts,amssymb}
\RequirePackage[numbers,sort&compress]{natbib}
\RequirePackage[colorlinks,citecolor=blue,urlcolor=blue]{hyperref}
\RequirePackage{graphicx}

\usepackage[dvips]{epsfig}
\usepackage{graphicx}
\usepackage{latexsym}
\usepackage{amsmath}
\usepackage{amsthm}
\usepackage{amssymb}
\usepackage{ulem}
\usepackage{bbm}

\usepackage{enumitem}

\usepackage{amsmath,amsthm,amsfonts,amssymb,amscd,amsbsy,dsfont,hyperref}

\usepackage{palatino}

\usepackage{eucal}
\usepackage{float}
\usepackage{xcolor}
\usepackage{multirow}

\startlocaldefs
\theoremstyle{plain}

\theoremstyle{remark}

\endlocaldefs

\newtheorem{thm}{Theorem}

\newtheorem{lem}[thm]{Lemma}

\newtheorem{cor}[thm]{Corollary}
\newtheorem{prop}[thm]{Proposition}
\newtheorem{rk}[thm]{Remark}

\DeclareMathOperator*{\barsum}{\overline{\sum}}

\newcommand{\norm}[1]{{\vert\kern-0.25ex\vert #1 
  \vert\kern-0.25ex\vert}}
\newcommand{\bignorm}[1]{{\big\vert\kern-0.25ex\big\vert #1 
  \big\vert\kern-0.25ex\big\vert}}

\newcommand{\opnorm}[1]{{\vert\kern-0.25ex\vert\kern-0.25ex\vert #1 
  \vert\kern-0.25ex\vert\kern-0.25ex\vert}}

\DeclareMathOperator{\1}{\mathbbm{1}}

\newcommand{\EX}[1] { \mathbb{E}  [ #1  ] }

\DeclareMathOperator{\Cov}{Cov}
\DeclareMathOperator{\Var}{Var}

\allowdisplaybreaks

\begin{document}
\begin{frontmatter}
\title{Statistical inference for rough volatility: Minimax theory}
\runtitle{Statistical inference for rough volatility: Minimax Theory}

\begin{aug}
\author[A]{\fnms{Carsten H.}~\snm{Chong}\ead[label=e1]{carstenchong@ust.hk}},
\author[B]{\fnms{Marc}~\snm{Hoffmann}\ead[label=e2]{hoffmann@ceremade.dauphine.fr}},
\author[C]{\fnms{Yanghui}~\snm{Liu}\ead[label=e3]{yanghui.liu@baruch.cuny.edu}},
\author[D]{\fnms{Mathieu}~\snm{Rosenbaum}\ead[label=e4]{mathieu.rosenbaum@polytechnique.edu}},
\and
\author[D]{\fnms{Gr\'egoire}~\snm{Szymansky}\ead[label=e5]{gregoire.szymanski@polytechnique.edu}}
\address[A]{Department of Information Systems, Business Statistics and Operations Management, The Hong Kong University of Science
and Technology\printead[presep={,\ }]{e1}}

\address[B]{Ceremade, University Paris Dauphine-PSL\printead[presep={,\ }]{e2}}
\address[C]{Department of Mathematics, Baruch College CUNY\printead[presep={,\ }]{e3}}
\address[D]{CMAP, \'Ecole Polytechnique\printead[presep={,\ }]{e4,e5}}
\end{aug}

\begin{abstract}
In recent years, rough volatility models have gained considerable attention in quantitative finance. In this paradigm, the stochastic volatility of the price of an asset has quantitative properties similar to that of a fractional Brownian motion with small Hurst index $H < 1/2$. In this work, we provide the first rigorous statistical analysis of the problem of estimating $H$ from historical observations of the underlying asset. We establish minimax lower bounds and design optimal procedures based on adaptive estimation of quadratic functionals based on wavelets. We prove in particular that the optimal rate of convergence for estimating $H$ based on price observations at $n$ time points is of order $n^{-1/(4H+2)}$ as $n$ grows to infinity, 
extending results that were known only for $H>1/2$. Our study positively answers  the question whether $H$ can be inferred, although it is the regularity of a latent process (the volatility); in rough models, when $H$ is close to $0$, we even obtain an accuracy comparable to usual $\sqrt{n}$-consistent  regular statistical models. 
\end{abstract}

\begin{keyword}[class=MSC]
\kwd[Primary ]{60G22}
\kwd{62C20}
\kwd[; secondary ]{62F12}
\kwd{62M09}
\kwd{62P20}
\end{keyword}

\begin{keyword}
\kwd{Rough volatility}
\kwd{fractional Brownian motion}
\kwd{wavelets}
\kwd{minimax optimality}
\kwd{pre-averaging}
\end{keyword}

\end{frontmatter}

\section{Introduction}

\label{sec:intro}

\subsection{The rough volatility paradigm}
Introduced for financial engineering purposes by \cite{gatheral2018volatility} in 2014, the essence of the rough volatility paradigm is that the volatility process $(\sigma_t)$, understood as a continuous-time process, exhibits irregular sample paths. As a prototype model for the log-price of an asset $(S_t)$, we let\\
\begin{equation}
\left\{
\begin{array}{l}
 S_t  = S_0+\int_0^t\sigma_s \,dB_s,  \\ \\
\sigma_t^2  = \sigma_0^2 \exp (\eta W_t^H), \label{roughvol1}
\end{array}
\right.
\vspace{3mm}
\end{equation}
where $(B_t)$ is a Brownian motion and $(W^H_t)$ is an independent  fractional Brownian motion with Hurst parameter $H\in(0,1)$. The constants $\sigma_0$ and $\eta$ are positive and considered as nuisance parameters. The key feature of rough volatility
modelling is that $H$ should be small, of order $10^{-1}$. Other rough volatility models involve more complex functionals of the fractional Brownian motion adapted to various applications, but with the same order of magnitude for $H$, see for example \cite{eleuch2019characteristic,gatheral2020quadratic}. In particular, for mathematical simplicity, we assume that $(B_t)$ and $(W_t^H)$ are independent, and this rules out leverage effects between price and volatility. The case of leverage effect is considered in our companion paper \cite{chong2022statistical}.\\ 

Historically, the rough volatility property was discovered using an empirical approach,  where data sets of interest are time series of daily measurement of historical volatility 
over many years and for thousands of assets, see \cite{bennedsen2016decoupling,gatheral2018volatility}. 
Its daily values were estimated by filtering high-frequency price observations, using various up-to-date inference methods for high-frequency data, all of them leading to analogous results. Several natural empirical statistics were computed on these volatility time series, in a model-agnostic way. Then it was shown in \cite{gatheral2018volatility} that strikingly similar patterns are observed when computing the same statistics in the simple Model \eqref{roughvol1} (actually a version of \eqref{roughvol1} where one considers a piecewise constant approximation of the volatility). For example, among the statistics advocating for rough volatility, empirical means of the structure function
\begin{equation} \label{eq: structure function}
\Delta \mapsto |\log (\sigma_{t+\Delta})- \log (\sigma_t)|^q,\;\;q>0,
\end{equation}
play an important role,
for $\Delta$ going from one day to several hundreds of days. 
For every value of $q$, the empirical counterpart of \eqref{eq: structure function}
systematically behaves like $\Delta^{Aq}$, where $A$ of order $10^{-1}$, for the whole range of $\Delta$. This scaling invariance is obviously reproduced if the volatility dynamics follow \eqref{roughvol1} with $H$ of order $10^{-1}$, thanks to the scaling property of fractional Brownian motion. In addition, the fact that this empirical fact also holds for large $\Delta$ somehow discards alternate stationary model candidates, where the moments of the increments no longer depend on $\Delta$ for large $\Delta$. It also kind of 
rules out the idea that this empirical scaling of the log-volatility increments could be an artefact due to the estimation error in the volatility process.\\ 

\subsection{Rough volatility in the literature}At first glance, the relevance of the parameter value in Model \eqref{roughvol1} may be surprising. It is in stark contrast with the first generation of fractional volatility models (FSV) where $H>1/2$ in a stationary environment, see \cite{comte1998long}. The goal of FSV models was notably to reproduce long memory properties of the volatility process and we know that fractional Brownian motion increments exhibit long memory when $H>1/2$. However, it turns out that when $H$ is very small, 
it remains consistent with the behaviour of financial data even on very long time scales, see \cite{gatheral2018volatility}. In addition to the stylised facts obtained from historical volatility, the data analysis obtained from implied volatility surfaces also support the rough volatility paradigm, see \cite{alos2007short,bayer2016pricing,fukasawa2021volatility,livieri2018rough}. In other words,  rough volatility models are, in financial terms, compatible with both historical and risk-neutral measures. Furthermore, rough volatility models can be micro-founded: in fact, only a rough nature for the volatility can allow financial markets to operate under no-statistical arbitrage conditions at the level of high-frequency trading, see \cite{eleuch2018microstructural,jusselin2020noarbitrage}. This has paved the way over the last few years to several new research directions in quantitative finance. Among other contributions, we mention risk management of complex derivatives, as considered for instance in \cite{abi2022characteristic,alos2022forward,eleuch2019characteristic,eleuch2018perfect,fukasawa2021hedging,gatheral2020quadratic,horvath2021deep,jacquier2021rough}, numerical issues as addressed in \cite{abi2019multifactor,bennedsen2017hybrid,callegaro2021fast,gassiat2022weak,mccrickerd2018turbocharging,rosenbaum2021deep}, asymptotic expansions are provided in \cite{chong2022short,eleuch2019short,forde2020rough,friz2022short,forde2017asymptotics,jacquier2022large} and theoretical considerations about the probabilistic
structure of rough volatility models as in \cite{abi2019affine,bayer2020regularity,cuchiero2019markovian,friz2022forests,gassiat2019martingale,gerhold2019moment,gatheral2019affine}.\\


Beyond the popularity of rough volatility models due to their remarkable ability to mimic data, the domain is certainly mature enough to take a step back with a view towards a mathematically sound statistical inference program.  This is the topic of the paper.  We undertake the rigorous statistical analysis of Model \eqref{roughvol1}, taken as a postulate or prototype of a rough volatility model. The intriguing question is obviously how well can one estimate $H$ from discrete historical data, if $H$ can be estimated at all! How does the  postulate of a small $H$ impact  a generic inference program? Put differently, how well  can we distinguish between two values of $H$ and therefore overcome the latent nature of the volatility and the noise in its estimation?\\

These fundamental questions have partially been addressed in the recent literature.
In \cite{fukasawa2019asymptotically}, one postulates the approximation
\begin{equation} \label{eq: dummy approximation}
\log \big(\widehat{\sigma_t^2}\big)\approx \log \big(\int_{(t-1)\Delta}^{t\Delta}\sigma^2_s ds\big)+\varepsilon_t,
\end{equation}
where $(\varepsilon_t)$ is a Gaussian white noise and $\widehat{\sigma^2_t}$ is the quadratic variation computed from high-frequency observations of 
the log-price over the interval $[(t-1)\delta,t\delta)$. Taking such an approximation for granted, an estimator of $H$ is obtained by spectral methods and a Whittle estimator is proved to be convergent and to satisfy a central limit theorem in a classical high-frequency framework. 
Unfortunately, the methodology is tightly related to the approximation \eqref{eq: dummy approximation},  which is not accurate enough to apply in Model \eqref{roughvol1}. Another interesting study is that of \cite{bolko2020GMM} that requires an ergodic framework and stationary assumptions. In the present contribution, we refrain from making such restrictive assumptions.
Our companion paper \cite{chong2022statistical} considers a similar setting to ours, with a slightly different class of models and a practitioner perspective in mind, 
focusing on the most useful rough volatility models like rough Heston and obtaining associated central limit theorems for estimating $H$. 

\subsection{Results of the paper} 

\subsubsection*{Piecewise constant volatility} In the spirit of rough volatility models inspired by financial engineering, we start our study with a version of Model \eqref{roughvol1}
where the volatility is piecewise constant at a certain time scale $\delta$. We consider $n$ regularly sampled observations of the price $(S_t)$ with $0<H<1$ over a fixed time interval $[0,T]$. Without loss of generality, we take $T=1$, hence the time step between observations is $\Delta = n^{-1}$. We assume $\Delta \ll \delta$, and this setting is compatible for instance with trading models that assume a constant volatility over a one day period (for $\delta$ of the order of one day). Mathematically, observing equivalently squared increments boils down to having spot variance observations (at the scale where the volatility is assumed to be constant) multiplied by noise. Taking logarithm reduces our problem to the setup of Gloter and Hoffmann \cite{gloter2007estimation}, where the estimation of the Hurst parameter for a fractional Brownian motion observed with additive measurement error is studied for $H>1/2$. The approach of \cite{gloter2007estimation} is based on the scaling properties of wavelet-based energy levels of fractional Brownian motion 
$$Q_j = \sum_k (d_{j,k})^2$$
where the $d_{j,k}$ are the wavelet coefficients (in an arbitrary wavelet basis) at a dyadic resolution $j$. Indeed, we prove that $2^{2jH}Q_{j}$ converges to some constant with an explicit rate as $j \rightarrow \infty$, and this yields a strategy for recovering $H$ based on an estimation of the ratio $Q_{j+1}/Q_j$, provided it can be estimated accurately enough from discrete data. Furthermore, the multiresolution nature of wavelet decompositions enables one to select the optimal resolution $j$ via an adaptive thresholding procedure. This is somehow simpler than   
 using the scaling properties of $p$-variations of the data, although both approaches are similar in spirit. In \cite{gloter2007estimation} the energy levels $Q_j$ are estimated from increments of the observations, a strategy that unfortunately cannot be directly applied here when $H$ is small: the roughness of the volatility paths induces a large bias in the estimation of the energy levels. We mitigate this phenomenon by using a pre-averaging technique introduced in \cite{szymanski2022optimal}.  More precisely, we show that the energy levels computed from the pre-averaged spot volatility process  still possess nice scaling properties, paving the way to construct a new estimator that achieves the rate
$\max(n^{-1/(4H+2)}, \delta^{1/2})$. This rate is indeed minimax optimal for any $H \in (0,1)$. In particular, we obtain -- from a rigorous statistical viewpoint --  that estimating $H$ in this toy model with rough piecewise constant volatility  has the same complexity as classical regular model, the rate $n^{-1/(4H+2)}$ is close to $n^{-1/2}$ in the rough limit $H \rightarrow 0$ whenever $\delta$ is small enough.\\

\subsubsection*{The general case}


In the generic setting of Model \eqref{roughvol1}, the law of the price increments is more intricate. The piecewise constant volatility part of the previous toy model is now replaced by local averages of the volatility process. As a matter of fact, the inference of $H$ has been studied  in \cite{rosenbaum2008estimation} for $H>1/2$: energy levels are computed from price increments and are shown to exhibit a scaling property around a stochastic limit; the same strategy as in \cite{gloter2007estimation} can then be undertaken.  However, this approach completely fails in the rough case $H<1/2$: over a short time interval, for small $H$,  local averages of  volatility are bad approximations of   spot volatility, and this is a crucial element in \cite{rosenbaum2008estimation} that requires the condition $H>1/2$. We overcome this difficulty as follows: by combining empirical means techniques used in \cite{gatheral2018volatility} together with our results in the previous piecewise constant case, we compute energy levels from the logarithm of the squared price increments and not from price increments. However, 
over a small time interval of the form $[i\delta, (i+1)\delta)$, the logarithm of price increments involves
\begin{align*}
\log \Big( \frac{1}{\delta} \int_{i\delta}^{(i+1)\delta} \exp(\eta W^H_s) ds \Big)
\end{align*}
which does not enjoy nice properties when $H$ is small: indeed; the roughness of the trajectories makes this quantity far from a discretized Gaussian process $\eta W^H_{i\delta}$. This creates a bias when computing $H$ from a ratio of energy levels and the scaling law is no longer exact: additional terms of order $2^{-2ajH}$ for some $a > 1$ appear. They can be removed thanks to a suitable bias correction procedure. To that end, we start with a pilot estimator with no bias correction on the energy levels. We plug in the pilot estimator to correct the initial energy level estimation and compute a second estimator with an improved rate of convergence. By bootstrapping this procedure $O(H^{-1})$ many times, we achieve the minimax optimal rate of convergence $n^{-1/(4H+2)}$ for every $H >0$.

\subsubsection*{Organisation of the paper}


The model with piecewise constant volatility is considered in Section \ref{sec:piecewise}. This assumption is removed in Section \ref{sec:general} where we consider the general model \eqref{roughvol1}. Discussions are gathered in Section \ref{sec:discussion}. Sections \ref{sec:proof:lower_multiscale}  to \ref{sec:estimator:proof} are devoted to the proofs. In Section \ref{sec:proof:lower_multiscale}, we prove the lower bound for the piecewise constant volatility model. Note that we formally have a hidden Markov chain, since, conditional on $(W_t^H)$, we have a Gaussian Markov chain (actually with independent increments). Yet, due to the absence of ergodicity or stationarity, hidden Markov chain techniques cannot be applied to handle a likelihood, the key to understanding lower bounds. Instead, we revisit the hidden pathwise strategy of \cite{hoffmann2002rate} and \cite{gloter2004stochastic}, later developed in 
\cite{gloter2007estimation} and for stochastic volatility by \cite{rosenbaum2008estimation} for $H>1/2$. Here, the case $H<1/2$ is substantially more intricate and requires delicate Gaussian calculus. Section \ref{sec:proof:upper_multiscale}, establishes that the estimator for $H$ constructed in Section \ref{sec:piecewise} in the piecewise constant volatility model is minimax optimal. This requires us to establish an appropriate scaling law for an averaged version of the energy levels (namely \eqref{eq:deviation_Q}) that is based on the Gaussian calculus; see Lemmas \ref{lem:kappa} and \ref{lemma:self_d}. Both Sections \ref{sec:proof:lower_multiscale} and \ref{sec:proof:upper_multiscale} are relatively straightforward in terms of testing and estimation strategy, once the techniques of \cite{gloter2007estimation} are well understood. Yet, the computations are more involved. Also, they lay the necessary ground, both technically and methodologically, for solving the significantly more difficult case of the general continuous volatility Model \eqref{roughvol1}. Its study is undertaken in Sections \ref{sec:proof:optimality} to \ref{sec:estimator:proof}, where both the lower bounds and upper bounds are proved. As far as the lower bound is concerned, only slight modifications are required. This is no longer the case for the upper bound, where a new type of scaling law for the energy levels must be found in order to follow our general approach. In particular, the scaling law of Proposition \ref{prop:energy}  and its estimation in order to mimick the results of \cite{gloter2007estimation} are the core of the paper. It is based on an appropriate expansion across scales of the energy levels that strongly uses the representation of the volatility as an exponential of fractional Brownian motion, and requires intricate Gaussian calculus, mostly based on Isserlis' theorem. The tools are presented  in Appendix  \ref{sec: first app} and Appendix \ref{sec: second app}, which also contain auxiliary  technical results.

\section{The piecewise constant rough volatility model}
\label{sec:piecewise}

\subsection{Model and notation}
\label{sec:multiscale_model}

We start with a simplified version of Model \ref{roughvol1}, as a prototype rough volatility model where the volatility is assumed to be piecewise constant at a given scale $\delta >0$. (In financial terms, $\delta$ is thought to be of the order of one trading day, {\it i.e.}\   volatility is assumed to be constant over a day.)
Given parameters $H \in (0,1)$ and $\eta >0$, on a rich enough probability space $(\Omega, \mathcal A, \mathbb P_{H,\eta})$, the price model takes the form
\begin{equation}
\left\{
\begin{array}{l}
 S_t  = S_0 + \int_0^t \sigma_s \,dB_s,  \\ \\
\sigma_t^2  = \sigma_0^2 \exp (X_t),\\ \\
X_t =  \eta W_{\lfloor t\delta^{-1} \rfloor \delta}
^H,
 \label{roughvol2}
\end{array}
\right.
\vspace{1mm}
\end{equation}
where $(B_t)$ is a Brownian motion and $(W^H_t)$ is an independent  fractional Brownian motion  with Hurst parameter $H\in(0,1)$ and $\eta, \sigma_0 >0$ are nuisance parameters. With no loss of generality, we further assume $\sigma_0^2=1$.\\

We observe a sample path of $(S_t)$ at $n+1$ discrete time points:
$$S_0,\;S_{1/n},\; S_{2/n},\;\ldots, S_1.$$
To simplify  notation and computations, with no loss  of generality, we assume that $n$ and $\delta$ live on a dyadic scale in the following sense: for some integer $N \geq 0$, we have
$$n = 2^{N}\;\;\text{and}\;\; m = n\delta = 2^N\delta\;\;\text{is an integer}.$$ 
Writing $\mathcal{A}^{n} \subset \mathcal A$ for the $\sigma$-algebra generated by the data $(S_{i/n})_{0 \leq i\leq n}$, we obtain a (sequence of) statistical experiment(s)
$$\mathcal E^n = \big(\Omega, \mathcal A^n, \big(\mathbb P_{H,\eta}\big)_{(H,\eta) \in \mathcal D} \big),$$
where the parameter $(H,\eta)$ lies in a compact (with non-empty interior) $\mathcal D \subset (0,1) \times (0,\infty)$. 

\begin{rk}[About the asymptotic regime]
Our asymptotic regime is dictated by a time step $\Delta = \Delta_n  = n^{-1}$. It formally looks quite much as a classical {\it high-frequency} approach, a dominant choice in the econometric literature \cite{ait2014high} . Yet, this is a bit misleading, and we do not think our model in terms of high-frequency data. We rather work at a level of a diffusive scale: in particular we have no microstructure noise and our angle is that we gain no information by any kind of ergodicity input, hence the lack of stationarity assumption at this level, which is irrelevant. In that setting, we recover information via sample path reconstruction and latent volatility filtering by nonparametric denoising only. (In particular, we rule out any drift effect by Girsanov since our timespan of the experiment does not grow to infinity). What ultimately matters in terms of statistical information is that the number of data $n$ is large. If we observe the process over the timespan $[0,T]$ at frequency $\Delta^{-1}$, this means that we consider $n = T/\Delta^{-1}$ as large, but not necessarily $T$ in mathematical terms, hence we further set $T=1$ with no loss of generality. This is compatible with a Donsker type approach, {\it i.e.} what happens when one considers a macroscopic diffusive model obtained from a microscopic one, extracted from Hawkes processes prices for instance \cite{jaisson2015limit}. In practice, we think of $\Delta$ to be of the order of a few minutes (when microstructure effects vanish) to one day, but this is only a broad guiding principle at this stage. 
\end{rk}


\subsection{Minimax lower bound}
\label{subsec:minimaxmultiscale}

The rate $v_n$ is said to be a lower rate of convergence for estimating $H$ over $\mathcal D$ in the experiment $\mathcal E^n$ if there exists $c>0$ such that
\begin{align*}
\liminf_{n\to\infty} 
\inf_{\widehat H_n}
\sup_{(H,\eta) \in \mathcal{D}}
\mathbb{P}_{H,\eta}
(
v_n^{-1}|\widehat H_n-H|\geq c
)
>
0,
\end{align*}
where the infimum is taken over all $\mathcal{A}^{n}$-measurable random variables $\widehat H_n$ ({\it i.e.} all estimators).

\begin{thm}
\label{thm:lower_multiscale}
The rate  
\begin{align*}
v_n(H) = \max\big(n^{-1/(4H+2)}, \delta^{1/2}\big)
\end{align*} 
is a lower rate of convergence for estimating $H$ in $\mathcal E^n$. Moreover, 
$$w_n(H) = \max\big(n^{-1/(4H+2)} \log n, \delta^{1/2}\big)$$ is a lower rate of convergence for estimating $\eta$ over $\mathcal D$ in $\mathcal E^n$ (with obvious modifications in the definition).
\end{thm}

In particular, this result encompasses arbitrarily small values for $H$. We also retrieve the lower bound of \cite{gloter2007estimation} and \cite{rosenbaum2008estimation} in the case $H>1/2$. Further comments are given in Section \ref{sec:discussion}. The proof of Theorem \ref{thm:lower_multiscale} is postponed to Section \ref{sec:proof:lower_multiscale}.

\subsection{Construction of an asymptotically minimax estimator}
\label{subsec:multiscale:construction}

\subsubsection*{Translation into a signal plus noise model}
Recall that we set $m = n\delta_n$, which is assumed to be an integer. 
For $0 \leq i < \delta_n^{-1}$, we have
\begin{align*}
\sum_{j=1}^m \big(S_{(im+j)/n} - S_{(im+j-1)/n}\big)^2 
=
 \sigma^2_{i\delta_n} \sum_{j=1}^m 
 \big(B_{(im+j)/n} - B_{(im+j-1)/n}\big)^2, 
\end{align*}
so that
\begin{align*}
\widehat{X}_{i,n} = \log\Big(\frac{n}{m} \sum_{j=1}^m \big (S_{(im+j)/n} - S_{(im+j-1)/n} \big)^2 \Big) = \eta W^H_{i\delta_n} + \varepsilon_{i,n,m},
\end{align*}
where the $(\varepsilon_{i,n,m})_i$ are independent, each $\varepsilon_{i,n,m}$ being distributed as 
$$\bar \varepsilon_m = \log \big ( m^{-1} \sum_{j=1}^m (\bar{\xi}_j)^2 \big)$$ for i.i.d. standard Gaussian $(\bar{\xi}_j)_j$. Therefore the random variable $n \delta_n \exp ( \varepsilon_{i,n,m} )$ has a chi-square law with $n\delta_n$ degrees of freedom. By Lemma \ref{lem:log_mom_chi2} below, we moreover have
\begin{align*}
\Var \big(\bar \varepsilon_m \big) &= \psi^{(1)}(\tfrac{m}{2}) \leq C m^{-1}
\;\;\;\;\text{and}\;\;\;\;
\mathbb{E} \big[
\bar \varepsilon_m^4
\big]
\leq C m^{-2}
\end{align*}
for some $C>0$, where the function $\psi^{(1)}$ is explicitly given in the Appendix  \ref{sec:log_mom_chi2}.\\ 

\subsubsection*{Estimating energy levels}
We now introduce the energy levels of the process $(X_t)$. For a scale $j \geq 2$, a location parameter $k < 2^{j-1}$ and an averaging index $p\geq 0$, we define
\begin{align*}
d_{j,k,p} = 
\frac{1}{2^{j/2+p}} 
\sum_{l=0}^{2^p-1}
\Big(
X_{(k+l2^{-p})2^{-j}}
-
2X_{(k+1+l2^{-p})2^{-j}}
+
X_{(k+2+l2^{-p})2^{-j}}
\Big)
\end{align*}
and
\begin{align*}
Q_{j,p} = \sum_{k=0}^{2^{j-1}-1} (d_{j,k,p})^2.
\end{align*}
Note that the last sum stops at index $k = 2^{j-1} - 1$ and is defined only for $j \geq 2$ because our observations stop at time $1$. In view of these definitions, we introduce their empirical counterparts. We set
\begin{align*}
\widetilde{d}_{j,k,p,n} = 
\frac{1}{2^{j/2+p}} 
\Big(
\sum_{l=0}^{2^p-1}
\widehat{X}_{2^{N-j}(k+l2^{-p}),n}
-
2\widehat{X}_{2^{N-j}(k+1+l2^{-p}),n}
+
\widehat{X}_{2^{N-j}(k+2+l2^{-p}),n}
\Big),
\end{align*}
which is $\mathcal{A}^{n}$-measurable provided that $p+j \leq N$. Note that we have the representation
$$\widetilde{d}_{j,k,p,n} = d_{j,k,p} + e_{j,k,p,n},$$
with
\begin{align}
\label{eq:def:e}
e_{j,k,p,n} 
&=
\frac{1}{2^{j/2+p}} 
\sum_{l=0}^{2^p-1}
\Big(
\varepsilon_{2^{N-j}(k+l2^{-p}),n,m}
-
2\varepsilon_{2^{N-j}(k+1+l2^{-p})2^{-j},n,m}
+\varepsilon_{2^{N-j}(k+2+l2^{-p}),n,m}\Big).
\end{align}

However $(\widetilde{d}_{j,k,p,n})^2$ is not a clever choice for estimating $d_{j,k,p}$ due to the presence of the noise term $e_{j,k,p,n}$. Following \cite{gloter2007estimation}, we can mitigate the effect of the noise by removing the bias $\EX{e_{j,k,p,n}^2} = 6\Var(\bar \varepsilon_m) 2^{-j-p}$ which is explicitly computable. We thus finally estimate the energy levels by
\begin{align}
\label{eq:def:Qhat}
\widehat{Q}_{j,p,n} = \sum_{k=0}^{2^{j-1}-1} \widehat{d^2_{}}_{\!\!\!j,k,p,n},
\;\;
\text{with}\;\;
\widehat{d^2}_{\!\!\!j,k,p,n} = (\widetilde{d}_{j,k,p,n})^2 - 6\Var(\bar \varepsilon_m) 2^{-j-p}.
\end{align}

Pick $\nu_0 > 0$. We obtain a family of estimators of $H$ by setting 
\begin{align*}
\widehat{H}_n = -\frac{1}{2}\log \bigg[
\frac{\widehat{Q}_{J_n^*+1,N-J_n^*-1,n} }{\widehat{Q}_{J_n^*,N-J_n^*-1,n} }
 \bigg]
 \end{align*}
where the data-driven rule for selecting the energy level is given by
\begin{align*}
J_n^* = \max \big\{ 2 \leq j \leq N-1: \widehat{Q}_{j,N-j-1,n} \geq \nu_0 2^jn^{-1} \big\}
\end{align*}
with the convention that $J_n^* = 2$ when this set is empty. The consistency and convergence rate of $\widehat H_n$ is given in the following theorem, discussed in Section \ref{sec:discussion} and proved in Section \ref{sec:proof:upper_multiscale}.
\begin{thm}
\label{thm:upper_multiscale}
The rate $v_n(H) = \max(n^{-1/(4H+2)},\delta^{1/2})$ is achievable for estimating $H$ over $\mathcal{D}$. 
More precisely, for $\nu_0 < \inf_{(H, \eta) \in \mathcal{D}} \eta^2 \kappa_0(H) 2^{2H}$ with $\kappa_0(H) = 4 - 
2^{2H}$, the family of random variables
$$\big(v_n^{-1} (\widehat{H}_n - H)\big)_{n \geq 1}$$
is tight in $\mathbb{P}_{H,\eta}$-probability, uniformly over $\mathcal{D}$.
\end{thm}

\section{The general rough volatility model}\label{sec:general}

\subsection{Model and main results}
\label{subsec:model}

We now consider Model \eqref{roughvol1}, with the same accompanying statistical experiment as in the previous section, but we do not require that the volatility process 
$(\sigma_t)$ is piecewise constant: we now have
\begin{align*}
\sigma_t^2 = \sigma_0^2 \exp ( \eta W^H_t ),
\end{align*}
where $W^H$ is a fractional Brownian motion with Hurst index $H$ and independent from $B$. We will write $\mathbb{E}_{H, \eta}$ for the expectation under the probability measure $\mathbb{P}_{H, \eta}$. Without loss of generality, we take $\sigma_0 = 1$. We assume that $(H, \eta)$ lies in a compact subset $\mathcal{D}$ of $(0, 3/4) \times (0, \infty) $. 
The upper bound $H < 3/4$ is a bit artificial and done mainly for technical purposes. It can be improved to $ H < 1$ using second-order increments, see Section \ref{sec:discussion}.\\

We next state a minimax lower bound. We then gather some results used later in the construction of an estimator in Section \ref{subsec:energy}. 
In Section \ref{subsec:estimator}, we build an estimator achieving the lower bound and thus establishing the minimax rate of convergence in the general rough volatility model \eqref{roughvol1}, with statistical experiment generated by $(S_{i/n}, i=0,\ldots, n)$.

\subsection{Minimax optimality}
\label{subsec:minimaxgeneral}

Recall from Section \ref{subsec:minimaxmultiscale} that the rate $v_n$ is said to be a lower rate of convergence over $\mathcal{D}$ for estimating $H$ if there exists $c>0$ such that
\begin{align*}
\liminf_{n\to\infty} 
\inf_{\widehat H_n}
\sup_{(H,\eta) \in \mathcal{D}}
\mathbb{P}_{H,\eta}
(
v_n^{-1}|\widehat H_n-H|\geq c
)
>
0,
\end{align*}
where the infimum is taken over all estimators $\widehat H_n$, and similarly for $\eta$. We obtain an analogous lower bound as in Theorem \ref{thm:lower_multiscale}. Its proof is given in Section \ref{sec:proof:optimality}.

\begin{thm}
\label{thm:optimality}
The rates $v_n(H) = n^{-1/(4H+2)}$ and $w_n(H) = n^{-1/(4H+2)} \log n $ are lower rates of convergence for estimating $H$ and $\eta$ respectively over $\mathcal D$.
\end{thm}

\subsection{Energy levels of the log integrated volatility}
\label{subsec:energy}
We set 
$$H_- = \min H,\; H_+ = \max H,\; \eta_- = \min \eta,\;\eta_+ = \max \eta,$$
where $\max$ and $\min$ are taken over $(H,\eta) \in \mathcal D$. 
Fix $j \geq 0$ and $p \geq 0$. For $0 \leq k < 2^p$, we introduce the pre-averaged local energy levels of the log-integrated-volatility as follows:
\begin{align*}
d_{j,k,p} = 2^{-p-j/2}
\sum_{l=0}^{2^p-1} \Big(
\log \Big( \int_{(k+1)2^{-j}+l2^{-j-p}}^{(k+1)2^{-j}+(l+1)2^{-j-p}} \sigma_u^2 du \Big)
-
\log \Big( \int_{k2^{-j}+l2^{-j-p}}^{k2^{-j}+(l+1)2^{-j-p}} \sigma_u^2 du \Big) \Big).
\end{align*}
The corresponding energy levels are obtained via
\begin{align*}
Q_{j,p} = \sum_{k<2^{j-1}} (d_{j,k,p})^2.
\end{align*}

These energy levels differ from those of \cite{rosenbaum2008estimation} since they are not obtained from the integrated realised variance of the price but rather from their logarithm. They also scale as $2^{-2jH}$ as in \cite{rosenbaum2008estimation}, but the logarithm generates two major differences in this scaling. On the one hand, we get rid of the 
stochastic limit appearing in \cite{rosenbaum2008estimation}
 but on the other hand, this scaling is not exact and additional terms appear. More precisely, the following concentration property is proved in Section \ref{subsec:energy:Q}.

\begin{prop}
\label{prop:energy}
There exist functions of $H$ denoted by $H \mapsto \kappa_{p,a}(H)$ and explicitly given in Equation \eqref{eq:def_kappas} such that if $S \geq 1/(4H_-) + 1/2$ and $S > H_+/(2H_-) - 1/2$, we have
\begin{align}
\label{eq:energy}
\mathbb{E}_{H, \eta} \Big [\big(
	Q_{j,p}
	-
	\sum_{a=1}^{S} \eta^{2a} 2^{-2aHj} \kappa_{p,a}(H)
\big)^2
\Big ]
\leq 
C 2^{-j(1+4H)}
\end{align}
for some constant $C$ depending only on $S$.
\end{prop}

Although the functions $\kappa$ are explicit, their numerical implementation is quite involved and requires the use of Isserlis' Theorem (see Theorem \ref{thm:gaussian_moments}). The following two lemmas give a control of the functions $\kappa$. They will be useful in the construction of the estimator. Their proofs are delayed until Section \ref{subsec:energy:kappa}.
\begin{lem}
There exist some positive constants $c_{-,1}$ and $c_{+,1}$ such that for any $p \geq 1$ and any $H_- \leq  H \leq H_+$,
\begin{align}
\label{eq:kappa1:bound}
c_{-,1} \leq \kappa_{p,1}(H) \leq c_{+,1}.
\end{align}

Moreover, there exists some positive constant $c_{\cdot, S}$ such that for any $p \geq 1$, any $ 2 \leq a \leq S$ and any $H_- \leq  H \leq H_+$,
\begin{align}
\label{eq:kappap:bound}
| 2^{(2a-1)Hp} \kappa_{p,a}(H) | < c_{\cdot, S}.
\end{align}
\end{lem}

\begin{lem}
\label{lem:bound:kappap}
The functions $\kappa_{p,a}$ are differentiable on $(H_-, H_+)$ and for any $a > 0$ there exists $c_a$ such that
\begin{align*}
\sup_{p\geq 0} |\kappa_{p,a}'(H)| \leq c_a.
\end{align*}
\end{lem}

Our next result shows that the rescaled energy levels $2^{2jH} Q_{j,N-j-1}$ are essentially bounded above and below in probability. Its proof is given in Section \ref{subsec:energy:bound}.

\begin{prop}
\label{prop:bound_energy}
Let $\varepsilon > 0$. Then there exist $0 < r_-(\varepsilon) < r_+(\varepsilon)$, $J_0(\varepsilon) > 0$ and $N_0(\varepsilon)$ such that for $N \geq N_0(\varepsilon)$, we have
\begin{align*}
\sup_{H, \eta}
\mathbb{P}	_{H,\eta}
\Big(
\inf_{J_0 \leq j \leq N-1} 2^{2jH} Q_{j,N-j-1} \leq r_-(\varepsilon)
\Big)
\leq \varepsilon
\end{align*}
and
\begin{align*}
\sup_{H, \eta}
\mathbb{P}_{H,\eta}
\Big(
\sup_{J_0 \leq j \leq N-1} 2^{2jH} Q_{j,N-j-1} \geq r_+(\varepsilon)
\Big)
\leq \varepsilon.
\end{align*}
\end{prop}

Because of the additional terms appearing in Equation \eqref{eq:energy}, we want to add a bias correction to the energy levels. To that end, for $S > 0$, $\nu > 0$ and $I > 0$, we define
\begin{align}
\label{eq:B}
Q_{j,p}^{(S)}(I,\nu)
=
Q_{j,p}
-
B_{j,p}^{(S)}(I,\nu),
\end{align}
where
\begin{align*}
B_{j,p}^{(S)}(I,\nu)= 	
	\sum_{a=2}^{S} \nu^{2a} 2^{-2aIj} \kappa_{p,a}(I).
\end{align*}
Note that unlike in Proposition \ref{prop:energy}, this sum starts at $a = 2$ so that we have
\begin{align*}
\mathbb{E} \Big [\big(
	Q_{j,p}^{(S)}(H,\eta) 
	-
	\eta^{2} 2^{-2Hj} \kappa_{p,1}(H)
\big)^2
\Big ]
\leq 
C 2^{-j(1+4H)},
\end{align*}
which now has the same behaviour as the term $Q_j$ in Proposition 3 of \cite{gloter2007estimation}. Therefore, we can derive bounds for $2^{2jH} Q_{j,N-j-1}^{(S)}(H,\eta) $ similarly to Proposition \ref{prop:bound_energy}. The proof, which is analogous to that of Proposition \ref{prop:bound_energy} in Section \ref{subsec:energy:bound}, is omitted.
\begin{prop}
\label{prop:bound_energy_corrected}
Let $\varepsilon > 0$. Then there exist $0 < r_-^{(S)}(\varepsilon) < r_+^{(S)}(\varepsilon)$, $J_0^c(\varepsilon) > 0$ and $N_0^c(\varepsilon)$ such that for $N \geq N_0^c(\varepsilon)$, we have
\begin{align*}
\sup_{H, \eta}
\mathbb{P}	_{H,\eta}
\Big(
\inf_{J_0 \leq j \leq N-1} 2^{2jH} Q_{j,N-j-1}^{(S)}(H,\eta)  \leq r_-^{(S)}(\varepsilon)
\Big)
\leq \varepsilon
\end{align*}
and
\begin{align*}
\sup_{H, \eta}
\mathbb{P}	_{H,\eta}
\Big(
\sup_{J_0 \leq j \leq N-1} 2^{2jH} Q_{j,N-j-1}^{(S)}(H,\eta)  \geq r_+^{(S)}(\varepsilon)
\Big)
\leq \varepsilon.&
\end{align*}
\end{prop}

We conclude this section by giving explicit Lipschitz bounds on the functions $B_{j,p}^{(S)}$:

\begin{lem}
\label{lem:dist:B}
There exists $c_B > 0$ such that for any $\eta_1, \eta_2 \in [\eta_-, \eta_+]$ and $H_1, H_2 \in [H_-, H_+]$, we have
\begin{align*}
|
B_{j,p}^{(S)}(\eta_1, H_1) 
-
B_{j,p}^{(S)}(\eta_2, H_2) 
|
\leq
c_B 2^{-4(H_1 \wedge H_2)j}
(
j |H_1 - H_2|
+
|
\eta_1 
-
\eta_2
|
).
\end{align*}
\end{lem}

This lemma is proved in Section \ref{sec:proof:lem:dist:B}.

\subsection{Construction of the estimator}
\label{subsec:estimator}

We now turn to the task of building a good estimator of $Q_{j,p}$ based on the price increments. Here, we fix an integer $S$ such that $S \geq 1/(4H_-) + 1/2$ and $S > H_+/(2H_-) - 1/2$ so that the result of Proposition \ref{prop:energy} is in force.\\

First, notice that for a fixed $(j,p)$ such that $j+p \leq N$, the price increment $S_{k2^{-j}+(l+1)2^{-j-p}} - S_{k2^{-j}+l2^{-j-p}}$ is $\mathcal{A}^{n}$-measurable and
\begin{align}
\label{eq:price_increments}
\Big( \big(S_{(l+1)2^{-j-p}} - S_{l2^{-j-p}}\big)^2 \Big)_{l\geq 0} = \Big( \int_{l2^{-j-p}}^{(l+1)2^{-j-p}} \sigma_u^2 du \; \xi_{j,p,l}^2\Big)_{l\geq 0} 
\end{align}
holds in distribution, where the $(\xi_{j,p,l})_l$ are i.i.d. standard Gaussian variables, independent of the volatility process $(\sigma_t)_t$. We estimate $d_{j,p,k}$ by 
\begin{align*}
\widetilde{d}_{j,p,k,n} = 
&2^{-p-j/2}
\sum_{l=0}^{2^p-1}
\log \Big( 
\big(S_{(k+1)2^{-j}+(l+1)2^{-j-p}} - S_{(k+1)2^{-j}+l2^{-j-p}}\big)^2
\Big)
\\&- 2^{-p-j/2}
\sum_{l=0}^{2^p-1}
\log \Big( 
\big(S_{k2^{-j}+(l+1)2^{-j-p}} - S_{k2^{-j}+l2^{-j-p}}\big)^2
\Big).
\end{align*}
Thanks to \eqref{eq:price_increments}, we have the representation 
$$
\widetilde{d}_{j,p,k,n} = d_{j,p,k} + e_{j,k,p,n},
$$
where
\begin{align*}
e_{j,k,p,n} = 2^{-p-j/2}
\sum_{l=0}^{2^p-1} \big(
\log ( 
\xi_{j,p,(k+1)2^{p}+l}^2
)
-
\log ( 
\xi_{j,p,k2^{p}+l}^2
) \big).
\end{align*}
In the same way as in \eqref{eq:def:Qhat}, we introduce a bias correction for estimating $(d_{j,k,p})^2$ due to the term $e^2_{j,k,p,n}$ by setting
\begin{align*}
\widehat{d^2}_{j,p,k,n} = (\widetilde{d}_{j,p,k,n})^2 - 2^{-j-p+1} \Var(\log \xi^2),
\end{align*}
where $\xi$ is a standard normal random variable. We obtain an estimator of the energy level by setting
\begin{align*}
\widehat{Q}_{j,p,n} = \sum_k \widehat{d}^2_{j,p,k,n}. 
\end{align*}
We are ready to construct our estimator for $H$. We proceed recursively in order to achieve optimality.  
A first estimator of $(H, \eta)$ is defined by $(\widehat{H}^{(0)}_n, \widehat{\eta}^{(0)}_n)$, with
\begin{align*}
\widehat{H}^{(0)}_n &= 
\min\Big(
\max\big( 
-\frac{1}{2} \log_2 \big[
\frac{\widehat{Q}_{J_n^{*} + 1, N-J_n^{*} - 1, n}}{\widehat{Q}_{J_n^{*}, N-J_n^{*} - 1, n}}
\big] 
, H_-\big), H_+ \Big)
\end{align*}
with  $J_n^* = \max \big\{j:\widehat{Q}_{j,N-j-1,n} \geq 2^jn^{-1} \big\}$,
and then
\begin{align*}
\widehat{\eta}^{(0)}_n = 
\min \Big(\max \Big( 
\Big( \frac{\widehat{Q}_{\widehat{j}_n,N-\widehat{j}_n,n} 2^{2\widehat{j}_n\widehat{H}^{(0)}_n}}{ \kappa_{N-\widehat{j}_n, 1}(\widehat{H}^{(0)}_n)} \Big)^{1/2}
, \eta_- \Big)
, \eta_+ \Big)
\end{align*}
where
$\widehat{j}_n
= \big\lfloor \frac{1}{2\widehat{H}^{(0)}_n+1}\log_2 n \big\rfloor.
$
The projection onto $[H_-, H_+] \times [\eta_-, \eta_+]  \supseteq \mathcal D$ guarantees the stability of the bias correction procedure. We then refine this estimator by correcting for a bias induced by the functions $\kappa_{p,a}$ that are necessary to achieve a sufficiently accurate scaling as in \eqref{eq:energy} of Proposition \ref{prop:energy}. For $\nu > 0$ and $0 < I < 1$, put
$$\widehat{Q}_{j,p,n}^{(S)}(I,\nu)
=
\widehat{Q}_{j,p,n}
-
B_{j,p}^{(S)}(I,\nu),
$$ where $B^{(S)}_{j,p}$ is defined in \eqref{eq:B}. Let
\begin{align}
\label{eq:def:bias_corrected_estimator_H}
\widehat{H}_n^{c}(I,\nu) &= 
\min \Big( 
\max\big( 
-\frac{1}{2} \log_2 \big[
\frac{\widehat{Q}^{(S)}_{J_n^{*c}(I,\nu) + 1, N-J_n^{*c}(I,\nu) - 1, n}(I,\nu)}{\widehat{Q}^{(S)}_{J_n^{*c}(I,\nu), N-J_n^{*c}(I,\nu) - 1, n}(I,\nu)}
\big]
, H_-
\big)
, H_+ \Big),
\end{align}
with
\begin{align}
\label{eq:def:bias_corrected_j}
J_n^{*c}(I,\nu) = \max \big\{j:\widehat{Q}^{(S)}_{j, N-j - 1, n}(I,\nu) \geq 2^jn^{-1} \big\}
\end{align}
and
\begin{align}
\label{eq:def:bias_corrected_estimator_eta}
\widehat{\eta}^{c}_n(I,\nu) &= 
\min \Big( 
\max \Big( \frac{\widehat{Q}^{(S)}_{\widehat{j}_n,N-\widehat{j}_n,n}(I,\nu) 2^{2\widehat{j}_nI}}{ \kappa_{N-\widehat{j}_n, 1}(I)} \Big)^{1/2},
\eta_-
\Big),
\eta_+ \Big).
\end{align}

We use these functions to build iteratively a sequence $(\widehat{H}^{(m)}_n, \widehat{\eta}^{(m)}_n)$ of estimators of $(H, \eta)$. Starting with $(\widehat{H}^{(0)}_n, \widehat{\eta}^{(0)}_n)$, we define for an integer $m \geq 1$:
\begin{align*}
\widehat{H}^{(m)}_n
&=
\widehat{H}_n^{c}(\widehat{H}^{(m-1)}_n, \widehat{\eta}^{(m-1)}_n),
\end{align*}
and
\begin{align*}
\widehat{\eta}^{(m)}_n
&=
\widehat{\eta}^{c}_n(\widehat{H}^{(m)}_n, \widehat{\eta}^{(m-1)}_n).
\end{align*}

Finally, pick an integer $m_{opt}$ such that  $m_{opt} > 1/(4H) - 2H - 1$ holds for any $H_- < H < H_+$. We obtain the following upper bound:
\begin{thm}
\label{thm:construction}
The rates $v_n(H) = n^{-1/(4H+2)}$ and $w_n(H) = n^{-1/(4H+2)} \log n$ are achievable for estimating $H$ and $\eta$ respectively over the parameter set $\mathcal{D}$.\\

More precisely, if $m_{opt}$ satisfies $m_{opt} > m > 1/(4H) - 2H - 1$ for any $H_- < H < H_+$,
then the sequences of random variables
$$\big(v_n(H)^{-1}(\widehat{H}^{(m_{opt})}_n - H)\big)_{n \geq 1}$$ 
and 
$$\big(w_n(H)^{-1} (\widehat{\eta}^{(m_{opt})}_n - \eta) \big)_{n \geq 1}$$ 
are bounded in $\mathbb P_{H,\eta}$ probability, uniformly over $\mathcal{D}$.
\end{thm}

The choice of $m_{opt}$ is discussed in Section \ref{sec:discussion}, while the proof is delayed until Section \ref{sec:estimator:proof}. Although the results of Theorem \ref{thm:construction} depends on $H_-$ and $H_+$, their value are not that important in practice.

\section{Discussion}
\label{sec:discussion}
We briefly elaborate on the results stated in Section \ref{sec:piecewise} and Section \ref{sec:general}.
\begin{itemize}
\item{\bf About the rates of convergence.} We obtain the same minimax rate $n^{-1/(4H+2)}$ for estimating $H$ as in \cite{gloter2007estimation} and \cite{rosenbaum2008estimation} except that our result is valid over the whole range $(0,1)$ (uniformly over compact sets). This may come as a slight surprise, but retrieving the Hurst exponent $H$ becomes easier as the trajectory becomes rougher. Heuristically,  and in analogy with more classical signal plus noise models, this can be explained by the fact that the more a signal oscillates, the less the noise affects the reconstruction of  its smoothness. Therefore, estimators of the roughness of   volatility should be quite accurate in rough volatility models. This can seem counter-intuitive at first glance as one knows that the optimal rate for estimating a $\beta$-H\"older continuous function (say for instance in the context of estimating a density from a sample of $n$ independent random variables) is  $n^{-\beta/(2\beta+1)}$ for most loss functions. In our setting, we can obtain significantly better rates because we do not try to reconstruct the signal itself but only its regularity. Although volatility remains hidden behind the multiplicative noise of the realised variance, we can retrieve its roughness with a fast convergence rate, close to the usual rate $n^{-1/2}$ of regular parametric statistical models in the limit $H \rightarrow 0$.\\

A referee points out the following observation: in their Theorem 5.7, \cite{bibinger2017nonparametric} establish an optimal rate of a change-point test for a regularity parameter of latent volatility. While this rate shows the same effect of being faster for rough volatility than for smooth volatility, from the non-identifiability result in Remark 5.8 of the mentioned reference, it seems that the rate for estimating the regularity of a general H\"older continuous volatility process could be worse than the lower bound we establish. This would imply that having a fractional volatility process (instead of  a general $H$-Hölder regular process)  is  crucial for the results obtained in this paper.\\ 

\item{\bf About the use of wavelets.} Our estimation strategy relies on wavelets and quadratic functionals of the underlying volatility, as in \cite{gloter2007estimation} and \cite{rosenbaum2008estimation}. The multiresolution nature of wavelets is particularly convenient in our setting from a technical viewpoint, notably when computing the dependence structure of the coefficients. Also, selecting optimal resolution levels can be done in a natural way in this framework. That is why we use this technique instead of $p$-variations of increments. However, both approaches are close in spirit given the strong links between the two objects {\it via} Besov spaces, see for instance \cite{rosenbaum2009besov, ciesielski1993quelques}.\\

\item {\bf Comparison with \cite{chong2022statistical}.} The estimators of $H$ in \cite{chong2022statistical} rely on the autocovariance function of spot volatility estimators. Those constructed and analyzed in this paper are based on scaling properties of quadratic variations for different sampling frequencies. Therefore, they can be considered, respectively, as variants of the autocovariance estimators and change-of-frequency estimators used to estimate the roughness of an observable process (see e.g., \cite{BNCP13}). Both approaches are methodologically similar, as they depend on ratio statistics involving linear combinations of autocovariances (of the spot volatility estimator). This being said, there are fine differences in how biases are removed in \cite{chong2022statistical} and the current paper. In both papers, an increasing number of debiasing terms are needed as $H\downarrow 0$. However, because spot variance is a linear function of a fractional process in \cite{chong2022statistical}, while this relationship is nonlinear (of exponential type) in the current setting, the debiasing strategies are different between the two papers.

\item{\bf About second-order increments.} The fact that second-order increments are needed to optimally estimate fractional signals when $H\geq 3/4$ is well known, see for example \cite{istas1997quadratic, coeurjolly2001estimating}. That is why we consider such increments in Section \ref{sec:piecewise}. For technical convenience, we restrict ourselves in Section \ref{sec:general} to the case $H<3/4$. This enables us to avoid additional issues coming from the asymptotic expansion of
\begin{align*}
\log \Big( \tfrac{1}{\delta} \int_{i\delta}^{(i+1)\delta} \exp(\eta W^H_s) ds \Big)
-
\log \Big( \tfrac{1}{\delta} \int_{(i-1)\delta}^{i\delta} \exp(\eta W^H_s) ds \Big)
\end{align*}
as developed in Proposition \ref{prop:development_iv}.\\

\item{\bf About the parameter \texorpdfstring{$\eta$}{eta}.} Although our model has two parameters, $\eta$ and $H$, the parameter of interest is obviously $H$. That is why, in the piecewise constant volatility model of Section \ref{sec:piecewise}, we do not provide an estimator for $\eta$. However, similarly to the results in Section \ref{sec:general}, one could show that the estimator 
\begin{align*}
\widehat{\eta}_n = 
\bigg(\frac{\widehat{Q}_{\widehat{j}_n,N-\widehat{j}_n,n} 2^{2\widehat{j}_n\widehat{H}_n}}{ \kappa_{N-\widehat{j}_n}(\widehat{H}_n)} \bigg)^{1/2}
 \;\text{ with }\;
\widehat{j}_n
= \bigg\lfloor \frac{1}{2\widehat{H}_n+1}\log_2 n \bigg\rfloor,
\end{align*}
where $\kappa_p$ is an explicit function defined within Lemma \ref{lem:kappa}, is consistent and achieves the  convergence rate $n^{-1/(4H+2)} \log n$. A minor modification of the proof of Theorem \ref{thm:lower_multiscale} shows that this rate is also minimax optimal.\\

\item{\bf Implementation and feasibility.} 
In the piecewise constant volatility model, the estimator is easy to implement and fast to compute, the only tuning parameter being $\nu_0$. From Theorem \ref{thm:upper_multiscale}, a suitable choice would be $\nu_0 = \tfrac{1}{2} \eta_-^2 \min(3, (4-2^{2H_+}) 2^{2H_+})$ where $H_+ = \sup_{(H, \eta) \in \mathcal{D}} H$ and $\eta_- = \inf_{(H, \eta) \in \mathcal{D}} \eta$.\\

The estimators in the general model are more intricate. This is first due to the presence of $\kappa_{p,a}$ in the debiasing procedure. For instance, the evaluation of the function $\kappa_{p,a}$ (see \eqref{eq:def_kappas}) involves the computation of $O(2^p)$ expectations of $2a$ correlated Gaussian variables. Explicit formulas for such products are given by Isserlis' theorem (see Theorem \ref{thm:gaussian_moments}) but they come at a slight computational cost. The second issue is the stopping rule of the iterated debiasing procedure. The quantity $m_{opt}$ must satisfy $m_{opt} > m > 1/(4H) - 2H - 1$ for any $H_- < H < H_+$. Since it needs to be an integer, a quick study of the function $x \mapsto  1/(4x) - 2x - 1$ ensures that one can always take $m_{opt} = \max ( \lfloor  1/(4H_-) - 2H_- \rfloor, \; 0 )$. This choice however has a strong impact on the computation time when $H$ is small. However, in most cases of interest, this iteration cost remains very reasonable: we need $5$ iterations for $H_- = 0.05$ and $24$ iterations for $H_- = 0.01$ for instance.\\

\item{\bf Model choice.} In this paper we consider the prototypical rough volatility model \eqref{roughvol1}. This is  certainly  a reasonable choice when studying fundamental inference questions such as minimax optimality, as it enables us to understand the core statistical structure of rough volatility. However, although rough volatility models were initially presented under such a simple form (see \cite{gatheral2018volatility}), they have been extended  since then, and many models involving various transforms of  fractional Brownian motion or related rough Gaussian processes have emerged, mostly driven by practical considerations. Taking this into account, we build in the companion paper \cite{chong2022statistical} non-parametric estimators of the roughness of   volatility  for these extensions of Model \eqref{roughvol1}, together with a full central limit theory that allows for constructing confidence intervals.

\end{itemize}

\section{Proof of Theorem \ref{thm:lower_multiscale}}
\label{sec:proof:lower_multiscale}

\subsection{Outline of the proof}
Suppose first that we directly observe $\sigma_{i\delta}$ for $0 \leq i \leq \delta^{-1}$. This would be equivalent to the observation of $(\eta W^H_{i\delta})_{i \leq \delta^{-1}}$. Optimal estimation in this model was first  studied in \cite{kawai2013fisher}, where the singular LAN property for fractional Brownian motion observed at high frequency is established, see also \cite{brouste2018lan} for minimax lower bounds. Optimal rates for the estimation of $H$ and $\eta$ are $\delta^{1/2}$ and $\delta^{1/2}\log n$ respectively. Since the model presented in Section \ref{sec:multiscale_model} carries less statistical information than the direct observation of $(\sigma_{i\delta})_{i \leq \delta^{-1}}$, the rate $\delta^{1/2}$ is a lower bound for estimating $H$ and $\delta^{1/2}\log n $ is a lower bound for estimating $\eta$ in the model of Section \ref{sec:multiscale_model}. This proves Theorem \ref{thm:lower_multiscale} whenever $\delta \geq n^{-1/(2H+1)}$.\\

For the remaining of this Section, we will suppose that $\delta \leq n^{-1/(2H+1)}$ so that $v_n(H) = n^{-1/(4H+2)}$. We aim at proving Theorem \ref{thm:lower_multiscale} by a similar strategy as Theorems 2 and 3 of \cite{gloter2007estimation}. However, two major differences need to be  taken into account here. First, we must include the case $H < 1/2$, and this is achieved by relying on  \cite{szymanski2022optimal}. Second, we need to show how our model, somewhat different from the additive noise model of \cite{gloter2007estimation} and \cite{szymanski2022optimal}, can actually fit in their setting. For completeness, we will go through the main ideas of the proof and emphasise the major  changes that need to be undertaken.\\

We need some notation. Write $\norm{\mu }_{TV} = \sup_{\norm{f}_{\infty} \leq 1} | \int f d\mu |$ for the the total variation of a signed measure $\mu$ and $\|f\|_\infty = \sup_{t \in [0,1]}|f(t)|$ for a real-valued  continuous function.  For two probability measures $\mu$ and $\nu$, set
\begin{align*}
d_{\mathrm{test}}(\mu, \nu) = \sup_{0 \leq f \leq 1} \Big | \int f d(\mu -\nu) \Big |
\end{align*}
so that $d_{test}(\mu, \nu) = \tfrac{1}{2} \norm{\mu-\nu}_{TV}$.  We denote by $\mathbb{P}_f^n$ the law of the observations $(S_{i/n})_i$ given $\eta W^H_t = f(t)$.\\ 

The following two results are key to the proof of the lower bounds. They respectively extend  Propositions 4 and 5 of \cite{gloter2007estimation} in our more general setting. First, we show that the law of the observations is  close whenever the underlying volatilities are close. 

\begin{prop}
\label{prop:bound_TV}
Let $f$ and $g$ be two bounded functions. Then there exists $c_0>0$ such that
\begin{align}
\label{eq:bound_TV:est1}
\norm{\mathbb{P}_f^n - \mathbb{P}_g^n}_{TV} \leq c_0 \sqrt{n} \norm{f-g}_\infty.
\end{align}
Moreover, there exist $c_1 > 0$ and a universal nonincreasing positive function $R$ such that
\begin{align}
\label{eq:bound_TV:est2}
1 - \tfrac{1}{2} \norm{\mathbb{P}_f^n - \mathbb{P}_g^n}_{TV} \geq R(n e^{\norm{f-g}_\infty} \norm{f-g}_\infty^2).
\end{align}
\end{prop}

Consider now $(H_0, \nu_0)$ in the interior of the domain $\mathcal{D}$. We pick $I > 0$ large enough and we set
\begin{align*}
H_1 = H_0 + \varepsilon_n \;\;\;\;\text{and}\;\;\;\;\sigma_1=\sigma_02^{j_0\varepsilon_n}
\end{align*}
where
\begin{align*}
\varepsilon_n = I^{-1}n^{-1/(4H_0+2)} \;\;\;\;\text{and}\;\;\;\;j_0=\lfloor\log_2(n^{1/(2H_0+1)})\rfloor.
\end{align*}

The next proposition shows that we can build two processes $\xi^{0,n}$ and $\xi^{1,n}$ that act as approximations of $\eta_0 W^{H_0}$ and $\eta_1 W^{H_1}$.

\begin{prop}
\label{prop:optimal:key}
For $I$ large enough, there exists a sequence of probability spaces $( X^{n}, \mathcal{X}_n, \mathbf{P}^{n} )$ on which we can be define two sequences of stochastic processes, $( \xi^{0,n}_t )_{t\in [0, 1]}$ and $( \xi^{1,n}_t )_{t\in [0, 1]}$ and a measurable transformation $T^n: X^{n} \to X^{n}$ such that the following hold:
\begin{enumerate}
  \setlength{\itemsep}{0pt}
  \setlength{\parskip}{0pt}

\item \label{prop:optimal:key:enum:2} If $P^{i,n}( \cdot ) = \int_{X^{n}}
\mathbb{P}_{\xi^{i,n}(\omega)}^{n}(\cdot) \mathbf{P}^{n} ( d\omega) $, then $\norm{P^{i,n} - \mathbb{P}_{H,\sigma}^{n}}_{TV} \to 0$ for both $i=0,1$, where $\mathbb{P}_{H,\sigma}^{n}$ is the law of the observations $(S_{i/n})_i$ under $\mathbb{P}_{H,\sigma}$.

\item \label{prop:optimal:key:enum:3} The sequence $
n \norm{
\xi^{1,n}(\omega)
-
\xi^{0,n}( T^n\omega )
}_\infty^2
$
is tight under $\mathbf{P}^{n}$.

\item \label{prop:optimal:key:enum:4}If $n$ is large enough, the probability measure $\mathbf{P}^{n}$ and its image $T^n \mathbf{P}^{n}$ are equivalent on $( X^n, \mathcal{X}_n)$, and there exists $0 < c^* < 2$ such that
$\norm{
\mathbf{P}^{n}
-
T^n \mathbf{P}^{n}
}_{TV} \leq 2 - c^* < 2
$ for $n$ large enough.

\end{enumerate}
\end{prop}

This proposition replaces Proposition 5 of \cite{gloter2007estimation}. Part  \ref{prop:optimal:key:enum:2} shows that we can asymptotically replace the fractional Brownian motions $\eta_i W^{H_i}$ by the processes $\xi^{i,n}(\omega)$ in the model presented in Section \ref{sec:multiscale_model}. Moreover, the processes $\xi^{i,n}(\omega)$ are defined in such a way that we can pathwise transform one process into the other in the probability space $( X^{n}, \mathcal{X}_n, \mathbf{P}^{n} )$. This property is essential to proving the lower bound and explicitly shows how one statistical experiment can be transformed into the other. This is the main goal of points \ref{prop:optimal:key:enum:3} and \ref{prop:optimal:key:enum:4}. For sake of completeness, we will cover the main idea of the proof of Proposition \ref{prop:optimal:key} in Section \ref{sec:proof:optimal:key}.
\\

We can now complete the proof of Theorem \ref{thm:lower_multiscale}. We again follow \cite{gloter2007estimation}. The same procedure applies for $H$ and $\eta$, hence we focus on the efficient rate for $H$. We start with an arbitrary estimator $\widehat{H}_n$ of $H$  and we choose $I > 0$ large enough so that Proposition \ref{prop:optimal:key} holds. Let $M < 1/(2I)$. Then we have, using Proposition \ref{prop:optimal:key} and the notation therein,
\begin{align*}
\sup_{(H,\eta)} &
\mathbb{P}_{H,\eta}^{n} (v_n(H)^{-1}|\widehat{H}_n - H| \geq M)
\\
&\geq \tfrac{1}{2} \big(
\mathbb{P}_{H_0,\eta_0}^{n} (v_n(H_0)^{-1}|\widehat{H}_n - H_0| \geq M)
+
\mathbb{P}_{H_1,\eta_1}^{n} (v_n(H_1)^{-1}|\widehat{H}_n - H_1| \geq M)
\big)
\\
&\geq \tfrac{1}{2} \big(
P^{0,n} (v_n(H_0)^{-1} |\widehat{H}_n - H_0| \geq M)
+
P^{1,n} (v_n(H_1)^{-1} |\widehat{H}_n - H_1| \geq M)
\big) + o(1)
\\
&=
\tfrac{1}{2}
\int_{X^{n}}
\Big(
\mathbb{P}_{\xi^{0,n}(\omega)}^{n}(A^0) 
+
\mathbb{P}_{\xi^{1,n}(\omega)}^{n}(A^1) 
\Big)
\; \mathbf{P}^{n} ( d\omega) + o(1),
\end{align*}
where $A^{i} = \{v_n(H_i)^{-1} |\widehat{H}_n - H_i| \geq M \}$. Taking $n$ large enough, it suffices to bound from below the integral appearing here. But since $\mathbf{P}^n$ and $T^n \mathbf{P}^n$ are equivalent (see \ref{prop:optimal:key:enum:4} of Proposition \ref{prop:optimal:key}), we have
\begin{align*}
\int_{X^{n}}
\mathbb{P}_{\xi^{0,n}(\omega)}^{n}(A^0) 
\;\mathbf{P}^{n} ( d\omega)
=
\int_{X^{n}}
\mathbb{P}_{\xi^{0,n}(T^n \omega)}^{n}(A^0) 
\frac{d\mathbf{P}^{n}}{dT^n \mathbf{P}^{n}}(T^n \omega) 
\; \mathbf{P}^{n} ( d\omega).
\end{align*}

For $r > 0$, we denote by $X_r^{n}$  the set of $\omega \in X^{n}$ such that 
\begin{align*}
n \norm{\xi^{0,n}(T^n\omega)-\xi^{1,n}(\omega)}_\infty^2 \leq r^2.
\end{align*}
Notice that this definition is slightly different from that of \cite{gloter2007estimation}. This is due to the differences in Proposition \ref{prop:bound_TV}. For  arbitrary  $\lambda > 0$, we obtain
\begin{align*}
\int_{X^{n}}
&\Big(\mathbb{P}_{\xi^{0,n}(\omega)}^{n}(A^0) 
+
\mathbb{P}_{\xi^{1,n}(\omega)}^{n}(A^1) 
\Big)\; \mathbf{P}^{n} ( d\omega)
\\
&=
\int_{X^{n}}
\big(
\mathbb{P}_{\xi^{0,n}(T^n \omega)}^{n}(A^0) 
\frac{d\mathbf{P}^{n}}{dT^n \mathbf{P}^{n}}(T^n \omega) + \mathbb{P}_{\xi^{1,n}(\omega)}^{n}(A^1) 
\big)
\; \mathbf{P}^{n} ( d\omega)
\\
&\geq
e^{-\lambda}\int_{X_r^{n}}
\big(
\mathbb{P}_{\xi^{0,n}(T^n \omega)}^{n}(A^0) 
 + \mathbb{P}_{\xi^{1,n}(\omega)}^{n}(A^1) 
\big)
\1_{\frac{d\mathbf{P}^{n}}{dT^n \mathbf{P}^{n}}(T^n \omega) \geq e^{-\lambda}}
\; \mathbf{P}^{n} ( d\omega).
\end{align*}
Note also that for $\omega \in X^n_r$, we have
\begin{align*}
\mathbb{P}_{\xi^{0,n}(T^n\omega)}^{n}&(A^{0})
	+
\mathbb{P}_{\xi^{1,n}(\omega)}^{n}(A^{1})
\\
&\geq
\mathbb{P}_{\xi^{0,n}(T^n\omega)}^{n}(n^{1/(4H_0+2)} |\widehat{H}_n - H_0| \geq M)
	+
\mathbb{P}_{\xi^{1,n}(\omega)}^{n}(n^{1/(4H_0+2)} |\widehat{H}_n - H_1| \geq M)
\end{align*}
since $H_1 \geq H_0$ and therefore,
\begin{align*}
\mathbb{P}_{\xi^{0,n}(T^n\omega)}^{n}&(A^{0})
	+
\mathbb{P}_{\xi^{1,n}(\omega)}^{n}(A^{1})
\\
&\geq 
\mathbb{P}_{\xi^{1,n}(\omega)}^{n}(n^{1/(4H_0+2)} |\widehat{H}_n - H_0| \geq M)
	+
\mathbb{P}_{\xi^{1,n}(\omega)}^{n}(n^{1/(4H_0+2)} |\widehat{H}_n - H_1| \geq M)
\\&\;\;\;\;-
d_{test} \big(\mathbb{P}_{\xi^{0,n}(T^n\omega)}^{n}, \mathbb{P}_{\xi^{1,n}(\omega)}^{n}\big)
\\
&\geq 1 - 
\tfrac{1}{2} \norm{
\mathbb{P}_{\xi^{0,n}(T^n\omega)}^{n} - \mathbb{P}_{\xi^{1,n}(\omega)}^{n}
}_{TV}
\end{align*}
since $n^{1/(4H_0+2)} |H_0 - H_1| \geq 2M$ by definition of $I$ and $M$. We now apply Proposition \ref{prop:bound_TV} to get
\begin{align*}
\mathbb{P}_{\xi^{0,n}(T^n\omega)}^{n}(A^{0})
	+
\mathbb{P}_{\xi^{1,n}(\omega)}^{n}(A^{0})
&
\geq 
R(n e^{\norm{\xi^{0,n}(T^n\omega)-\xi^{1,n}(\omega)}_\infty}\norm{\xi^{0,n}(T^n\omega)-\xi^{1,n}(\omega)}_\infty^2)
\\
&\geq 
R(e^{r/\sqrt{n}}r^2)
\geq 
R(e^{r}r^2)
\end{align*}
since $R$ is non-increasing. We infer
\begin{align*}
\int_{X^{n}}
&
\Big(
\mathbb{P}_{\xi^{0,n}(\omega)}^{n}(A^0) 
+
\mathbb{P}_{\xi^{1,n}(\omega)}^{n}(A^1) 
\Big)
\; \mathbf{P}^{n} ( d\omega)
\\
&\geq
e^{-\lambda} R(e^{r}r^2) \int_{X_r^{n}}
\1_{\frac{d\mathbf{P}^{n}}{dT^n \mathbf{P}^{n}}(T^n \omega) \geq e^{-\lambda}}
\; \mathbf{P}^{n} ( d\omega)
\\
&\geq
e^{-\lambda} R(e^{r}r^2) \mathbf{P}^{n}
\Big (
X_r^{n}
\cap
\Big \{
\frac
	{d\mathbf{P}^{n}( T^n\omega)}
	{dT^{n}\mathbf{P}^{n}}
	\geq e^{-\lambda}
\Big \}
\Big)
\\
&\geq
e^{-\lambda} R(e^{r}r^2) 
\Big (
\mathbf{P}^{n}
(
X_r^{n}
)
-
T^n \mathbf{P}^{n}
\Big (
\frac
	{dT^{n}\mathbf{P}^{n}}
	{d\mathbf{P}^{n}}
	\geq e^{\lambda}
\Big )
\Big).
\end{align*}
Markov's inequality and Proposition \ref{prop:optimal:key} yield
\begin{align*}
T^n \mathbf{P}^{n}
\bigg (
\frac
	{dT^{n}\mathbf{P}^{n}}
	{d\mathbf{P}^{n}}
	\geq e^{\lambda}
\bigg ) 
&
\leq 
 \mathbf{P}^{n}
\bigg (
\frac
	{dT^{n}\mathbf{P}^{n}}
	{d\mathbf{P}^{n}}
	\geq e^{\lambda}
\bigg )
+ d_{test}(T^{n}\mathbf{P}^{n},\mathbf{P}^{n})
\\
&\leq
 e^{-\lambda}
+ \tfrac{1}{2}\norm{T^{n}\mathbf{P}^{n}-\mathbf{P}^{n}}_{TV}
\\
&\leq
 e^{-\lambda}
+ 1 - \frac{c^*}{2},
\end{align*}
hence
\begin{align*}
\int_{X^{n}}
&\mathbb{P}_{\xi^{0,n}(\omega)}^{n}(A^0) 
+
\mathbb{P}_{\xi^{1,n}(\omega)}^{n}(A^1) 
\; \mathbf{P}^{n} ( d\omega)
\geq 
e^{-\lambda}
R(e^{r}r^2) 
\big(
\mathbf{P}^{n}
( 
X_r^{n}
)
-
e^{-\lambda}
- 1
+
\frac{c^*}{2}
\big).
\end{align*}

Moreover, 
\begin{align*}
\lim_{r\to\infty}\liminf_{n\to\infty}\mathbf{P}^{n}
( 
X_r^{n}
)
=
\lim_{r\to\infty}\liminf_{n\to\infty}\mathbf{P}^{n}
( n \norm{\xi^{0,n}(T^n\omega)-\xi^{1,n}(\omega)}_\infty^2 \leq r^2 )
= 1
\end{align*}
since $n \norm{\xi^{0,n}(T^n\omega)-\xi^{1,n}(\omega)}_\infty^2$ is tight by Proposition \ref{prop:optimal:key}. We conclude by taking $\lambda$ and $r$ large enough.

\subsection{Proof of Proposition \ref{prop:bound_TV}}

Let $K( \mu, \nu)
=
\int \log ( \frac{d\mu}{d\nu} ) \; d\mu$ denote the Kullback--Leibler divergence between two probability measures $\mu$ and $\nu$. Recall also  Pinsker's inequality $\norm{\mu - \nu}_{TV}^2 \leq 2K( \mu, \nu)$. Under $\mathbb{P}_f^n$, the increments of the observations $S_{(j+1)/n} - S_{j/n}$ are independent Gaussian variables with variance 
\begin{align*}
n^{-1} \exp(f(\lfloor j n^{-1}\delta^{-1} \rfloor\delta))
\end{align*}
so that 
\begin{align*}
K(\mathbb{P}_f^n, \mathbb{P}_g^n)
=  n \delta \sum_{i=0}^{\delta^{-1}-1} A ( (f-g) (i\delta) )
\end{align*}
for $A(x) = \tfrac{1}{2}(e^x - x - 1)$. The function $A$ is increasing on $[0, \infty)$ and  $A(x) \leq A(|x|)$ for $x\geq 0$. Therefore,
\begin{align}
\label{eq:ineq_K_fg}
K(\mathbb{P}_f^n, \mathbb{P}_g^n) \leq  n A(\norm{f-g}_\infty).
\end{align}
Note in addition that $A(x) \leq x^2$ for $0\leq x \leq 1$, hence $K(\mathbb{P}_f^n, \mathbb{P}_g^n) \leq n \norm{f-g}_\infty^2$ if $\norm{f-g}_\infty \leq 1$. Pinsker's inequality completes the proof of the estimate \eqref{eq:bound_TV:est1} in that case.  When $\norm{f-g}_\infty \geq 1$, we write $K = \lceil \norm{f-g}_\infty \rceil$ and $f_k = (kg+(K-k)f)/K$ for all $k=0, \dots, K$ so that $\norm{f_k-f_{k+1}}_\infty \leq 1$ for any $k$. Using \eqref{eq:bound_TV:est1} for the functions $(f_k, f_{k+1})$, we get
\begin{align*}
\norm{\mathbb{P}_f^n - \mathbb{P}_g^n}_{TV}
\leq \sum_k \norm{\mathbb{P}_{f_k}^n - \mathbb{P}_{f_{k+1}}^n}_{TV}
\leq \sum_k C n^{1/2} \norm{f_k-f_{k+1}}_\infty \leq 
C n^{1/2} \norm{f-g}_\infty
\end{align*} 
and \eqref{eq:bound_TV:est1} follows. We now turn to \eqref{eq:bound_TV:est2}. First, notice that for $x \geq 0$, we also have $A(x) \leq x^2 e^x$. Thus,  \eqref{eq:ineq_K_fg} yields
\begin{align*}
K(\mathbb{P}_f^n, \mathbb{P}_g^n)
\leq
n e^{\norm{f-g}_\infty} \norm{f-g}_\infty^2,
\end{align*}
and we can conclude since $\norm{\mu - \nu}_{TV}$ remains bounded away from $2$ when $K(\mu, \nu)$ and $K(\nu, \mu)$ are bounded away from $\infty$.

\subsection{Proof of Proposition \ref{prop:optimal:key}}
\label{sec:proof:optimal:key}

\subsubsection*{Introduction and notation}

Proposition \ref{prop:optimal:key} is a modification of Proposition 5 in \cite{gloter2007estimation}. Indeed, points \ref{prop:optimal:key:enum:2} and \ref{prop:optimal:key:enum:4} are the same as in \cite{gloter2007estimation} and the construction of the approximation processes $\xi^{i,n}$ in our proof is exactly the same as in \cite{gloter2007estimation}. Therefore, we will quickly recall the main arguments concerning the construction of the $\xi^{i,n}$ while skipping the tedious computations already done therein. We will also skip the proof of points \ref{prop:optimal:key:enum:2} and \ref{prop:optimal:key:enum:4} and focus instead on \ref{prop:optimal:key:enum:3}.\\

The main idea of the approximation is to decompose the fractional Brownian motion $\eta W^H$ over a wavelet basis and to keep only low frequencies. Indeed, high frequencies encode local information of the fractional Brownian motion, such as its Hölder regularity. Thus it should vary a lot even for small changes of $H$ even if the absolute value of a single coefficient is generally small. On the other side, low-frequency coefficients should be high to determine the global trends of the fractional Brownian motion, but their behaviour should be continuous in $H$. Thus, by only keeping low-frequency components, we will get a process close enough to the original fractional Brownian motion. Moreover, a slight technical modification in the low-frequency process approximating $\eta_1 W^{H_1}$ ensures this remains close to $\eta_1 W^{H_1}$ while being closer to the low-frequency process approximating $\eta_0 W^{H_0}$. The main parameter remaining is the cut-off level in the frequencies and, as one may expect, we will see that $\log_2 n^{1/(1+2H_0)}$ is  a suitable choice.\\

We now elaborate on the approximation procedure. First, recall from \cite{gloter2007estimation} that for any $H$, one may realise a fractional Brownian motion with Hurst index $H$ via the random series representation
\begin{align}
\label{eq:repr:fbm}
2^{-j_0 (H+1/2)}
\sum_{k=-\infty}^{\infty}
\Theta^H_{j_0, k}(t)\varepsilon^H_k
+
\sum_{j=j_0}^\infty 2^{-j(H+1/2)} \sum_{|k|\leq 2^{j_0+1}}  (\psi^H_{j,k}(t) - \psi^H_{j,k}(0)) \varepsilon_{j,k} 
\end{align}
where $(\varepsilon_{j,k})_{j,k}$ are independent standard Gaussian variables, $(\varepsilon^H_k)_k$ is a Gaussian stationary family with spectral density given by $v \mapsto |2\sin(v/2)|^{1-2H}$ independent of $(\varepsilon_{j,k})_{j,k}$ and where the $\Theta^H_{j,k}$ and $\psi^H_{j,k}$ are defined as fractional derivatives of wavelet functions. For brevity, we do not give their construction, which can be found in \cite{gloter2007estimation}, Section 7.1. All we  need is the following property.

\begin{lem}[Lemma 5 (i) in \cite{gloter2007estimation}]
\label{lem:lemma5gloter}
For any $M>0$ there exists $c=c(M)$ such that for all $j_0$ and all $H \in [H_-, H_+]$
\begin{align*}
\sum_{j\geq j_0}
\sum_{|k|\leq 2^{j+1}}
\norm{ \psi_{j,k}^{H}}_{\infty}  (1+j)^{1/2} (1+|k|)^{1/2}
\leq c(M)2^{-Mj_0}.
\end{align*} 
\end{lem}

An essential feature of representation \eqref{eq:repr:fbm} is that the random variables appearing in the high-frequency part $\sum_{j=j_0}^\infty 2^{-j(H+1/2)} \sum_{|k|\leq 2^{j_0+1}}  (\psi^H_{j,k}(t) - \psi^H_{j,k}(0)) \varepsilon_{j,k} 
$ are independent of the low-frequency terms $2^{-j_0 (H+1/2)}
\sum_{k=-\infty}^{\infty}
\Theta^H_{j_0, k}(t)\varepsilon^H_k
$.
\subsubsection*{Construction of the space \texorpdfstring{$X^n$}{Xn}}

As in \cite{gloter2007estimation}, we take
\begin{align*}
X^n = 
\big( \otimes_{k=-2^{j_0+1}}^{2^{j_0+1}} 
\mathbb{R}
\big)
\otimes
\big( \otimes_{j=j_0}^\infty \otimes_{k=-2^{j+1}}^{2^{j+1}} 
\mathbb{R}
\big)
=:
X^n_e \otimes X^n_d.
\end{align*}
We write $\mathcal{X}^n$ for the Borel product sigma-algebra of $X^n$, $\omega = (\omega^e, \omega^d)$ for the elements of $X^n$, and 
\begin{align*}
\varepsilon_k(\omega) = \omega^e_k \;\;\;\;\text{and}\;\;\;\; \varepsilon_{j,k}(\omega) = \omega^d_{j,k}
\end{align*}
for the projections on the coordinates of $\omega$.  In view of \eqref{eq:repr:fbm}, we then define
\begin{align}
\label{eq:def:xi0}
\xi_t^{0,n}
=
\eta_0
2^{-j_0 (H_0+1/2)}
\sum_{|k|\leq 2^{j_0+1}}
\Theta^{H_0}_{j_0, k}(t) \varepsilon_k
+
\eta_0
\sum_{j=j_0}^\infty 2^{-j(H_0+1/2)} \sum_{|k|\leq 2^{j+1}} (\psi^{H_0}_{j,k}(t) - \psi^{H_0}_{j,k}(0)) \varepsilon_{j,k}.
\end{align}
To ensure this is a correct approximation of $\sigma_0 W^{H_0}$, we define the probability measure 
$$\mathbf{P}^n := \mathbf{P}^n_e \otimes \mathbf{P}^n_d$$ such that under $\mathbf{P}^n_e$, $(\varepsilon_k)_k$ is a centered Gaussian stationary sequence with spectral density
$$|2\sin(v/2)|^{1-2H_0}$$
 and under $\mathbf{P}^n_d$, $(\varepsilon_{j,k})_{j,k}$ are independent standard Gaussian variables. Therefore, under $\mathbf{P}^n$, the law of $\xi^{0,n}$ is close to the law of $\eta_0 W^{H_0}$. Now we want to define an approximation of $\eta_1 W^{H_1}$. As explained in \cite{gloter2007estimation,szymanski2022optimal}, replacing $\varepsilon_k$ by a stationary sequence $(\varepsilon_k')_k$ with spectral density $|2\sin(v/2)|^{1-2H_1}$ is not enough and one should incorporate correction terms corresponding to the expansion of $\Theta^{H_1}=\Theta^{H_0+\varepsilon_n}$. Following \cite{gloter2007estimation}, a suitable approximation for $\eta_1 W^{H_1}$ is given by
\begin{align*}
\xi_t^{1,n} & =
\eta_1
2^{-j_0 (H_1+1/2)}
\sum_{|k|\leq 2^{j_0+1}}
\Theta^{H_0}_{j_0, k}(t) \varepsilon_k' \\
&+
\eta_1
\sum_{j=j_0}^\infty \Big( 2^{-j(H_1+1/2)} \sum_{|k|\leq 2^{j+1}} \big(\psi^{H_1}_{j,k}(t) - \psi^{H_1}_{j,k}(0)\big) \varepsilon_{j,k} \Big)
\\
&
+ \eta_1 2^{j_0(H_1+1/2)} \sum_{|l|\leq2^{j_0+1}}  \sum_{|k|\leq 2^{j_0+1}} 
\Big(
\Theta_{j_0,l}^{H_0}(t) a_{l-k} \varepsilon_k'
+ (\psi_{j_0,l}^{H_0}(t) - \psi_{j_0,l}^{H_0}(0) ) b_{l-k}\varepsilon_k'\Big),
\end{align*} 
where the sequences $(a_k)_k$ and $(b_k)_k$ are defined in Lemma 5 of \cite{gloter2007estimation}.\\

\cite{gloter2007estimation} also provides the mapping $T^n$ of Proposition \ref{prop:optimal:key}. This transformation is divided into two parts. The first one acts on $X^n_d$ and transforms the stationary sequence $(\varepsilon_k)_k$ with spectral density $|2\sin(v/2)|^{1-2H_0}$ into a stationary sequence $(\varepsilon_k')_k$ with spectral density $|2\sin(v/2)|^{1-2H_1}$, which can be done since these measures are mutually absolutely continuous. The second one is slightly more involved and uses the sequences $(a_k)_k$ and $(b_k)_k$ to linearly transform the $(\varepsilon_{j,k})_{j,k}$. We refer once more to \cite{gloter2007estimation} for the details. The only important result regarding this construction for the proof of point \ref{prop:optimal:key:enum:3} of Proposition \ref{prop:optimal:key} is the following expansion:
\begin{align}
\label{eq:diff:xi}
\begin{aligned}
\xi_t^{1,n}(\omega)
-
\xi_t^{0,n}(T^n\omega)
&=
\sum_{j\geq j_0}
\sum_{|k| \leq 2^{j+1}}
\Big[
\eta_1 2^{j(H_1+1/2)}
(\psi^{H_1}_{j,k}(t) - \psi^{H_1}_{j,k}(0)) \varepsilon_{j,k}(\omega)
\Big]
\\&\;\;\;\;+
\sum_{j\geq j_0}
\sum_{|k| \leq 2^{j+1}}
\Big[
\eta_0 2^{j(H_0+1/2)}
(\psi^{H_0}_{j,k}(t) - \psi^{H_0}_{j,k}(0)) \varepsilon_{j,k}(\omega)\Big].
\end{aligned}
\end{align}

\subsubsection*{Proof of Proposition \ref{prop:optimal:key}.\ref{prop:optimal:key:enum:3}}

Notice first that it is enough to prove that 
$$\sqrt{n} \norm{
\xi^{1,n}(\omega)
-
\xi^{0,n}( T^n\omega )
}_{\infty}$$ is bounded in $L^1$ since Markov's inequality for positive random variables will then ensure tightness. By \eqref{eq:diff:xi}, it is enough to bound
$$
\sum_{j\geq j_0}
\sum_{|k|\leq 2^{j+1}}
\Big(
\eta_1 2^{-j( H_1 + 1/2 )}
\norm{ \psi_{j,k}^{H_1}}_{\infty} \left | \varepsilon_{j,k}(\omega) \right |\Big)
$$ 
and
$$\sum_{j\geq j_0}
\sum_{|k|\leq 2^{j+1}}
\Big(
\eta_0 2^{-j( H_0+ 1/2 )}
\norm{ \psi_{j,k}^{H_0} }_{\infty} \left |\varepsilon_{j,k}(\omega)\right |
\Big).
$$
Both of these terms are handled similarly so we will focus on the first one here. We now need to use Lemma 3 of \cite{meyer1999wavelets}.
\begin{lem}
Let $(\varepsilon_{j,k})_{j\geq 0, k \in \mathbb{Z}}$ be independent standard Gaussian variables. Then there exists a random variable $C(\omega)$ having finite moments of all orders such that
\begin{align*}
|\varepsilon_{j,k}(\omega)| \leq C(\omega) \big( {\log(2+j) \log(2+|k|)}\big)^{1/2}.
\end{align*}
\end{lem}
In particular, there exists a positive random variable $C = C(\omega)$ with finite moments for any order  such that  for any $j \geq 0$ and $|k| \leq 2^{j+1}$, we have
\begin{align*}
\left |\varepsilon_{j,k}(\omega)\right | 
\leq 
C(\omega) (1+j)^{1/2} (1+|k|)^{1/2}
\end{align*}
and therefore
\begin{align*}
\sum_{j\geq j_0}
&\sum_{|k|\leq 2^{j+1}}
\Big(
\eta_1 2^{-j( H_1 + 1/2 )}
\norm{ \psi_{j,k}^{H_1}}_{\infty} \left | \varepsilon_{j,k}(\omega) \right |
\Big)
\\
&\leq
C(\omega) \eta_1 2^{-j_0( H_1 + 1/2 )}
\sum_{j\geq j_0}
\sum_{|k|\leq 2^{j+1}}
\Big(
\norm{ \psi_{j,k}^{H_1}}_{\infty}  (1+j)^{1/2} (1+|k|)^{1/2}
\Big)
\\
&\leq 
C(\omega) c(M) \eta_1 2^{-j_0( H_1 + 1/2 )-Mj_0}
\end{align*}
by Lemma \ref{lem:lemma5gloter}, where $M > 0$ is arbitrary. But $C(\omega)$ is bounded in $L^1$, so for any $M>0$ there exists a constant $c' = c'(M)$ such that
\begin{align*}
\EX{
\norm{
\xi^{1,n}(\omega)
-
\xi^{0,n}( T^n\omega )
}_{\infty}
}
\leq c' \eta_1 2^{-j_0( M + H_1 + \frac{1}{2} )} \leq c''n^{-1/2},
\end{align*}
where the last inequality is obtained taking $M$ large enough.

\section{Proof of Theorem \ref{thm:upper_multiscale}}
\label{sec:proof:upper_multiscale}

\subsection{Preparations for the proof}
We set 
$$H_- = \min H,\; H_+ = \max H,\; \eta_- = \min \eta,\;\eta_+ = \max \eta,$$
where $\max$ and $\min$ are taken over $(H,\eta) \in \mathcal D$.
We start by giving a crucial result on the behaviour of the pre-averaged energy levels and their empirical counterparts. It relies on a similar strategy as in \cite{gloter2007estimation} and is based on the two following results.

\begin{prop}
\label{propo:energy_levels}
Let $\varepsilon > 0$. Then there exist $0 < r_-(\varepsilon) < \exp(-(2H_++1))$, $J_0(\varepsilon) > 0$ and $N_0(\varepsilon)$ depending on $\mathcal D$ only such that for $N \geq N_0(\varepsilon)$, we have
\begin{align}
\sup_{(H, \eta) \in \mathcal D}
\mathbb{P}_{H,\eta}
\big(
\inf_{J_0 \leq j \leq N-1} 2^{2jH} Q_{j,N-j-1} \leq r_-(\varepsilon)
\big)
\leq \varepsilon. \label{eq: energy levels}
\end{align}
\end{prop}

\begin{prop} For any $j$, $p$ and $n$, we have
\label{propo:emp_energy_levels}
\begin{align}
\sup_{(H,\eta) \in \mathcal D}\mathbb{E}_{H, \eta}\big[\big(\widehat{
Q}_{j,p,n} - Q_{j,p}\big)^2
\big]
\leq 
C
\big(
(n \delta)^{-2} 2^{-j-2p}
+ (n \delta)^{-1} 2^{-j(1+2H)-p}\big).
\label{eq: empirical pre-ave behaviour}
\end{align}
\end{prop}

Their proofs are given in Sections \ref{sec:proof:propoenergylevels} and \ref{sec:proof:propoempenergylevels}. Then notice that 
\begin{align*}
\Big | \frac{\widehat{Q}_{J_n^*+1,N-J_n^*-1,n} }{\widehat{Q}_{J_n^*,N-J_n^*-1,n} } - 2^{-2H} \Big |
\leq 
B_n + V_n^{(1)} + V_n^{(2)}
\end{align*}
where
\begin{align*}
B_n &= \Big | \frac{{Q}_{J_n^*+1,N-J_n^*-1} }{{Q}_{J_n^*,N-J_n^*-1} } - 2^{-2H} \Big |,
\\
V_n^{(1)} &= \Big | \frac{\widehat{Q}_{J_n^*+1,N-J_n^*-1,n} - {Q}_{J_n^*+1,N-J_n^*-1} }{\widehat{Q}_{J_n^*,N-J_n^*-1,n} } \Big |,
\\
V_n^{(2)} &= \Big | \frac{{Q}_{J_n^*+1,N-J_n^*-1} 
( \widehat{Q}_{J_n^*,N-J_n^*-1,n}-{Q}_{J_n^*,N-J_n^*-1}  )
}{\widehat{Q}_{J_n^*,N-J_n^*-1,n}{Q}_{J_n^*,N-J_n^*-1} } \Big |.
\end{align*}

Therefore the proof amounts to establishing that the sequences of random variables $(v_n^{-1} B_n)_n$, $(v_n^{-1}V_n^{(1)})_n$ and $(v_n^{-1} V_n^{(2)})_n$ are tight. We separately treat the cases $\delta \leq n^{-1/(2H+1)}$ and  $\delta \geq n^{-1/(2H+1)}$ in Sections \ref{sec:completion:small} and \ref{sec:completion:large} respectively.\\

For $\varepsilon > 0$, introduce 
\begin{align*}
J_n^-(\varepsilon) = \max \big\{ j \geq 1: \; r_-(\varepsilon) 2^{-2jH} \geq 2^j n^{-1} \big\} \wedge (N-1),
\end{align*}
where $r_-(\varepsilon)$ is defined in Proposition \ref{propo:energy_levels}.

\subsection{Proof of Proposition \ref{propo:energy_levels}}
\label{sec:proof:propoenergylevels}

Our proof is an adaptation of Proposition 1 in \cite{gloter2007estimation}. 

\begin{lem} 
\label{lem:kappa}
For any $p$, there exists a function $\kappa_p : (0,1) \to (0, \infty)$ such that for any $(H,\eta) \in  \mathcal{D}$ and any $j\geq 0$, we have
 \begin{align}
\label{eq:def_kappa}
\mathbb{E}_{H, \eta}\big[(d_{j,k,p})^2\big] = \kappa_p(H)  \eta^2 2^{-j(1+2H)}.
\end{align}
Moreover, we have
\begin{align*}
0 < \inf_{(H,\eta) \in \mathcal D, p \geq 0}\kappa_p(H) \leq \sup_{(H,\eta) \in \mathcal D, p \geq 0 }\kappa_p(H) <\infty.
\end{align*}
\end{lem}

\begin{proof}
We start with the explicit computation of $\mathbb{E}_{H, \eta}[(d_{j,k,p})^2]$. By self-similarity and the fact that $(W_t^H)$ has stationary increments, this expectation equals
\begin{align*}
&
\frac{\eta^2}{2^{j(1+2H)+2p}} 
\sum_{\ell_1, \ell_2=0}^{2^p-1}
\mathbb{E}_{H, \eta}\big[
(
W^H_{(\ell_1 - \ell_2)2^{-p}}
-
2W^H_{1 + (\ell_1 - \ell_2)2^{-p}}
+
W^H_{2 +(\ell_1 - \ell_2)2^{-p}}
)
(
W^H_2 - 2 W^H_1
)
\big]
\\
&= 
\frac{\eta^2}{2^{j(1+2H)}} 
2^{-p}
\sum_{\ell \in \mathbb Z, |\ell| < 2^p}
(1-|\ell |2^{-p})
\mathbb{E}_{H, \eta}\big[
(
W^H_{\ell 2^{-p}}
-
2W^H_{1 +\ell 2^{-p}}
+
W^H_{2 +\ell 2^{-p}}
)
(
W^H_2 - 2 W^H_1
)
\big].
\end{align*}
It follows that $\kappa_p(H)$ is well defined for all values of $H$ and $p$. It is also positive since $\mathbb{E}_{H, \eta}[d^2_{j,k,p}]>0$. Moreover, a direct computation of the expectation above yields
\begin{align*}
\kappa_p(H) = 
2^{-p}
\sum_{\ell \in \mathbb Z, |\ell| < 2^p}
(1-|\ell | 2^{-p}) \phi_H(\ell 2^{-p}),
\end{align*}
with $\phi_H(x) = \tfrac{1}{2}\sum_{k=0}^{4} (-1)^{k+1} \binom{4}{k} |x+k-2|^{2H}$. Since $H \mapsto \phi_H(x)$ is continuous for any $x$, we deduce that $H \mapsto \kappa_p(H)$ is bounded below and above for fixed $p$. Moreover, $x \mapsto \phi_H(x)$ is continuous and there exists $C>0$ independent of $H$ such that $|\phi_H(x) - \phi_H(y)| \leq C |x-y|^{2H\wedge 1}$  for any $-1 \leq x, y \leq 1$. Thus $\kappa_p(H)$ converges to $\kappa_\infty(H) = \int_{-1}^1 (1-|x|) \phi_H(x) dx$ as $p \rightarrow \infty$. More precisely, we have
\begin{align*}
| \kappa_p(H) - \kappa_\infty(H)| 
&\leq 
\sum_{\ell \in \mathbb Z, |\ell| < 2^p}
\int_{
(\ell - 1) 2^{-p}
}^
{
\ell 2^{-p}
}
\bigg|
(1-|\ell | 2^{-p}) \phi_H(\ell 2^{-p})
-
(1-|x|) \phi_H(x)
\bigg|
dx
\\
&\;\;\;\;+
\int_{
1- 2^{-p}
}^
{
1
}
(1-|x|) | \phi_H(x) |
dx.
\end{align*}

The integral in the sum is bounded by $C2^{-2p} + C2^{-(2H\wedge 1)p-p}$ and the last integral is bounded by
$C2^{-p}$, so that
\begin{align*}
| \kappa_p(H) - \kappa_\infty(H)| \leq C2^{-p(2H\wedge 1)}
\end{align*}
where the convergence is uniform. Also, notice that
\begin{align*}
\kappa_\infty(H) = \mathbb{E}_{H, \eta}\Big[ \Big(\int_0^1 (W^H_u - 2 W^H_{u+1} + W^H_{u+2})du \Big)^2\Big] > 0,
\end{align*} 
hence Lemma \ref{lem:kappa} follows from the continuity of $H \mapsto \kappa_\infty(H)$.
\end{proof}

\begin{lem}[Decorrelation of the wavelet coefficients] \label{lemma:self_d}
For $ k_1, k_2$ such that $|k_1 - k_2| \geq 3$, we have:
\begin{align*}
|\mathbb{E}_{H,\eta}{[d_{j,k_1,p}d_{j,k_2,p}]}|
\leq 
C
\eta^2 2^{-j(1+2H)}
(1+|k_1-k_2|)^{2H-4}.
\end{align*}
\end{lem}

\begin{proof}
Suppose $k_1 \geq k_2 + 3$. Since $(W^H_t)$ is self-similar and has stationary increments, we have
\begin{align}
\label{eq:explicit_correl_d}
\mathbb{E}_{H,\eta}\big[d_{j,k_1,p}d_{j,k_2,p}\big]
=
\frac{\eta^2}{2^{j(1+2H)}}
2^{-p} \sum_{\ell \in \mathbb{Z}, |\ell |<2^p}
(1-\ell 2^{-p})\phi_{H}(k_1-k_2+ \ell 2^{-p}),
\end{align}
where $\phi_H(x) = \tfrac{1}{2}\sum_{k=0}^{4} (-1)^{k+1} \binom{4}{k} |x+k-2|^{2H}$ is the same as in the proof of Lemma \ref{lem:kappa}. Notice that for $x > 2$, the absolute values appearing in the expression of $\phi_{H}(x)$ can be removed. If $F(x,t) = |x+t|^{2H}$, Taylor's formula yields
\begin{align*}
\phi_H(x) &= \frac{1}{2\cdot 4!} \int_{0}^1 (1-t)^3 \big( 
	16 \partial^{4}_t F(x,-2t)
	-4\partial^{4}_t F(x,-t)
	-4\partial^{4}_t F(x,t)
	+16 \partial^{4}_t F(x,2t)
\big ) dt
\\
&= \frac{1}{2\cdot 4!} \int_{-1}^1 (1-|t|)^3 \big( 
	-4\partial^{4}_t F(x,t)
	+16 \partial^{4}_t F(x,2t)
\big ) dt.
\end{align*}
We infer $|\phi_H(x)| \leq C |x-2|^{2H-4}$ for some constant $C$ independent of $H \in [H_-, H_+]$. Summing over $\ell$ in \eqref{eq:explicit_correl_d} yields the result.
\end{proof}

Using both Lemmas \ref{lem:kappa} and \ref{lemma:self_d}, we have
\begin{align}
\label{eq:deviation_Q}
\mathbb{E}_{H, \eta}\big[\big(Q_{j,p} - \kappa_p(H) \eta^2 2^{-2jH-1}\big)^2\big] \leq C2^{-j(1+4H)}
\end{align}
for some constant $C$ independent of $H$ and $\eta$. The proof of \eqref{eq:deviation_Q} is obtained exactly in the same way as in Proposition 3 of \cite{gloter2007estimation} and is thus omitted. \\

We are ready to prove the estimate \eqref{eq: energy levels} of Proposition \ref{propo:energy_levels}. Let $J_0$ and $N$ be two arbitrary integers and $r>0$. We have
\begin{align*}
\mathbb{P}_{H,\eta}
\Big(
&\inf_{J_0 \leq j \leq N-1} 2^{2jH} Q_{j,N-j-1} \leq r
\Big)
\leq 
\sum_{j=J_0}^{N-1}
\mathbb{P}_{H,\eta}
\big(
2^{2jH} Q_{j,N-j-1} \leq r
\big)
\\
&\leq   
\sum_{j=J_0}^{N-1}
\mathbb{P}_{H,\eta}
\big(
 Q_{j,N-j-1} - \kappa_{N-j-1}(H) \eta^2 2^{-2jH-1}  \leq 
  (r - \kappa_{N-j-1}(H) \eta^2/2)2^{-2jH}
\big).
\end{align*}

By Lemma \ref{lem:kappa}, we can pick $r$ small enough so that $r - \kappa_p(H) \eta^2/2 \leq - c$ for some $c > 0$ fixed, uniformly in $p \geq 0$ and $(H,\eta) \in \mathcal D$. The estimate \eqref{eq:deviation_Q} yields
\begin{align*}
\mathbb{P}_{H,\eta}
\Big(
\inf_{J_0 \leq j \leq N-1} 2^{2jH} Q_{j,N-j-1} \leq r
\Big)
&\leq 
\sum_{j=J_0}^{N-1}
	C 2^{-j(1+4H)}
  c^{-2}2^{4jH}
\leq
C' 2^{-J_0}
\end{align*}
and \eqref{eq: energy levels} follows.

\subsection{Proof of Proposition \ref{propo:emp_energy_levels}}
\label{sec:proof:propoempenergylevels}

Recall from \eqref{eq:def:Qhat} that   
$\widehat{Q}_{j,p,n} = \sum_{k=0}^{2^{j-1}-1} \widehat{d^2_{}}_{\!\!\!j,k,p,n}$
where
$\widehat{d^2}_{\!\!\!j,k,p,n} = (\widetilde{d}_{j,k,p,n})^2 - \lambda_{j,p,n}$ and $ \lambda_{j,p,n} = 6\Var( \varepsilon_{1,m}) 2^{-j-p}$. Moreover, recall the decomposition
$\widetilde{d}_{j,k,p, n} = d_{j,k,p} + e_{j,k,p,n}$ where $e_{j,k,p,n}$ is defined in \eqref{eq:def:e}. We obtain
\begin{align*}
\widehat{Q}_{j,p,n} = Q_{j,p} 
+ \sum_{k=0}^{2^{j-1}-1} ( e_{j,k,p,n}^2 - \lambda_{j,p,n})
+ 2 \sum_{k=0}^{2^{j-1}-1} e_{j,k,p,n}d_{j,k,p}.
\end{align*}
We plan to prove the following estimates: 
\begin{align}
\label{eq:estimate_1}
&\mathbb{E}_{H,\eta} \Big[\Big(\sum_{k = 0}^{2^{-j-1}-1} ( e_{jk,p,n}^2 - \lambda_{j,p,n})\Big)^2\Big] \leq C(n \delta)^{-2} 2^{-j-2p}
\end{align}
and
\begin{align}
\label{eq:estimate_2}
\mathbb{E}_{H,\eta} \Big[\Big(\sum_{k = 0}^{2^{-j-1}-1} e_{j,k,p,n} d_{j,k,p}\Big)^2\Big] \leq C(n \delta)^{-1} 2^{-j(1+2H)-p}.
\end{align}

To prove \eqref{eq:estimate_1}, first notice that the random variables $e_{j,k,p,n}^2 - \lambda_{j,p,n}$ are centered and that $e_{j,k_1,p,n}^2 - \lambda_{j,p,n}$ and $e_{j,k_2,p,n}^2 - \lambda_{j,p,n}$ are independent whenever $|k_1 - k_2| \geq 3$. Thus 
\begin{align*}
\mathbb{E}_{H,\eta} \Big[\Big(\sum_k ( e_{j,k,p,n}^2 - \lambda_{j,p,n})\Big)^2\Big]
&= 
\sum_{k_1, k_2} \mathbb{E}_{H,\eta}\Big[( e_{j,k_1,p,n}^2 - \lambda_{j,p,n})( e_{j,k_2,p,n}^2 - \lambda_{j,p,n})\Big]
\\
&\leq 
C\sum_{k} \mathbb{E}_{H,\eta}\Big[{( e_{j,k,p,n}^2 - \lambda_{j,p,n})^2}\Big]
\\
&
\leq 
C\sum_{k} \mathbb{E}_{H,\eta}\Big[{ e_{j,k,p,n}^4}\Big].
\end{align*}

The noise variable $e_{j,k,p,n}$  is a sum of $2^p$ independent centered random variables with the same law as $\varepsilon_{1,n}-2 \varepsilon_{2,n}+\varepsilon_{3,n}$. Therefore, its 4-th moment is of order $2^{2p}$ up to a multiplicative factor bounded above by $\mathbb{E}_{H,\eta} \big[\varepsilon_{1,n}^4]$, as implied  by Rosenthal's inequality. We derive 
\begin{align*}
& \mathbb{E}_{H,\eta}\big[e_{j,k,p,n}^4\big] \leq C2^{-2j-2p} \mathbb{E}_{H,\eta} \big[\varepsilon_{1,n}^4]. 
\end{align*}
Moreover, $n\delta \exp(\varepsilon_{1,n})$ has a chi-square distribution with  $n\delta$ degrees of freedom, hence $\mathbb{E}_{H,\eta}\big[\varepsilon_{1,n}^4\big] \leq C (n\delta)^{-2}$ by Lemma \ref{lem:log_mom_chi2}. Thus
\begin{align*}
& \mathbb{E}_{H,\eta}\big[e_{j,k,p,n}^4\big] \leq C2^{-2j-2p} (n\delta)^{-2}
\end{align*}
and \eqref{eq:estimate_1} follows.\\

We now focus on the estimate \eqref{eq:estimate_2}. Notice that the random variables $(e_{j,k,p,n})_k$ are centered and independent of the variables $(d_{j,k,p})_k$. Moreover, $e_{j,k_1,p,n}$ and $e_{j,k_2,p,n}$ are independent if $|k_1-k_2| \geq 3$. Therefore we get
\begin{align*}
\mathbb{E}_{H,\eta} \Big[\Big(\sum_{k = 0}^{2^{-j-1}-1} e_{j,k,p,n} d_{j,k,p}\Big)^2\Big] 
&
\leq 
C
\sum_{k} 
\mathbb{E}_{H,\eta} \Big[ e_{j,k,p,n}^2 d_{j,k,p}^2\Big] 
\\
&\leq 
C
\sum_{k} 
\mathbb{E}_{H,\eta} \Big[ e_{j,k,p,n}^2 \Big] \mathbb{E}_{H,\eta} \Big[ d_{j,k,p}^2\Big].
\end{align*}
We conclude using Lemma \ref{lem:kappa} and the estimate 
$$\mathbb{E}_{H,\eta} \Big[ e_{j,k,p,n}^2 \Big] = \lambda_{j,p,n} = 6\Var( \varepsilon_{1,m}) 2^{-j-p}\leq C 2^{-j-p} (n\delta)^{-1}.$$

\subsection{Completion of proof when \texorpdfstring{$\delta \leq n^{-1/(2H+1)}$}{delta small}}
\label{sec:completion:small}

Suppose first that $\delta \leq n^{-1/(2H+1)}$ so that $n \leq \delta^{-(2H+1)} \leq 2^{(2H+1)N}$. Then we define for any $\varepsilon > 0$
\begin{align*}
J_n^-(\varepsilon) 
&=
\max \big\{ j \geq 1: \; r_-(\varepsilon) 2^{-2jH} \geq 2^j n^{-1} \big\} \wedge (N-1)
\\
&=
\bigg\lfloor 
\frac
{\log \big( r_-(\varepsilon) n\big )}
{2H+1}
\bigg \rfloor
 \wedge (N-1)
\\&
\leq
\bigg \lfloor 
N + 
\frac
{\log r_-(\varepsilon) }
{2H+1}
\bigg\rfloor
 \wedge (N-1).
\end{align*}
Since $r_-(\varepsilon) < \exp(-(2H+1))$ by Proposition \ref{propo:energy_levels}, we have $\log r_-(\varepsilon) < -(2H+1)$ and $J_n^-(\varepsilon) =
\lfloor 
\frac
{\log ( r_-(\varepsilon) n)}
{2H+1}
\rfloor$ so that
\begin{align}
\label{eq:order_Jnm}
\tfrac{1}{2} 
(r_-(\varepsilon) n)^{
1/(2H+1)
}
\leq
2^{J_n^-(\varepsilon)}
\leq
(r_-(\varepsilon) n)^{1/(2H+1)
}.
\end{align}
The following estimate ensures that with overwhelming probability, $ J_n^*$ can be controlled by $J_n^-(\varepsilon)$. 
\begin{lem}
\label{lemma:bound_Jnstar}
For any $\varepsilon > 0$, there exist $L(\varepsilon) > 0$  and $\varphi_n(\varepsilon)$ such that $\varphi_n(\varepsilon) \to 0$ as $n \to \infty$ and
\begin{align*}
\sup_{(H,\eta) \in \mathcal D}
\mathbb{P}_{H,\eta}(J_n^* < J_n^-(\varepsilon) - L(\varepsilon) ) \leq \varepsilon + \varphi_n(\varepsilon).
\end{align*}
\end{lem}

\begin{proof}
Let $L > 0$. For notational simplicity, we set $\overline{J}_n =  J_n^-(\varepsilon) - L$. We have
\begin{align*}
\mathbb{P}_{H,\eta}(J_n^* \geq \overline{J}_n )
&\geq
\mathbb{P}_{H,\eta}(\widehat{Q}_{\overline{J}_n, N-\overline{J}_n-1,n} \geq \nu_0  2^{\overline{J}_n}n^{-1}
 )
 \\
 &\geq
\mathbb{P}_{H,\eta}(
\widehat{Q}_{\overline{J}_n, N-\overline{J}_n-1}
-
Q_{\overline{J}_n, N-\overline{J}_n-1,n}
\geq \nu_0 2^{\overline{J}_n} n^{-1}
-
2^{-2H\overline{J}_n}r_-(\varepsilon)
 )
 \\
&\;\;\;\;-
\mathbb{P}_{H,\eta}(
Q_{\overline{J}_n, N-\overline{J}_n-1}
\leq 2^{-2H\overline{J}_n}r_-(\varepsilon)
).
\end{align*}

Since $J_n^-(\varepsilon) \to \infty$ as $n\to\infty$, we have $J_0(\varepsilon) \leq \overline{J}_n \leq N-1$ for large enough $n$ ($J_0(\varepsilon)$ is defined in Proposition \ref{propo:energy_levels}).
This implies by Proposition \ref{propo:energy_levels} that
\begin{align*}
\mathbb{P}_{H,\eta}\big(
Q_{\overline{J}_n, N-\overline{J}_n-1}
\leq 2^{-2H\overline{J}_n}r_-(\varepsilon)
\big) \leq \mathbb{P}_{H,\eta}\Big(
\inf_{J_0(\varepsilon) \leq j \leq N-1} 2^{2Hj} Q_{j, N-j-1}
\leq r_-(\varepsilon)
\Big) \leq \varepsilon.
\end{align*}
By definition of $J_n^-(\varepsilon)$, we also have
$ r_-(\varepsilon) 2^{-2J_n^-(\varepsilon)H} \geq 2^{J_n^-(\varepsilon)} n^{-1}$. Thus,
\begin{align*}
\mathbb{P}_{H,\eta}\big( J_n^* \geq \overline{J}_n )
 &\geq
\mathbb{P}_{H,\eta}(
\widehat{Q}_{\overline{J}_n, N-\overline{J}_n-1}
-
Q_{\overline{J}_n, N-\overline{J}_n-1,n}
\geq 
2^{J_n^-(\varepsilon)}n^{-1}
(
\nu_0 2^{-L}
-
2^{2HL}
)
 \big)
-\varepsilon \\
 &\geq
1-\mathbb{P}_{H,\eta}(
\widehat{Q}_{\overline{J}_n, N-\overline{J}_n-1}
-
Q_{\overline{J}_n, N-\overline{J}_n-1,n}
\leq 
-2^{J_n^-(\varepsilon)}n^{-1}
 \big)
-\varepsilon
\end{align*}
as soon as $L$ is taken sufficiently large, so that $\nu_0 2^{-L}
-
2^{2HL} \leq -1$. Using Proposition \ref{propo:emp_energy_levels}, we derive that
\begin{align*}
\mathbb{P}_{H,\eta}( J_n^* < \overline{J}_n) 
& \leq \varepsilon
+\mathbb{P}_{H,\eta}\big(
|
\widehat{Q}_{\overline{J}_n, N_n-\overline{J}_n-1}
-
Q_{\overline{J}_n, N_n-\overline{J}_n-1,n}
|
\geq 
2^{J_n^-(\varepsilon)}n^{-1}
\big) 
\\
& \leq \varepsilon + 
C2^{-2J_n^-(\varepsilon)}n^{2} \big(  2^{-2N+\overline{J}_n}n^{-2}\delta^{-2}
+
n^{-1} \delta^{-1} 2^{-N-2H\overline{J}_n}
 \big ) 
\\
& \leq \varepsilon + C'
\big(
2^{-J_n^-(\varepsilon)}
+ 
2^{-2(H+1)J_n^-(\varepsilon)}n 
\big),
\end{align*}
where the constant $C'$ may depend on $\varepsilon$ but is independent of the model parameters. The estimate \eqref{eq:order_Jnm} completes the proof.
\end{proof}

We now prove that the sequence $(v_n^{-1} B_n)_n$ is tight under $\delta  = \delta_n \leq n^{-1/(2H+1)}$. Let $n\geq 1$ be sufficiently large so that 
\begin{align*}
J_n^-(\varepsilon) - L(\varepsilon) \geq J_0(\varepsilon)\;\;\text{and}\;\;N \geq N_0(\varepsilon)
\end{align*}
simultaneously.  Let $M > 0$. 
We first write
\begin{align*}
\mathbb{P}_{H,\eta}\big(v_n^{-1} B_n \geq M \big) 
\leq I+II+III,
\end{align*}
with
\begin{align}
\nonumber
I & = \mathbb{P}_{H,\eta}\Big( 
	v_n^{-1}  B_n \geq M, \; 
	{ J_n^*} \geq J_n^-(\varepsilon) - L(\varepsilon), \;
	\inf_{J_0 \leq j \leq N-1} 2^{2jH} Q_{j,N-j-1} \geq r_-(\varepsilon)
\Big),\\
\label{eq:multilevel:II} II & = \mathbb{P}_{H,\eta}( J_n^* < J_n^-(\varepsilon) - L(\varepsilon) ),\\
\nonumber III &= 
\mathbb{P}_{H,\eta} \Big( \inf_{J_0 \leq j \leq N-1} 2^{2jH} Q_{j,N-j-1} \leq r_-(\varepsilon) \Big)
. 
\end{align}
The term $II$ is smaller than $\varepsilon+\varphi_n(\varepsilon)$ by Lemma \ref{lemma:bound_Jnstar}, while $III$ is smaller than $\varepsilon$ by Proposition \ref{propo:energy_levels}. For the term $I$, we apply
estimate \eqref{eq:deviation_Q} to obtain
\begin{align*}
I
&\leq
\sum_{j=J_n^-(\varepsilon) - L(\varepsilon)}^{N-1}
\mathbb{P}_{H,\eta} \Big( v_n^{-1} B_n \geq M,\;J_n^* = j,\;{Q}_{j,N-j-1} \geq 2^{-2Hj}r_-(\varepsilon)\Big )
\\
&\leq
\sum_{j=J_n^-(\varepsilon) - L(\varepsilon)}^{N-1}
\mathbb{P}_{H,\eta}\big(
\big | {Q}_{j+1,N-j-1}  - 2^{-2H}{Q}_{j,N-j-1}  \big | \geq M 2^{-2Hj} r_-(\varepsilon) v_n
\big)
\\
&\leq
\sum_{j=J_n^-(\varepsilon) - L(\varepsilon)}^{N-1}
2^{-j} C M^{-2} r_-(\varepsilon)^{-2} v_n^{-2}
\\
&
\leq 
C 2^{-J_n^-(\varepsilon) + L(\varepsilon)}  M^{-2} r_-(\varepsilon)^{-2} v_n^{-2}.
\end{align*}
Now we obtain the tightness of $(v_n^{-1} B_n)_n$ thanks to the fact that  $2^{-J_n^-(\varepsilon)} v_n^{-2}$ is bounded, see \eqref{eq:order_Jnm}. \\

We next study the tightness of $(v_n^{-1} V_n^{(1)})_n$. Let $M > 0$. We have
\begin{align*}
\mathbb{P}_{H,\eta} \Big({v_n^{-1} V_n^{(1)} \geq M}\Big)
&\leq
I' + II
\end{align*}
where
\begin{align*}
I' &= \mathbb{P}_{H,\eta} \Big({v_n^{-1} V_n^{(1)} \geq M}, \; J_n^* \geq J_n^-(\varepsilon) - L(\varepsilon) \Big)
\end{align*}
and where $II$ is defined in \eqref{eq:multilevel:II}. Recall that $II \leq \varepsilon + \varphi_n(\varepsilon)$ by Lemma \ref{lemma:bound_Jnstar}.
By definition, we also have
\begin{align*}
\widehat{Q}_{J_n^*,N-J_n^*-1,n} \geq \nu_0 2^{J_n^*}n^{-1}
\;\;\text{on}\;\;
\big\{
	J_n^* \geq J_n^-(\varepsilon) - L(\varepsilon)
\big\}.
\end{align*}
Hence, by Proposition \ref{propo:emp_energy_levels},
\begin{align*}
I' &\leq
\sum_{j=J_n^-(\varepsilon) - L(\varepsilon) }^{N-1}\mathbb{P}_{H,\eta}\Big({v_n^{-1} | \widehat{Q}_{j+1,N-j-1,n} - {Q}_{j+1,N-j-1} | \geq M \nu_0 2^{j}n^{-1} } \Big)
\\
&\leq
C \sum_{j=J_n^-(\varepsilon) - L(\varepsilon) }^{N-1}
M^{-2} \nu_0^{-2} 2^{-2j}n^{2}v_n^{-2}
\Big(
n^{-2} \delta^{-2} 2^{j-2N} + n^{-1}\delta^{-1}2^{-2Hj-N } \Big)
\\
&\leq
C 
M^{-2}\nu_0^{-2}v_n^{-2}
\sum_{j=J_n^-(\varepsilon) - L(\varepsilon) }^{N-1}
 \big(2^{-j} + n2^{-2(H+1)j }\big)
\\
&\leq
C' 
M^{-2}\nu_0^{-2}v_n^{-2}
( 2^{-J_n^-(\varepsilon)} + n2^{-2(H+1)J_n^-(\varepsilon)}),
\end{align*}
where $C'$ may depend on $\varepsilon$. We conclude noticing that \eqref{eq:order_Jnm} implies that both $v_n^{-2}
2^{-J_n^-(\varepsilon)}$ and $v_n^{-2}
n2^{-2(H+1)J_n^-(\varepsilon)}$ are bounded. \\

We eventually prove the tightness of $(v_n^{-1}V_n^{(2)})_n$. First, note that  $v_n^{-1}B_n \leq M'$ for some $M'>0$ implies
\begin{align*}
2^{-2H}- M'v_n
\leq
\frac{{Q}_{J_n^*+1,N_n- J_n^*-1}}{{Q}_{ J_n^*,N_n- J_n^*-1}}
\leq 
2^{-2H}+ M'v_n.
\end{align*}
For $M' > 0$, we have $2^{-2H}- M' v_n > 0$, at least when $n$ is large enough since $v_n \to 0$ as $n \rightarrow \infty$. In that case,
\begin{align*}
\frac{{Q}_{J_n^*,N-J_n^*-1}}{{Q}_{J_n^*+1,N-J_n^*-1}} \geq (2^{-2H}+ M'v_n)^{-1}.
\end{align*}

Let $M >0$. It follows that
\begin{align*}
&\mathbb{P}_{H,\eta}\Big({v_n^{-1} V_n^{(2)} \geq M}\Big) \\
&\leq 
\mathbb{P}_{H,\eta}\Big({v_n^{-1}  
|  	 
		\widehat{Q}_{J_n^*,N-J_n^*-1,n}-
		         {Q}_{J_n^*,N-J_n^*-1}  	
| 
\geq 
\frac{M\widehat{Q}_{J_n^*,N-J_n^*-1,n}{Q}_{J_n^*,N-J_n^*-1}}{Q_{J_n^*+1,N-J_n^*-1}}}\Big)
\\
&\leq 
\mathbb{P}_{H,\eta}\Big(
|  	 
		\widehat{Q}_{J_n^*,N-J_n^*-1,n}-
		         {Q}_{J_n^*,N-J_n^*-1}  	  
| 
\geq 
\frac{M\widehat{Q}_{J_n^*,N-J_n^*-1,n}}{2^{-2H}v_n^{-1} + M'}\Big)
+ 
\mathbb{P}_{H,\eta}\Big({v_n^{-1} B_n \geq M'}\Big).
\end{align*}
We can then repeat the proof for the tightness of $(v_n^{-1}V_n^{(1)})_n$  in the same way and we conclude by noticing that $2^{-2H}v_n^{-1} + M'$ is of the same order as $v_n^{-1}$.

\subsection{Completion of proof when \texorpdfstring{$\delta  = \delta_n \geq n^{-1/(2H+1)}$}{delta large}}
\label{sec:completion:large}

Suppose now that $\delta \geq n^{-1/(2H+1)}$. We quickly cover this case using the same arguments as in the first case. Note that  now  $v_n = \delta^{1/2}$. \\

Recall that $J_n^*$ is defined by $J_n^* = \max \{ 2 \leq j \leq N-1: \widehat{Q}_{j,N-j-1,n} \geq \nu_0 2^j n^{-1} \}$. The following estimate replaces Lemma \ref{lemma:bound_Jnstar} from the previous case.

\begin{lem}
\label{lemma:behavior_Jnstar}
We have
\begin{align*}
\sup_{(H,\eta) \in \mathcal D} \mathbb{P}_{H,\eta}(J_n^* < N-1) \to 0.
\end{align*}
\end{lem}

\begin{proof}
Recall that $\nu_0 < \inf_{H, \eta} \eta^2 \kappa_0(H) 2^{2H}$ so there exists a constant $\iota > 0$ such that for any $H$ and $\eta$, we have $\nu_0 -\eta^2 \kappa_0(H) 2^{2H}  < - 2 \iota$. Then
\begin{align*}
\mathbb{P}_{H,\eta}(J_n^* < N-1) 
&= 
\mathbb{P}_{H,\eta} (\widehat{Q}_{N-1,0,n} < \nu_0 / (2n\delta) ).
\end{align*}
Since $\delta \geq n^{-1/(2H+1)}$, we have $(n\delta)^{-1} \leq \delta^{2H}$ and therefore
\begin{align*}
\mathbb{P}_{H,\eta}(J_n^* < N-1) 
&= 
	\mathbb{P}_{H,\eta} (\widehat{Q}_{N-1,0,n} < \tfrac{1}{2}\nu_0 \delta^{2H}  )
\leq I + II
\end{align*}
where
\begin{align*}
I &= \mathbb{P}_{H,\eta} (\widehat{Q}_{N-1,0,n} - Q_{N-1,0} < \tfrac{1}{2}(\nu_0 - \kappa_0(H) \eta^22^{2H} + \iota) \delta^{2H}  )
\\
\text{and}\;\;\;\; II &= \mathbb{P}_{H,\eta} ( Q_{N-1,0} < \tfrac{1}{2}(\kappa_0(H) \eta^22^{2H} - \iota) \delta^{2H} ).
\end{align*}
Then we have by Proposition \ref{propo:emp_energy_levels}
\begin{align*}
I 
&\leq 
\mathbb{P}_{H,\eta} (\widehat{Q}_{N-1,0,n} - Q_{N-1,0} < - \tfrac{1}{2}\iota \delta^{2H}  )
\\
&\leq 
C
\big(
n^{-2} \delta^{-2} 2^{-N}
+ n^{-1} \delta^{-1} 2^{-N(1+2H)}\big)
 \iota^{-2} \delta^{-4H}
 \\
&\leq
C
\big(
n^{-2} \delta^{-(1+4H)} 
+ n^{-1} \delta^{-2H)}\big)
 \iota^{-2} 
\end{align*}
since $\delta = 2^{-N}$. This shows $I \to 0$ uniformly over $\mathcal{D}$ as $n\to\infty$ since $\delta \geq n^{-1/(2H+1)}$. We then have
\begin{align*}
II
&= \mathbb{P}_{H,\eta} ( Q_{N-1,0} - \tfrac{1}{2}\kappa_0(H) \eta^2 2^{2H} \delta^{2H}  < - \tfrac{1}{2}\iota \delta^{2H} )
\\
&\leq C\delta \iota^{-2}
\end{align*}
by \eqref{eq:deviation_Q} and we can conclude since $\delta \to 0$ when $n \to \infty$.
\end{proof}

We are ready to prove the tightness part. We start with the sequence $(v_n^{-1} B_n)_n$. We have
\begin{align*}
&\mathbb{P}_{H,\eta}\Big({v_n^{-1} B_n \geq M}\Big) \leq 
	\mathbb{P}_{H,\eta}\Big({ \Big | \frac{{Q}_{N,0} }{{Q}_{N-1,0} } - 2^{-2H} \Big | \geq M\delta^{1/2}}\Big)
+\mathbb{P}(J_n^* < N-1) 
\end{align*}
so that we can focus on the first term by Lemma \ref{lemma:behavior_Jnstar}. By Proposition \ref{propo:energy_levels}
\begin{align*}
& \mathbb{P}_{H,\eta}\Big({ \Big | \frac{{Q}_{N,0} }{{Q}_{N-1,0} } - 2^{-2H} \Big | \geq M\delta^{1/2}}\Big) \\
&\leq 
	\mathbb{P}_{H,\eta} (
		| {Q}_{N,0}  - 2^{-2H}{Q}_{N-1,0} | 
		\geq 
		M\delta^{1/2+2H}r_-(\varepsilon)2^{2H}
	)
+
\varepsilon
\end{align*}
since $2^N=\delta^{-1}$ by definition. Using also \eqref{eq:deviation_Q}, we eventually obtain
\begin{align*}
\mathbb{P}_{H,\eta}\Big({ \Big | \frac{{Q}_{N,0} }{{Q}_{N-1,0} } - 2^{-2H} \Big | \geq M\delta^{1/2}}\Big)
\leq C
		M^{-2}r_-(\varepsilon)^{-2}
+
\varepsilon,
\end{align*}
and this proves the result for $(v_n^{-1}B_n)_n$. Concerning the sequence  $(v_n^{-1} V_n^{(1)})_n$, for $M >0$ we have
\begin{align*}
\mathbb{P}_{H,\eta}\big(v_n^{-1}  V_n^{(1)} \geq M\big) 
\leq 
\mathbb{P}_{H,\eta}\big(  | \widehat{Q}_{N,0,n} - {Q}_{N,0}  | \geq \tfrac{1}{2}M \nu_0 \delta^{1/2+2H}\big)
+
\mathbb{P}_{H,\eta}(J_n^* < N-1) 
\end{align*}
since $\widehat{Q}_{N-1,0,n} \geq \tfrac{1}{2}\nu_0 \delta^{2H}$ on $\{J_n^* = N-1\}$. By Proposition \ref{propo:energy_levels} and using $\delta_n \geq n^{-1/(2H+1)}$, we obtain
\begin{align*}
\mathbb{P}_{H,\eta}\big(v_n^{-1}  V_n^{(1)} \geq M \big)
&\leq
C\kappa^{-2} \delta^{-4H-2}n^{-2}+
+\mathbb{P}_{H,\eta}(J_n^* < N-1) 
\\&\leq CM^{-2} + 
\mathbb{P}_{H,\eta}(J_n^* < N-1), 
\end{align*}
 and the tightness of $(v_n^{-1}  V_n^{(1)})_n$ follows. We eventually deal with the sequence $(n_n^{-1} V_n^{(2)})_n$.
Proceeding as for the case $\delta \leq n^{-1/(2H+1)}$, for $M, M' > 0$ and $n$ large enough, we have
\begin{align*}
& \mathbb{P}_{H,\eta}\big({v_n^{-1} V_n^{(2)} \geq M}\big) \\
&
\leq \mathbb{P}_{H,\eta}\Big(
|  	 
		\widehat{Q}_{J_n^*,N-J_n^*-1,n}-
		         {Q}_{J_n^*,N-J_n^*-1}  	  
| 
\geq 
\frac{M\widehat{Q}_{J_n^*,N-J_n^*-1,n}}{2^{-2H}v_n^{-1} + M'}\Big)
+ 
\mathbb{P}_{H,\eta}\big({v_n^{-1} B_n \geq M'}\big).
\end{align*}

The first term is similar to $V_n^{(1)}$ since $2^{-2H}v_n^{-1} + M'$ is of the same order as $v_n^{-1}$. The second term can be made arbitrarily small since $(v_n^{-1} B_n)_n$ is tight.

\section{Proof of Theorem \ref{thm:optimality}}
\label{sec:proof:optimality}

The proof of Theorem \ref{thm:optimality} is similar to that of Theorem \ref{thm:lower_multiscale}. For conciseness, we only provide here a result similar to Proposition \ref{prop:bound_TV}. The same expansion as in  Section \ref{sec:proof:lower_multiscale} would then conclude the proof of Theorem \ref{thm:optimality}. Recall that $\norm{\mu }_{TV} = \sup_{\norm{f}_{\infty} \leq 1} | \int f d\mu |$ is the total variation of a signed measure $\mu$. We also denote by $\mathbb{P}_f^n$ the law of $(S_{i/n})_i$ given $\eta W^H_t = f(t)$.

\begin{prop}
\label{prop:bound_TV_general}
Let $f$ and $g$ be two bounded functions. Then there exists $c_0>0$ such that
\begin{align}
\label{eq:bound_TV_general:est1}
\norm{\mathbb{P}_f^n - \mathbb{P}_g^n}_{TV} \leq c_0 
\sqrt{n} e^{c_0(\norm{f}_\infty + \norm{g}_\infty)}
\norm{f - g}_\infty.
\end{align}
Moreover, there exist $c_1 > 0$ and a universal nonincreasing positive function $R$ such that
\begin{align}
\label{eq:bound_TV_general:est2}
1 - \tfrac{1}{2} \norm{\mathbb{P}_f^n - \mathbb{P}_g^n}_{TV} \geq R(n e^{c_1(\norm{f}_\infty + \norm{g}_\infty)}
\norm{f - g}_\infty^2).
\end{align}
\end{prop}

\begin{proof}
As in the proof of Proposition \ref{prop:bound_TV}, let $K( \mu, \nu)
=
\int \log ( \frac{d\mu}{d\nu} ) \; d\mu$ denote the Kullback--Leibler divergence between two probability measures $\mu$ and $\nu$. Recall Pinsker's inequality $\norm{\mu - \nu}_{TV}^2 \leq 2K( \mu, \nu)$.

Note that $(S_{i/n} - S_{(i-1)/n})_i$ has the same law as
$\big(\int_{(i-1)/n}^{i/n} \sigma^2_t dt \big)^{1/2} \xi_{i,n}
$ where the random variables $(\xi_{i,n})_i$ are independent standard Gaussian variables since $(W_t^H)$ and $(B_t)$ are independent. Therefore
\begin{align*}
K(\mathbb{P}_f^n, \mathbb{P}_g^n)
\leq
\frac{1}{2}
\sum_{i=1}^n B \Big( \Big(\int_{(i-1)/n}^{i/n} e^{f(t)} dt\Big) \Big( \int_{(i-1)/n}^{i/n} e^{g(t)} dt\Big)^{-1} \Big),
\end{align*}
where $B(x) =\ x - \log x - 1$. Pinsker's inequality yields
\begin{align*}
\norm{\mathbb{P}_f^n - \mathbb{P}_g^n}_{TV}^2
\leq
\sum_{i=1}^n B \Big( \Big(\int_{(i-1)/n}^{i/n} e^{f(t)} dt\Big) \Big( \int_{(i-1)/n}^{i/n} e^{g(t)} dt\Big)^{-1} \Big).
\end{align*}
Notice then that for $x > 0$, $B(x) \leq (x-1)^2 + (1/x - 1)^2$ and that entails
\begin{align*}
\norm{\mathbb{P}_f^n - \mathbb{P}_g^n}_{TV}^2
\leq
\sum_{i=1}^n 
\bigg[
\Big( \frac{\int_{(i-1)/n}^{i/n} \big(e^{f(t)}- e^{g(t)}\big) dt}
{\int_{(i-1)/n}^{i/n} e^{g(t)} dt}
\Big)^2
+
\Big( \frac{\int_{(i-1)/n}^{i/n} \big(e^{g(t)}- e^{f(t)}\big) dt}
{\int_{(i-1)/n}^{i/n} e^{f(t)} dt}
\Big)^2
\bigg]
.
\end{align*}

Moreover, 
\begin{align*}
\int_{(i-1)/n}^{i/n} e^{g(t)} dt \geq n^{-1} e^{-\norm{g}_\infty}
\text{ and }
\int_{(i-1)/n}^{i/n} e^{f(t)} dt \geq n^{-1} e^{-\norm{f}_\infty}
\end{align*}
so
\begin{align*}
\norm{\mathbb{P}_f^n - \mathbb{P}_g^n}_{TV}^2
&\leq
2n^2 (e^{2\norm{f}_\infty} + e^{2\norm{g}_\infty})
\sum_{i=1}^n 
\Big( \int_{(i-1)/n}^{i/n} \big|e^{f(t)}- e^{g(t)}\big| dt
\Big)^2.
\end{align*}

By Jensen inequality, we obtain
\begin{align*}
\norm{\mathbb{P}_f^n - \mathbb{P}_g^n}_{TV}^2
&\leq
2n (e^{2\norm{f}_\infty} + e^{2\norm{g}_\infty}) \sum_{i=1}^n 
\int_{(i-1)/n}^{i/n} \Big(  e^{f(t)}- e^{g(t)} \Big)^2 dt
\\&\leq
2n (e^{2\norm{f}_\infty} + e^{2\norm{g}_\infty})
\norm{e^{f}- e^{g}}_\infty^2.
\end{align*}

Since $\norm{e^{f}- e^{g}}_\infty \leq e^{\max(\norm{f}_\infty,\norm{g}_\infty)} \norm{f - g}_\infty $ we finally obtain
\begin{align*}
\norm{\mathbb{P}_f^n - \mathbb{P}_g^n}_{TV}^2
&\leq
2n (e^{2\norm{f}_\infty} + e^{2\norm{g}_\infty})
e^{2\max(\norm{f}_\infty,\norm{g}_\infty)} \norm{f - g}_\infty^2
\\
&\leq
4n e^{4\norm{f}_\infty + 4\norm{g}_\infty}
\norm{f - g}_\infty^2,
\end{align*}
which proves \eqref{eq:bound_TV_general:est1}. For \eqref{eq:bound_TV_general:est2}, we proceed exactly as for \eqref{eq:bound_TV:est2}. We omit the details.
\end{proof}

\section{Proof of the results of Section \ref{subsec:energy}}
\label{sec:energy:proof}

\subsection{Notation and organisation of the proofs}
\label{subsec:energy:proof:notations}

Let $b \geq 1$ and $s\geq 1$. Consider $\mathbf{r}\in \{ 1, \dots, 2S \}^s$ such that $\sum_i \mathbf{r}_i = b$. If $b = 1$, we write
\begin{align*}
\mathfrak{W}^{H,\mathbf{r}}_{j,p,k,l} 
&= 
2^{j+p} \int_{k2^{-j}+l2^{-j-p}}^{k2^{-j}+(l+1)2^{-j-p}} W^H_u du =
\int_{0}^{1} W^H_{k2^{-j} + (l+u)2^{-j-p}} du
\end{align*}
and if $b \geq 2$, we write
\begin{align*}
\mathfrak{W}^{H,\mathbf{r}}_{j,p,k,l} 
&= 
\prod_{i=1}^{s}
\frac{2^{j+p}}{\mathbf{r}_i!} \int_{k2^{-j}+l2^{-j-p}}^{k2^{-j}+(l+1)2^{-j-p}} (W^H_u-W^H_{k2^{-j}+l2^{-j-p}})^{\mathbf{r}_i}
 du
\\
&= 
\frac{1}{\mathbf{r}!}
\int_{[0,1]^s}
\prod_{i=1}^{s}
(W^H_{k2^{-j} + (l+u_i)2^{-j-p}}-W^H_{k2^{-j}+l2^{-j-p}})^{\mathbf{r}_i}
 du
\end{align*}
where $\mathbf{r}! =
\prod_{i=1}^{s}
\mathbf{r}_i! $. Define also $\mathfrak{W}^{H,\mathbf{r}}_{p,l} = \mathfrak{W}^{H,\mathbf{r}}_{0,p,0,l}$. Since $\mathfrak{W}^{H,\mathbf{r}}_{j,p,k,l} = \mathfrak{W}^{H,\mathbf{r}}_{j,p,0,2^{p}k+l}$, we get by self-similarity of   fractional Brownian motion that
\begin{align}
\label{eq:selfsimilarity}
\Big( 
\mathfrak{W}^{H,\mathbf{r}}_{j,p,0,l}
\Big)_{\mathbf{r}, l}
=
\Big( 
2^{-jH\sum_i \mathbf{r}_i}
\mathfrak{W}^{H,\mathbf{r}}_{p,l}
\Big)_{\mathbf{r}, l} \text{ in distribution.}
\end{align}

We fix $H^* \in (0, H_-)$ such that $(2S+1)H^* \geq H_+$. With the notation introduced in Proposition \ref{prop:development_iv}, the following decomposition holds: 
$$d_{j,k,p} = g_{j,k,p} + z_{j,p,k},$$ where
\begin{align}
\label{eq:g:def}
g_{j,k,p} 
&= 
2^{-p-j/2}
\sum_{l=0}^{2^p-1}
\sum_{b=1}^{2S} 
\bigg[
\eta^{b}
\sum_{s=1}^{2S} \Big(\frac{(-1)^{s-1}}{s}
\sum_{\substack{\mathbf{r} \in \{ 1, \dots, 2S \}^s \\ \sum_j \mathbf{r}_j = b}} \big( \mathfrak{W}^{H,\mathbf{r}}_{j,p,k+1,l}
-
\mathfrak{W}^{H,\mathbf{r}}_{j,p,k,l} \big)\Big)\bigg],
\\
\label{eq:z:def}
z_{j,p,k} &=
2^{-p-j/2}2^{-(j+p)(2S+1)H^*}
\sum_{l=0}^{2^p-1}
\Big(
Z((k+1)2^{p}+l, 2^{-j-p}) - Z(k2^{p}+l, 2^{-j-p})\Big).
\end{align}
Writing $G_{j,p} = \sum_k g_{j,p,k}^2$, the proofs of the results of Section \ref{subsec:energy} follow from similar results on $G$ together with an appropriate control of the error terms arising from $z$. In Section \ref{subsec:energy:proof:computational}, we gather useful computational lemmas. Section \ref{subsec:energy:Q} details the proof of Proposition \ref{prop:energy}. In Section \ref{subsec:energy:kappa} we gather the proofs of the estimates for the functions $\kappa$ and finally, in Section \ref{subsec:energy:bound} we prove Proposition \ref{prop:bound_energy}.\\

In the following, we will also use the notation $\barsum$ to indicate a sum over all indices $1 \leq b_1, b_2 \leq 2S$, $1 \leq s_1, s_2 \leq 2S$ and multi-indices $\mathbf{r_1} \in \{ 1, \dots, S \}^{s_1}$ and $\mathbf{r_2} \in \{ 1, \dots, S \}^{s_2}$ such that $\sum_j \mathbf{r_1}_j = b_1$ and $\sum_j \mathbf{r_2}_j = b_2$. Additional subscripts to $\barsum$ will denote additional constraints on the indices.

\subsection{Computational lemmas}
\label{subsec:energy:proof:computational}

\begin{lem}
\label{lem:m2_g}
\begin{align}
\label{eq:dev_g}
\mathbb{E} \Big [
	g_{j,p,k}^2
\Big ]
&=
2^{-j}\sum_{a=1}^{S} \Big( \eta^{2a} 2^{-2jaH} \kappa_{p,a}(H) \Big) + O(2^{-j(2S+1)H-j})
\end{align}
uniformly in $H$, $j$,$k$ and $p$. The functions $H \mapsto \kappa_{p,a}(H)$ are explictly given by
\begin{align}
\label{eq:def_kappas}
\kappa_{p,a}(H)
=
2^{-2p}
\sum_{|l| < 2^p} 
{\barsum_{b_1 + b_2 = 2a}}
\frac{(-1)^{s_1+s_2}} 
{s_1 s_2} (2^p - |l|)
\mathbb{E} \Big [
	\big( \mathfrak{W}^{H,\mathbf{r_1}}_{p,2^p} - \mathfrak{W}^{H,\mathbf{r_1}}_{p,0} \big)
	\big( \mathfrak{W}^{H,\mathbf{r_2}}_{p,2^p + l} - \mathfrak{W}^{H,\mathbf{r_2}}_{p,l} \big)
\Big ].
\end{align}

\end{lem}

\begin{proof}

Using \eqref{eq:selfsimilarity}, we have that $\mathbb{E} \Big [g_{j,p,k}^2
\Big ]$ equals
\begin{align*}
&
\frac{1}{2^{2p+j}}
\sum_{l_1, l_2 = 0}^{2^p - 1}
\barsum
\eta^{b_1+b_2}
\frac{(-1)^{s_1+s_2}}{s_1s_2}
\mathbb{E} \Big [
	\big( \mathfrak{W}^{H,\mathbf{r_1}}_{j,p,k+1,l_1} - \mathfrak{W}^{H,\mathbf{r_1}}_{j,p,k,l_1} \big)
	\big( \mathfrak{W}^{H,\mathbf{r_2}}_{j,p,k+1,l_2} - \mathfrak{W}^{H,\mathbf{r_2}}_{j,p,k,l_2} \big)
\Big ]
\\
&=
\frac{1}{2^{2p+j}}
\sum_{l_1, l_2 = 0}^{2^p - 1}
\barsum
\frac{(-1)^{s_1+s_2} \eta^{b_1+b_2} }{s_1s_2 2^{j(b_1 + b_2)H}} 
\mathbb{E} \Big [
	\big( \mathfrak{W}^{H,\mathbf{r_1}}_{p,2^p + l_1} - \mathfrak{W}^{H,\mathbf{r_1}}_{p,l_1} \big)
	\big( \mathfrak{W}^{H,\mathbf{r_2}}_{p,2^p +l_2} - \mathfrak{W}^{H,\mathbf{r_2}}_{p,l_2} \big)
\Big ].
\end{align*}
By the stationarity of   fractional Brownian motion, we also have
\begin{align*}
&\mathbb{E} \Big [
	\big( \mathfrak{W}^{H,\mathbf{r_1}}_{p,2^p + l_1} - \mathfrak{W}^{H,\mathbf{r_1}}_{p,l_1} \big)
	\big( \mathfrak{W}^{H,\mathbf{r_2}}_{p,2^p +l_2} - \mathfrak{W}^{H,\mathbf{r_2}}_{p,l_2} \big)
\Big ] \\
& =
\mathbb{E} \Big [
	\big( \mathfrak{W}^{H,\mathbf{r_1}}_{p,2^p} - \mathfrak{W}^{H,\mathbf{r_1}}_{p,0} \big)
	\big( \mathfrak{W}^{H,\mathbf{r_2}}_{p,2^p + l_2 - l_1} - \mathfrak{W}^{H,\mathbf{r_2}}_{p,l_2 - l_1} \big)
\Big ].
\end{align*}
Thus writing $l = l_2 - l_1$, we obtain that $\mathbb{E} \Big [g_{j,p,k}^2\Big ]$ equals
\begin{align*}
\frac{1}{2^{2p+j}}
\sum_{|l| < 2^p}
\barsum
\frac{(-1)^{s_1+s_2} \eta^{b_1+b_2} }{s_1s_2 2^{j(b_1 + b_2)H}} (2^p - |l|)
\mathbb{E} \Big [
	\big( \mathfrak{W}^{H,\mathbf{r_1}}_{p,2^p} - \mathfrak{W}^{H,\mathbf{r_1}}_{p,0} \big)
	\big( \mathfrak{W}^{H,\mathbf{r_2}}_{p,2^p + l} - \mathfrak{W}^{H,\mathbf{r_2}}_{p,l} \big)
\Big ].
\end{align*}

We can expand the expectation $\mathbb{E} \Big [
	\big( \mathfrak{W}^{H,\mathbf{r_1}}_{p,2^p} - \mathfrak{W}^{H,\mathbf{r_1}}_{p,0} \big)
	\big( \mathfrak{W}^{H,\mathbf{r_2}}_{p,2^p + l} - \mathfrak{W}^{H,\mathbf{r_2}}_{p,l} \big)
\Big ]$ by linearity. Since the expectation of the product of an odd number of (centered) Gaussian variables is null, we have
\begin{align*}
\mathbb{E} \Big [
	\big( \mathfrak{W}^{H,\mathbf{r_1}}_{p,2^p} - \mathfrak{W}^{H,\mathbf{r_1}}_{p,0} \big)
	\big( \mathfrak{W}^{H,\mathbf{r_2}}_{p,2^p + l} - \mathfrak{W}^{H,\mathbf{r_2}}_{p,l} \big)
\Big ]
=
0
\end{align*} 
if $b_1 + b_2$ is odd. Hölder's inequality also ensures that when $b_1 + b_2$ is even, 
this expectation is bounded uniformly in $p$, $H$, $l$, $\mathbf{r_1}$ and $\mathbf{r_2}$. 
\end{proof}

\begin{lem}
\label{lem:m4_g}
For any $j,p \geq 0$ and $k \leq 2^j$, we have
\begin{align*}
\mathbb{E} \big[ g_{j,p,k}^4 \big] \leq C2^{-j(4H+2)}.
\end{align*}
\end{lem}
\begin{proof}
Recall that $g_{j,k,p}$ is defined in \eqref{eq:g:def} so that 
\begin{align*}
\mathbb{E}[g_{j,k,p}^4]^{1/4}
&\leq
2^{-p-j/2}
\sum_{l=0}^{2^p-1}
\sum_{b=1}^{2S} 
\eta^{b}
\sum_{s=1}^{2S} \frac{1}{s}
\sum_{\substack{\mathbf{r} \in \{ 1, \dots, 2S \}^s \\ \sum_j \mathbf{r}_j = b}} 
\mathbb{E}[
\big(\mathfrak{W}^{H,\mathbf{r}}_{j,p,k+1,l}
-
\mathfrak{W}^{H,\mathbf{r}}_{j,p,k,l}\big)^4]^{1/4}
\end{align*}
and we conclude using Lemma \ref{lem:m4_incW}.
\end{proof}

\begin{lem}
\label{lem:m4_incW}
For any $j,p \geq 0$, $k \leq 2^j$, $l \leq 2^p$, $1 \leq s \leq 2S$, $1 \leq b \leq 2S$ and $\mathbf{r} \in \{ 1, \dots, 2S \}^{s}$ such that $\sum_j \mathbf{r}_j = b$, we have
\begin{align*}
\mathbb{E} \Big[\big(
\mathfrak{W}^{H,\mathbf{r}}_{j,p,k+1,l}
-
\mathfrak{W}^{H,\mathbf{r}}_{j,p,k,l}
\big )^4 \Big] \leq 
\begin{cases}
C 2^{-4jH} & \text{ if } b=1,
\\
C 2^{-4(j+p)Hb}
 & \text{ otherwise}.
\end{cases}
\end{align*}
\end{lem}

\begin{proof}
We remark that since  there are only finitely many indices $s$, $b$ and $\mathbf{r}$ satisfying the conditions of Lemma \ref{lem:m4_incW}, we can suppose that these index are fixed and show the result for some constant $C$ that is uniform in $j$, $p$ and $k$. Suppose first that $b=1$. Then by Jensen's inequality,
\begin{align*}
\mathbb{E}
\Big[\big(
\mathfrak{W}^{H,\mathbf{r}}_{j,p,k+1,l}
-
\mathfrak{W}^{H,\mathbf{r}}_{j,p,k,l}
\big)^4\Big]
&=
\mathbb{E}
\Big[\Big(\int_{0}^{1} W^H_{(k+1)2^{-j} + (l+u)2^{-j-p}} - W^H_{k2^{-j} + (l+u)2^{-j-p}} du\Big)^4\Big]
\\
&\leq
\int_{0}^{1} \mathbb{E} \Big[\Big( W^H_{(k+1)2^{-j} + (l+u)2^{-j-p}} - W^H_{k2^{-j} + (l+u)2^{-j-p}}\Big)^4\Big] du
\\&=2^{-4jH}.
\end{align*}

Suppose now that $b>1$. Then we have 
$$\mathbb{E} \Big[\big(
\mathfrak{W}^{H,\mathbf{r}}_{j,p,k+1,l}
-
\mathfrak{W}^{H,\mathbf{r}}_{j,p,k,l}
\big )^4\Big] \leq \mathbb{E} \Big[(
\mathfrak{W}^{H,\mathbf{r}}_{j,p,k+1,l})^4
+
(\mathfrak{W}^{H,\mathbf{r}}_{j,p,k,l})^4\Big].$$
We only show the bound for $\mathbb{E} \big[
(\mathfrak{W}^{H,\mathbf{r}}_{j,p,k,l})^4
\big]$, the other term being similar. By Hölder's and Jensen's inequalities, 
\begin{align*}
\mathbb{E}
\Big[\big(
\mathfrak{W}^{H,\mathbf{r}}_{j,p,k,l}
\big)^4\Big]
&=
\mathbb{E}
\Big[\Big(
\frac{1}{\mathbf{r}!}
\int_{[0,1]^s}
\prod_{i=1}^{s}
(W^H_{k2^{-j} + (l+u_i)2^{-j-p}}-W^H_{k2^{-j}+l2^{-j-p}})^{\mathbf{r}_i}
 du\Big)^4\Big]
\\
&\leq
\frac{1}{\mathbf{r}!^4}
\int_{[0,1]^s}
\mathbb{E}
\Big[
\prod_{i=1}^{s}
(W^H_{k2^{-j} + (l+u_i)2^{-j-p}}-W^H_{k2^{-j}+l2^{-j-p}})^{4\mathbf{r}_i}
\Big]
 du
\\
&\leq
\frac{1}{\mathbf{r}!^4}
\int_{[0,1]^s}
\prod_{i=1}^{s}
\mathbb{E}
\Big[
(W^H_{k2^{-j} + (l+u_i)2^{-j-p}}-W^H_{k2^{-j}+l2^{-j-p}})^{4\mathbf{r}_is}
\Big]^{1/s}
 du
\\
&\leq
C
\int_{[0,1]^s}
\prod_{i=1}^{s}
(u_i2^{-j-p})^{4\mathbf{r}_iH}
 du
\\
&\leq C 2^{-(j+p)4Hb}.
\end{align*}
\end{proof}

\begin{lem} We have
\label{lem:cov_g}
\begin{align*}
\Big |
\Cov( g_{j,k_1,p}^2, g_{j,k_2,p}^2)
\Big |
\leq
C 2^{-j(2+4H)}
\Big ( ( 1 + |k_1 - k_2|)^{4(H-1)}
+
 2^{-2jH} ( 1 + |k_1 - k_2|)^{2(H-1)}\Big ).
\end{align*}
\end{lem}

\begin{proof} By Lemma \ref{lem:m4_g}, we can suppose without loss of generality that $|k_1 - k_2| \geq 4$. By symmetry, we may also assume $k_2 > k_1$. Then recall that we have
\begin{align*}
g_{j,k,p}^2
&= 
2^{-2p-j}
\sum_{l_1, l_2 = 0}^{2^p - 1}
\barsum
\Big[
\eta^{b_1+b_2}
\frac{(-1)^{s_1+s_2}}{s_1s_2}
(
	\mathfrak{W}^{H,\mathbf{r_1}}_{j,p,k+1,l_1}
	-
	\mathfrak{W}^{H,\mathbf{r_1}}_{j,p,k,l_1}
)
(
	\mathfrak{W}^{H,\mathbf{r_2}}_{j,p,k+1,l_2}
	-
	\mathfrak{W}^{H,\mathbf{r_2}}_{j,p,k,l_2}
)
\Big]
\end{align*}
so that we can expand $ \Cov \Big [
	 g_{j,p,k_1}^2,
	 g_{j,p,k_2}^2
\Big ]$ as 
\begin{align*}
2^{-4p-2j}
\sum &
\eta^{b_1+b_2+b_3+b_4}
\frac{(-1)^{s_1+s_2+s_3+s_4}}{s_1s_2s_3s_4} \\
&\times \Cov \Big [
(
	\mathfrak{W}^{H,\mathbf{r_1}}_{j,p,k_1+1,l_1}
	-
	\mathfrak{W}^{H,\mathbf{r_1}}_{j,p,k_1,l_1}
)
(
	\mathfrak{W}^{H,\mathbf{r_2}}_{j,p,k_1+1,l_2}
	-
	\mathfrak{W}^{H,\mathbf{r_2}}_{j,p,k_1,l_2}
),
\\& \;\;\;\;\;\;\;\;
(
	\mathfrak{W}^{H,\mathbf{r_3}}_{j,p,k_2+1,l_3}
	-
	\mathfrak{W}^{H,\mathbf{r_3}}_{j,p,k_2,l_3}
)
(
	\mathfrak{W}^{H,\mathbf{r_4}}_{j,p,k_2+1,l_4}
	-
	\mathfrak{W}^{H,\mathbf{r_4}}_{j,p,k_2,l_4}
)
\Big ],
\end{align*}
where the sum is taken over all indices $l_1,l_2,l_3,l_4$, $b_1,b_2,b_3,b_4$, and multi-indices $\mathbf{r_1}, \mathbf{r_2}, \mathbf{r_3}, \mathbf{r_4}$ such that $\sum_i \mathbf{r_1}_i = b_1$, $ \mathbf{r_2}_i = b_2$, $\sum_i \mathbf{r_3}_i = b_3$, $\sum_i \mathbf{r_4}_i = b_4$. Note that the proof is complete once we show that the covariance appearing inside this sum is dominated uniformly over all indices. We consider separately the cases $b_\cdot = 1$ and $b_\cdot > 1$. For conciseness, we will only deal with the following two cases: $b_1 = b_2 = b_3 = b_4 = 1$ and $b_1, b_2,b_3,b_4 > 1$, the other cases being similar. Let us start with the first case. By definition we have
\begin{align*}
\mathfrak{W}^{H,\mathbf{r}}_{j,p,k+1,l}
	-
\mathfrak{W}^{H,\mathbf{r}}_{j,p,k,l}
=
\int_{0}^{1} W^H_{(k+1)2^{-j} + (l+u)2^{-j-p}} - W^H_{k2^{-j} + (l+u)2^{-j-p}} du.
\end{align*}
By self-similarity and stationarity of the fractional Brownian motion increments, we obtain
\begin{align*}
\Cov \Big [
&(
	\mathfrak{W}^{H,\mathbf{r_1}}_{j,p,k_1+1,l_1}
	-
	\mathfrak{W}^{H,\mathbf{r_1}}_{j,p,k_1,l_1}
)
(
	\mathfrak{W}^{H,\mathbf{r_2}}_{j,p,k_1+1,l_2}
	-
	\mathfrak{W}^{H,\mathbf{r_2}}_{j,p,k_1,l_2}
),
\\&\;\;\;\;\;\;\;\;(
	\mathfrak{W}^{H,\mathbf{r_3}}_{j,p,k_2+1,l_3}
	-
	\mathfrak{W}^{H,\mathbf{r_3}}_{j,p,k_2,l_3}
)
(
	\mathfrak{W}^{H,\mathbf{r_4}}_{j,p,k_2+1,l_4}
	-
	\mathfrak{W}^{H,\mathbf{r_4}}_{j,p,k_2,l_4}
)
\Big ]
\\&= 2^{-4jH}
\int_{[0,1]^4}
\Cov \Big [
\big ( W^H_{1 + (l_1+u_1)2^{-p}} - W^H_{(l_1+u_1)2^{-p}} \big )
\big ( W^H_{1 + (l_2+u_2)2^{-p}} - W^H_{(l_2+u_2)2^{-p}} \big )
,
\\
&\;\;\;\;\;\;\;\;
\big ( W^H_{(\tau+1) + (l_3+u_2)2^{-p}} - W^H_{\tau + (l_3+u_2)2^{-p}} \big )
\big ( W^H_{(\tau+1) + (l_4+u_3)2^{-p}} - W^H_{\tau + (l_4+u_4)2^{-p}} \big )
\Big ]du
\end{align*}
where $\tau = k_2 - k_1$. We then use \eqref{eq:cov_4_gaussian} so that the covariance in the integral reduces to 
\begin{align*}
E(l_1,l_3,u_1,u_3)
E(l_2,l_4,u_2,u_4)
+
E(l_1,l_4,u_1,u_4)
E(l_2,l_3,u_2,u_3),
\end{align*}
where
\begin{align*}
E(l,m,u,v) 
&=
\mathbb{E}
\Big[
	\big ( W^H_{1 + (l+u)2^{-p}} - W^H_{(l+u)2^{-p}} \big )
	\big ( W^H_{(\tau+1) + (m+v)2^{-p}} - W^H_{\tau + (m+v)2^{-p}} \big )
\Big ]
\\
&=
\mathbb{E}
\Big[
	W^H_{1}
	\big ( W^H_{(\tau+1) + (m-l+v-u)2^{-p}} - W^H_{\tau + (m-l+v-u)2^{-p}} \big )
\Big ]
\\
&=
D_H(\tau + (m-l+v-u)2^{-p})
\end{align*}
and $D_H(x) = \frac{1}{2}(|x+1|^{2H} - 2|x|^{2H} + |x-1|^{2H})$. By Taylor's formula
$| D_H(x) | \leq C |x - 1|^{2H-2}
$, provided that $x \geq 1$. For $x\geq 3$, we even have $| D_H(x) | \leq C |x+1|^{2H-2}$, for another constant $C$, independent of $H$. Since $(m-l+v-u)2^{-p} \geq -1$ and since $\tau \geq 4$, we have
\begin{align*}
E(l,m,u,v) \leq C |\tau|^{2H-2}.
\end{align*}
Therefore, 
\begin{align*}
|
E(l_1,l_3,u_1,u_3)
E(l_2,l_4,u_2,u_4)
+
E(l_1,l_4,u_1,u_4)
E(l_2,l_3,u_2,u_3)
|
\leq 
C |\tau|^{4H-4}.
\end{align*}
This yields
\begin{align*}
\Big |
\Cov \Big [
&(
	\mathfrak{W}^{H,\mathbf{r_1}}_{j,p,k_1+1,l_1}
	-
	\mathfrak{W}^{H,\mathbf{r_1}}_{j,p,k_1,l_1}
)
(
	\mathfrak{W}^{H,\mathbf{r_2}}_{j,p,k_1+1,l_2}
	-
	\mathfrak{W}^{H,\mathbf{r_2}}_{j,p,k_1,l_2}
),
\\&
(
	\mathfrak{W}^{H,\mathbf{r_3}}_{j,p,k_2+1,l_3}
	-
	\mathfrak{W}^{H,\mathbf{r_3}}_{j,p,k_2,l_3}
)
(
	\mathfrak{W}^{H,\mathbf{r_4}}_{j,p,k_2+1,l_4}
	-
	\mathfrak{W}^{H,\mathbf{r_4}}_{j,p,k_2,l_4}
)
\Big ]
\Big | \leq C 2^{-4jH} | \tau|^{4H-4}.
\end{align*}

We now suppose that $b_1,b_2,b_3,b_4 \geq 2$. In that case, we expand linearly each difference $\mathfrak{W}^{H,\mathbf{r_\cdot}}_{j,p,k_\cdot+1,l_\cdot}
	-
	\mathfrak{W}^{H,\mathbf{r_\cdot}}_{j,p,k_\cdot,l_\cdot}$ in the covariance. This leaves us with $16$ covariances of the form
\begin{align*}
\Cov \Big [
&
	\mathfrak{W}^{H,\mathbf{r_1}}_{j,p,k_1',l_1}
	\mathfrak{W}^{H,\mathbf{r_2}}_{j,p,k_1'',l_2}
,
	\mathfrak{W}^{H,\mathbf{r_3}}_{j,p,k_2',l_3}
	\mathfrak{W}^{H,\mathbf{r_4}}_{j,p,k_2'',l_4}
\Big]
\end{align*}
with $k_q', k_q'' \in \{k_q, k_q+1\}$. By definition, this covariance equals
\begin{align*}
\frac{1}{\mathbf{r_1}!\mathbf{r_2}!\mathbf{r_3}!\mathbf{r_4}!}
\Cov & \Big [
\int_{[0,1]^{s_1}}\prod_{i=1}^{s_1}(W^H_{k_1'2^{-j} + (l_1+u_i)2^{-j-p}}-W^H_{k_1'2^{-j}+l_12^{-j-p}})^{\mathbf{r_1}_i}du
\\&\int_{[0,1]^{s_2}}\prod_{i=1}^{s_2}(W^H_{k_1''2^{-j} + (l_2+u_i)2^{-j-p}}-W^H_{k_1''2^{-j}+l_22^{-j-p}})^{\mathbf{r_2}_i}du
,
\\&\int_{[0,1]^{s_3}}\prod_{i=1}^{s_3}(W^H_{k_2'2^{-j} + (l_3+u_i)2^{-j-p}}-W^H_{k_2'2^{-j}+l_32^{-j-p}})^{\mathbf{r_3}_i}du
\\&\int_{[0,1]^{s_4}}\prod_{i=1}^{s_4}(W^H_{k_2''2^{-j} + (l_4+u_i)2^{-j-p}}-W^H_{k_2''2^{-j}+l_42^{-j-p}})^{\mathbf{r_4}_i}du
\Big].
\end{align*}
Therefore, it is enough to show that the covariance between
\begin{align*}
\prod_{i=1}^{s_1}(W^H_{k_1'2^{-j} + (l_1+u_{1,i})2^{-j-p}}-W^H_{k_1'2^{-j}+l_12^{-j-p}})^{\mathbf{r_1}_i}
\prod_{i=1}^{s_2}(W^H_{k_1''2^{-j} + (l_2+u_{2,i})2^{-j-p}}-W^H_{k_1''2^{-j}+l_22^{-j-p}})^{\mathbf{r_2}_i}
\end{align*}
and
\begin{align*}
\prod_{i=1}^{s_3}(W^H_{k_2'2^{-j} + (l_3+u_{3,i})2^{-j-p}}-W^H_{k_2'2^{-j}+l_32^{-j-p}})^{\mathbf{r_3}_i}
\prod_{i=1}^{s_4}(W^H_{k_2''2^{-j} + (l_4+u_{4,i})2^{-j-p}}-W^H_{k_2''2^{-j}+l_42^{-j-p}})^{\mathbf{r_4}_i}
\end{align*}
is bounded by $C 2^{-4jH} ( 1 + |k_1 - k_2|)^{4(H-1)}$ uniformly for $(u_1, \dots, u_4) \in [0,1]^{s_1+\dots+s_4}$. We aim at applying Proposition \ref{prop:covariance_gaussian_product} to prove this result. Notice that
\begin{align*}
\mathbb{E}\Big[\big(W^H_{k_1'2^{-j} + (l_1+u_{1,i})2^{-j-p}}-W^H_{k_1'2^{-j}+l_12^{-j-p}}\big)^2\Big] \leq 2^{-2(j+p)H}
\end{align*}
and
\begin{align*}
\mathbb{E}&\Big[\big(W^H_{k_1'2^{-j} + (l_1+u_{1,i})2^{-j-p}}-W^H_{k_1'2^{-j}+l_12^{-j-p}}\big)\big(W^H_{k_2'2^{-j} + (l_3+u_{3,i})2^{-j-p}}-W^H_{k_2'2^{-j}+l_32^{-j-p}}\big) \Big]
\\
&= 2^{-2jH-1}(|k_2'-k_1'+2^{-p}(l_3-l_1+u_{3,i})|^{2H} - |k_2'-k_1'+2^{-p}(l_3-l_1)|^{2H}
\\
&\;\;\;\;+ |k_2'-k_1'+2^{-p}(l_3-l_1-u_{1,i})|^{2H} - |k_2'-k_1'+2^{-p}(l_3-l_1+u_{3,i} - u_{1,i})|^{2H}
).
\end{align*}
In addition, we have $k_2'-k_1' \geq 3$ and also $l_3-l_1+u_3, \; l_3-l_1, \; l_3-l_1-u_1$ and $l_3-l_1+u_3 - u_1\geq 2^p$. Thus we can apply Taylor's formula and we expand $|\cdot|^{2H}$ around $k_2'-k_1'+2^{-p}(l_3-l_1)$. The last expression reduces to
\begin{align*}
2^{-2jH-2p}H(2H-1)
\Big(
&u_3^2
\int_0^1
(1-t)
|k_2'-k_1'+2^{-p}(l_3-l_1+tu_3)|^{2H-2}
dt
\\
&+
u_1^2
\int_0^1
(1-t)
|k_2'-k_1'+2^{-p}(l_3-l_1-tu_1)|^{2H-2}
dt
\\
&-
(u_3-u_1)^2
\int_0^1
(1-t)
|k_2'-k_1'+2^{-p}(l_3-l_1+t(u_3-u_1))|^{2H-2}
dt
\Big).
\end{align*}
Since $k_2'-k_1' \geq 3$, the absolute value of this last quantity is bounded by $C 2^{-2jH-2p}|k_2'-k_1'|^{2H-2}$. All other variances and covariances of interest are controlled similarly and we can  apply Proposition \ref{prop:covariance_gaussian_product}. Therefore, the covariance between
\begin{align*}
\prod_{i=1}^{s_1}(W^H_{k_1'2^{-j} + (l_1+u_{1,i})2^{-j-p}}-W^H_{k_1'2^{-j}+l_12^{-j-p}})^{\mathbf{r_1}_i}
\prod_{i=1}^{s_2}(W^H_{k_1''2^{-j} + (l_2+u_{2,i})2^{-j-p}}-W^H_{k_1''2^{-j}+l_22^{-j-p}})^{\mathbf{r_2}_i}
\end{align*} 
and
\begin{align*}
\prod_{i=1}^{s_3}(W^H_{k_2'2^{-j} + (l_3+u_{3,i})2^{-j-p}}-W^H_{k_2'2^{-j}+l_32^{-j-p}})^{\mathbf{r_3}_i}
\prod_{i=1}^{s_4}(W^H_{k_2''2^{-j} + (l_4+u_{4,i})2^{-j-p}}-W^H_{k_2''2^{-j}+l_42^{-j-p}})^{\mathbf{r_4}_i}
\end{align*} 
is bounded by $
C 2^{-2jH - 2p}|k_2-k_1|^{2H-2} 2^{-(b_1+b_2+b_3+b_4)(j+p)H},
$
and this concludes the proof.
\end{proof}

\begin{lem}
\label{lem:m2_d}
For any $j,p \geq 0$ and $k \leq 2^j$, we have
\begin{align*}
\mathbb{E}\big[(d_{j,p,k})^2\big] \leq C2^{-j(2H+1)}.
\end{align*}
\end{lem}
\begin{proof}
Recall that $z_{j,k,p}$ is defined in \eqref{eq:z:def}. We have
\begin{align*}
\EX{z_{j,p,k}^2} \leq 
C
2^{-2p-j}2^{-2(j+p)(2S+1)H^*}
2^{p+1}
\sum_{l=0}^{2^{p+1}-1}
\EX{ Z(k2^{p}+l, 2^{-j-p})^2}
\leq
C 2^{-j-2j(2S+1)H^*}
\end{align*}
by Proposition \ref{prop:development_iv}. Moreover, $\EX{g_{j,p,k}^2} \leq C2^{-j(2H+1)}$ by Lemma \ref{lem:m2_g} so that
\begin{align*}
\mathbb{E}[d_{j,p,k}^2] 
&\leq 2 (\mathbb{E}[g_{j,p,k}^2] + \mathbb{E}[z_{j,p,k}^2])
\\
&\leq C(2^{-j(2H+1)} + 2^{-j-2j(2S+1)H^*}) \leq C2^{-j(2H+1)}
\end{align*}
since $(2S+1)H^* \geq H_+ \geq H$.
\end{proof}

\subsection{Asymptotic behaviour of \texorpdfstring{$Q_{j,p}$}{Q}}
\label{subsec:energy:Q}

We now prove Proposition \ref{prop:energy}. The proof splits in two steps. In a first part, we show a similar result for $G_{j,p}$, namely 
\begin{align*}
\mathbb{E} \Big [\big(
	G_{j,p}
	-
	\sum_{a=1}^{S} \eta^{2a} 2^{-2aHj} \kappa_{p,a}(H)
\big)^2
\Big ]
\leq 
C 2^{-j(1+4H)} + O(2^{-2j(2S+1)H)}).
\end{align*}
In a second part, we show how this result extends to $Q$.\\

First, summing \eqref{eq:dev_g} in $k$ gives
\begin{align*}
\mathbb{E} \Big [
	G_{j,p}
\Big ]
&=
\sum_{a=1}^{S} \eta^{2a} 2^{-2aHj} \kappa_{p,a}(H) + O(2^{-j(2S+1)H)})
\end{align*}
uniformly in $H$, $j$, $k$ and $p$. Moreover, by Lemma \ref{lem:cov_g}
\begin{align*}
\Var \Big [
	G_{j,p}
\Big ]
&=
\Var \Big [
	\sum_k g_{j,p,k}^2 
\Big ]
=\sum_{k_1,k_2} \Cov \Big [
	 g_{j,p,k_1}^2,
	 g_{j,p,k_2}^2
\Big ]
\\
&\leq
C 2^{-j(2+4H)} \sum_{k_1,k_2} \Big [ ( 1 + |k_1 - k_2|)^{4(H-1)} + 2^{-2jH}( 1 + |k_1 - k_2|)^{2(H-1)}
\Big ]
\\
&\leq C 2^{-j(2+4H)} \sum_{|\tau|<2^j} \Big [ (2^j - |\tau|) (( 1 + |\tau|)^{4(H-1)} + 2^{-2jH}( 1 + |\tau|)^{2(H-1)})\Big ]
\\
&\leq C 2^{-j(1+4H)} \sum_{\tau = 1}^{2^j} \Big [ \tau^{4(H-1)} + 2^{-2jH}\tau^{2(H-1)}\Big ]
\\
&\leq C 2^{-j(1+4H)} (1 + 2^{-2jH} 2^{j(2H-1)}) 
\\
&\leq C 2^{-j(1+4H)}
\end{align*}
since $H < 3/4$.
Therefore, a standard bias-variance decomposition ensures  that
\begin{align*}
\mathbb{E} \Big [\big(
	G_{j,p}
	-
	\sum_{a=1}^{S} \eta^{2a} 2^{-2aHj} \kappa_{p,a}(H)
\big)^2
\Big ]
&=
\Var \big [
	G_{j,p}
\big ]
+ 
\big(
\mathbb{E}  \big[
	G_{j,p}
\big]
-
\sum_{a=1}^{S} \eta^{2a} 2^{-2aHj} \kappa_{p,a}(H)
\big)^2
\\
&\leq C 2^{-j(1+4H)} + O(2^{-2j(2S+1)H}).
\end{align*}
Since $2(2S+1)H \geq 1+4H$ (as $S \geq 1/(4H_-) + 1/2$), we obtain
\begin{align*}
\mathbb{E} \Big [\big(
	G_{j,p}
	-
	\sum_{a=1}^{S} \eta^{2a} 2^{-2aHj} \kappa_{p,a}(H)
\big)^2
\Big ]
\leq 
C 2^{-j(1+4H)}.
\end{align*}

We now focus on $\mathbb{E} \big [\big (Q_{j,p} - G_{j,p} \big )^2 \big ]$. By Hölder's and Jensen's inequalities, we have
\begin{align*}
\mathbb{E} \big [\big (Q_{j,p} - G_{j,p} \big )^2 \big ]
&=
\mathbb{E} \Big [\Big (\sum_k z_{j,p,k}(z_{j,p,k} + 2 g_{j,p,k})\Big )^2 \Big ]
\\
&\leq 
2^j
\sum_k
\mathbb{E} \Big [z_{j,p,k}^2 \big(z_{j,p,k} + 2 g_{j,p,k}\big )^2 \Big ]
\\
&\leq 
2^j
\sum_k
\mathbb{E} \big [z_{j,p,k}^4\big]^{1/2} \mathbb{E} \Big [\big(z_{j,p,k} + 2 g_{j,p,k}\big )^4 \Big ]^{1/2}.
\end{align*}

Recall that $z$ is defined in Equation \eqref{eq:z:def}. By Hölder's inequality and by the bound obtained on the variables $Z$ of Proposition \ref{prop:development_iv}, we have
\begin{align*}
\EX{z_{j,p,k}^4} \leq C 2^{-2j} 2^{-4j(2S+1)H^*}.
\end{align*}
Moreover, $\EX{g_{j,p,k}^4} \leq C2^{-j(4H+2)}$ by Lemma \ref{lem:m4_g} so that we get 
\begin{align*}
\mathbb{E} \big [\big (Q_{j,p} - G_{j,p} \big )^2 \big ]
&\leq
C 2^{2j}
\cdot
2^{-j} 2^{-2j(2S+1)H^*}
\cdot
\big(
2^{-2j} 2^{-4j(2S+1)H^*}
+
2^{-j(4H+2)}
\big)^{1/2}.
\end{align*}
Since $(2S+1)H^* \geq H_+ \geq H$ and $2(2S+1)H^* \geq 1+2H_+ \geq 1+2H$, we conclude that
\begin{align*}
\mathbb{E} \big [\big (Q_{j,p} - G_{j,p} \big )^2 \big ]
&\leq
C2^{-2j(2S+1)H^*-2jH} \leq C2^{-j(1+4H)}
\end{align*}
and this proves Proposition \ref{prop:energy}.

\subsection{The behaviour of the function \texorpdfstring{$\kappa_{p,a}$}{kappapa}}
\label{subsec:energy:kappa}

\subsubsection*{Bounds on \texorpdfstring{$\kappa_{p,1}$}{kappap1}}

We first deal with the functions $\kappa_{p,1}$. From \eqref{eq:def_kappas}, we get that 
\begin{align*}
\kappa_{p,1}(H)
&=
2^{-2p}
\sum_{\substack{|l| < 2^p}}
(2^p-|l|)
\mathbb{E} 
\Big [
	\big( \mathfrak{W}^{H,\mathbf{1}}_{p,2^p} - \mathfrak{W}^{H,\mathbf{1}}_{p,0} \big)
	\big( \mathfrak{W}^{H,\mathbf{1}}_{p,2^p + l} - \mathfrak{W}^{H,\mathbf{1}}_{p,l} \big)
\Big ]
\\
&=
2^{-2p}
\sum_{\substack{|l| < 2^p}}
(2^p-|l|)
\int_0^1
\int_0^1
\mathbb{E} 
\Big [
	\big( W^H_{1+u2^{-p}} - W^H_{u2^{-p}} \big)
	\big( W^H_{1+(l+v)2^{-p}} - W^H_{(l+v)2^{-p}} \big)
\Big ]
du dv
\\
&=
2^{-2p}
\sum_{\substack{|l| < 2^p}}
(2^p-|l|)
\int_{-1}^1
(1-|w|)
\varphi_H((l+w)2^{-p})dw
\end{align*}
where $\varphi_H(x)=
\tfrac{1}{2}\big(|x+1|^{2H}-2|x|^{2H}+|x-1|^{2H}\big)$. Notice that $\varphi_H$ is continuous and there exists $C$ independent of $H$ such that $|\varphi_H(x) - \varphi_H(y)| \leq C2^{2H\wedge 1}$ and $|\varphi_H(x)| \leq C$ for any $-1 \leq x,y \leq 1$. We define $\kappa_{\infty,1}(H) = \int_{-1}^{1} (1-|x|) \varphi_H(x) dx $ and we will show that $\kappa_{p,1}$ converges towards $\kappa_{\infty, 1}$ as $p\to \infty$ uniformly on $H$. First, we rewrite $\kappa_{p, 1}$ as
\begin{align*}
\kappa_{p,1}(H)
&=
2^{-p}
\sum_{\substack{|l| < 2^p}}
(1-|l2^{-p}|)
\varphi_H(l2^{-p})
\\&\;\;\;\;+
2^{-2p}
\sum_{\substack{|l| < 2^p}}
(2^p-|l|)
\int_{-1}^1
(1-|w|)
(\varphi_H((l+w)2^{-p}) - \varphi_H(l2^{-p}))dw
\end{align*}
and we study both sums separately. For the first one, we have
\begin{align*}
\Big |
&\kappa_{\infty, 1}(H) -
2^{-p}
\sum_{\substack{|l| < 2^p}}
(1-|l2^{-p}|)
\varphi_H(l2^{-p})
\Big |
\\
&\leq
\Big |
2^{-p}
\sum_{l=-2^p}^{2^p-1}
\int_{l2^{-p}}^{(l+1)2^{-p}}
(1-|x|)\varphi_H(x)
-
(1-|l2^{-p}|)
\varphi_H(l2^{-p})
dx
\Big | \\
&+
\int_{-1}^{-1+2^{-p}} (1-|x|) \varphi_H(x) dx
\\
&\leq 
\sum_{l=-2^p}^{2^p-1}
\int_{l2^{-p}}^{(l+1)2^{-p}}
\Big[
(1-|x|)
\big |
\varphi_H(x)
-
\varphi_H(l2^{-p})
\big |
+
\varphi_H(l2^{-p})
|
(1-|x|)
-
(1-|l2^{-p}|)
|
\Big]
dx\\
&+ C2^{-p}
\\
&\leq 
C
\sum_{l=-2^p}^{2^p-1}
\int_{l2^{-p}}^{(l+1)2^{-p}}
(1-|x|)
|
x
-
l2^{-p}
|^{2H\wedge 1}
+
\varphi_H(l2^{-p})
|
x
-l2^{-p}
|
dx+ C2^{-p}
\\
&\leq 
C
2^{-p(2H\wedge 1)}.
\end{align*}

For the second sum,
\begin{align*}
2^{-2p}
&\Big |
\sum_{\substack{|l| < 2^p}}
(2^p-|l|)
\int_{-1}^1
(1-|w|)
(\varphi_H((l+w)2^{-p}) - \varphi_H(l2^{-p}))dw
\Big |
\\&\leq
2^{-2p}
\sum_{\substack{|l| < 2^p}}
(2^p-|l|)
\int_{-1}^1
(1-|w|)
\Big |\varphi_H((l+w)2^{-p}) - \varphi_H(l2^{-p})
\Big |
dw
\\&\leq
C 2^{-2p}
\sum_{\substack{|l| < 2^p}}
(2^p-|l|)
\int_{-1}^1
(1-|w|)
|w2^{-p}|^{2H\wedge 1}
dw
\\&\leq
C
2^{-p (2H\wedge 1)}.
\end{align*} 

Thus, $\kappa_{p,1}(H) \to \kappa_{p,\infty}(H)$ uniformly. But $\kappa_{p,\infty}$ is continuous by definition and we can easily check that 
$\kappa_{\infty,1}(H) = 
\mathbb{E} \big[\big(
\int_{0}^1 W^H_{u+1}-W^H_udu\big)^2\big] > 0
$ so that $\kappa_{\infty,1}(H)$ is bounded below and above by positive constants. Since each  $\kappa_{p,1}$ is also a positive continuous function, it is also bounded below and above by positive constants. The uniform convergence establishes the bound \eqref{eq:kappa1:bound}.

\subsubsection*{Bounds on $\kappa_{p,a}$}
Suppose $a \geq 2$. By \eqref{eq:def_kappas}, we have
\begin{align*}
| \kappa_{p,a}(H) | 
\leq C \sup_{
	\substack{
		|l| < 2^p ,b_1 + b_2 = 2a\\
		s_1, s_2, \mathbf{r_1}, \mathbf{r_2},\\
		\sum_i \mathbf{r_1}_i = b_1,
		\sum_i \mathbf{r_2}_i = b_2}
} 
\Big | 
\mathbb{E} \Big [
	\big( \mathfrak{W}^{H,\mathbf{r_1}}_{p,2^p} - \mathfrak{W}^{H,\mathbf{r_1}}_{p,0} \big)
	\big( \mathfrak{W}^{H,\mathbf{r_2}}_{p,2^p + l} - \mathfrak{W}^{H,\mathbf{r_2}}_{p,l} \big)
\Big ]
\Big |
\end{align*}
for some constant $C$ independent of $a$, $p$ and $H$. We also have
\begin{align*}
&\Big | 
\mathbb{E} \Big [
	\big( \mathfrak{W}^{H,\mathbf{r_1}}_{p,2^p} - \mathfrak{W}^{H,\mathbf{r_1}}_{p,0} \big)
	\big( \mathfrak{W}^{H,\mathbf{r_2}}_{p,2^p + l} - \mathfrak{W}^{H,\mathbf{r_2}}_{p,l} \big)
\Big ]
\Big | \\
&\leq 
\mathbb{E} \Big [
	\big( \mathfrak{W}^{H,\mathbf{r_1}}_{p,2^p} - \mathfrak{W}^{H,\mathbf{r_1}}_{p,0} \big)^2
\Big ]^{1/2}
\mathbb{E} \Big [
	\big( \mathfrak{W}^{H,\mathbf{r_2}}_{p,2^p + l} - \mathfrak{W}^{H,\mathbf{r_2}}_{p,l} \big)^2
\Big ]^{1/2}
\\&\leq 
C
(2^{-pHb_1} + \mathbbm{1}_{b_1=1})
(2^{-pHb_2} + \mathbbm{1}_{b_2=1})
\end{align*}
by Lemma \ref{lem:m4_incW}. But since $b_1 + b_2 = 2a \geq 4$ and $b_1, b_2 \geq 1$, we obtain \eqref{eq:kappap:bound}.

\subsubsection*{Bounds on $\kappa'_{p,a}$}

Recall that
\begin{align*}
\kappa_{p,a}(H)
=
2^{-2p}
\sum_{\substack{|l| < 2^p ,b_1 + b_2 = 2a,s_1, s_2, \mathbf{r_1}, \mathbf{r_2},\\\sum_i \mathbf{r_1}_i = b_1, \;\sum_i \mathbf{r_2}_i = b_2}}
\frac{(-1)^{s_1+s_2}} {s_1 s_2} (2^p - |l|)
E_{r_1, r_2,p,l}(H)
\end{align*}
where
\begin{align*}
E_{r_1, r_2,p,l}(H) = 
\mathbb{E} \Big [
	\big( \mathfrak{W}^{H,\mathbf{r_1}}_{p,2^p} - \mathfrak{W}^{H,\mathbf{r_1}}_{p,0} \big)
	\big( \mathfrak{W}^{H,\mathbf{r_2}}_{p,2^p + l} - \mathfrak{W}^{H,\mathbf{r_2}}_{p,l} \big)
\Big ].
\end{align*}

Therefore Lemma \ref{lem:bound:kappap} amounts to proving the existence of functions $E_{r_1, r_2,p,l}'$ that are bounded on $[H_-, H_+]$ by a constant independent of $p$ and $l$. (The dependence on $r_1$ and $r_2$ does not matter since there are finitely many possible indices $r_1$ and $r_2$ appearing in the sum.) We consider separately the cases $b_1 = 1$ and $b_1 > 1$ and the cases $b_2 = 1$ and $b_2 > 1$. In the following, we only deal with $b_1 = b_2 = 1$ and $b_1, b_2 > 1$. The two others cases, $b_1 = 1, b_2 > 1$ and $b_1> 1, b_2=1$ can be treated with the same methods.\\

Suppose first that $b_1 = b_2 = 1$. Then $r_1 = r_2 = 1$ and we get
\begin{align*}
E_{r_1, r_2,p,l}(H) 
&= 
\mathbb{E} \Big [
	\int_{0}^{1} \Big(W^H_{1 + u2^{-p}} - W^H_{u2^{-p}} \Big) du
	\int_{0}^{1} \Big(W^H_{1 + (l+v)2^{-p}} - W^H_{(l+v)2^{-p}}\Big) dv
\Big ]
\\
&= 
\int_{0}^{1}
\int_{0}^{1}
\mathbb{E} \Big [
	(W^H_{1 + u2^{-p}} - W^H_{u2^{-p}})
	(W^H_{1 + (l+v)2^{-p}} - W^H_{(l+v)2^{-p}})
\Big ]
du dv
\\
&= 
\frac{1}{2}
\int_{0}^{1}
\int_{0}^{1}
\Big(
	2 \Big|
		\frac{l+v-u}{2^{p}}
	\Big|^{2H}
	-
	\Big|
		\frac{l+v-u}{2^{p}} - 1
	\Big|^{2H}
	-
	\Big|
		\frac{l+v-u}{2^{p}} + 1
	\Big|^{2H}
	\Big)
du dv
\\
&= 
\int_{-1}^{1}
F(w, H) 
dw
\end{align*}
where $F(w, H) = (1-|w|)
	(|
		(l+w)2^{-p}
	|^{2H}
	-
	|
		1 + (l+w)2^{-p}
	|^{2H})$. We now prove that $E_{r_1, r_2,p,l}'(H) = \int_{-1}^{1} \partial_H
F(w, H) 
dw$. Note the following points:
\begin{itemize}
\item For any $H$, $w \mapsto F(w, H)$ is integrable since $| F(w, H) |$ is bounded uniformly on $[-1,1] \times [H_-, H_+]$.
\item For any $w$, $H \mapsto F(w, H)$ is differentiable since $x \mapsto |a|^x$ is always differentiable, no matter whether $a = 0$ (constant case) or $a\neq 0$. In both cases,  the derivative of $x \mapsto |a|^x$ is $x \mapsto  |a|^x \log|a|$, with the convention that $0 \times \log 0 = 0$. Therefore
\begin{align*}
\partial_H
F(w, H) = (1-|w|)
	\Big(\Big|
		\frac{l+w}{2^{p}}
	\Big|^{2H} \log \Big|
		\frac{l+w}{2^{p}}
	\Big|
	-
	\Big|
		1 + \frac{l+w}{2^{p}}
	\Big|^{2H}
	\log \Big|
		1 + \frac{l+w}{2^{p}}
	\Big|\Big).
\end{align*}
\item We have $|
		\frac{l+w}{2^{p}}
	| \leq 2$ and $|
		1 + \frac{l+w}{2^{p}}
	|\leq 2$ so $|
		\frac{l+w}{2^{p}}
	|^{2H} \leq 4$ and $|
		1 + \frac{l+w}{2^{p}}
	|^{2H} \leq 4$. Moreover, $x \mapsto x \log(x)$ is bounded by a constant $C$ on $[0,4]$ so that
\begin{align}
\label{eq:bound:partialH1}
| \partial_H
F(w, H) | \leq 2C (1-|w|) / (2H) \leq C/ H_-,
\end{align}
which is integrable on $[-1,1]$.
\end{itemize}

Therefore we can differentiate under the integral to obtain
\begin{align*}
E_{r_1, r_2,p,l}'(H) 
&=
\int_{-1}^{1}
\partial_H F(w, H) 
dw
\\
&=
\int_{-1}^{1}
(1-|w|)
	\Big(\Big|
		\frac{l+w}{2^{p}}
	\Big|^{2H} \log \Big|
		\frac{l+w}{2^{p}}
	\Big|
	-
	\Big|
		1 + \frac{l+w}{2^{p}}
	\Big|^{2H}
	\log \Big|
		1 + \frac{l+w}{2^{p}}
	\Big|\Big)
dw.
\end{align*}

By \eqref{eq:bound:partialH1}, $|E_{r_1, r_2,p,l}'(H) | \leq 2 C/H_-$ with $C$ that does not depend on $p$ or $l$, and the result follows.\\

We now consider the case $b_1 > 1$ and $b_2 > 1$. We decompose $E_{r_1, r_2,p,l}(H)$ as $E_{r_1, r_2,p,l}^{(1)}(H)-E_{r_1, r_2,p,l}^{(2)}(H)-E_{r_1, r_2,p,l}^{(3)}(H)+E_{r_1, r_2,p,l}^{(4)}(H)$ with
\begin{align*}
E_{r_1, r_2,p,l}^{(1)}(H) &= 
\mathbb{E} \Big [
	\mathfrak{W}^{H,\mathbf{r_1}}_{p,0}
	\mathfrak{W}^{H,\mathbf{r_2}}_{p,l} 
\Big ],
\hspace{1cm}
E_{r_1, r_2,p,l}^{(2)}(H) = 
\mathbb{E} \Big [
	\mathfrak{W}^{H,\mathbf{r_1}}_{p,0}
	\mathfrak{W}^{H,\mathbf{r_2}}_{p,2^p + l} 
\Big ],
\\
E_{r_1, r_2,p,l}^{(3)}(H) &= 
\mathbb{E} \Big [
	\mathfrak{W}^{H,\mathbf{r_1}}_{p,2^p}\mathfrak{W}^{H,\mathbf{r_2}}_{p,l}
\Big ]
,
\hspace{1cm}
E_{r_1, r_2,p,l}^{(4)}(H) = 
\mathbb{E} \Big [
	\mathfrak{W}^{H,\mathbf{r_1}}_{p,2^p}
	\mathfrak{W}^{H,\mathbf{r_2}}_{p,2^p + l}
\Big ].
\end{align*}
We may study each term separately. For simplicity, we only detail the proof for $E_{r_1, r_2,p,l}^{(1)}(H)$. By definition, we have
\begin{align*}
E_{r_1, r_2,p,l}^{(1)}(H)
&= 
\mathbb{E} \Big [
	\mathfrak{W}^{H,\mathbf{r_1}}_{p,0}
	\mathfrak{W}^{H,\mathbf{r_2}}_{p,l} 
\Big ]
\\
&= 
\mathbb{E} \Big [
	\int_{[0,1]^{s_1}}\prod_{i=1}^{s_1} (W^H_{u_i})^{\mathbf{r_1}_i}du
	\int_{[0,1]^{s_2}}\prod_{j=1}^{s_1} (W^H_{(l+v_j)2^{-p}}-W^H_{l2^{-p}})^{\mathbf{r_2}_j}dv
\Big ]
\\
&= 
	\int_{[0,1]^{s_1}} \int_{[0,1]^{s_2}}
	\mathbb{E} \Big [
	\prod_{i=1}^{s_1} (W^H_{u_i})^{\mathbf{r_1}_i}
	\prod_{j=1}^{s_2} (W^H_{(l+v_j)2^{-p}}-W^H_{l2^{-p}})^{\mathbf{r_2}_j}
\Big ]
dudv.
\end{align*}

We aim at applying Theorem \ref{thm:gaussian_moments} to compute $\mathbb{E} \Big [
	\prod_{i=1}^{s_1} (W^H_{u_i})^{\mathbf{r_1}_i}
	\prod_{j=1}^{s_2} (W^H_{(l+v_j)2^{-p}}-W^H_{l2^{-p}})^{\mathbf{r_2}_j}
\Big ]$. Note that this expectation can be rewritten as
$\mathbb{E} \Big [
	\prod_{k=1}^{b_1+b_2} X_k
\Big ]$ with either $X_k = X_k^H(w) = W^H_{w_k}$ if $k \leq b_1$ and $X_k = X_k^H(w) = W^H_{(l+w_k)2^{-p}}-W^H_{l2^{-p}}$ otherwise, where $w$ is the vector obtained by concatenating
\begin{align*}
(\underbrace{ u_1, \cdots, u_1}_{\mathbf{r_1}_1 \text{ times}}, \underbrace{u_2, \cdots, u_2}_{\mathbf{r_1}_2 \text{ times}}, \cdots , \cdots, \underbrace{u_{s_1}, \cdots, u_{s_1}}_{\mathbf{r_1}_{s_1} \text{ times}})
\end{align*}
and the corresponding vector of $v$. Thus Theorem \ref{thm:gaussian_moments} yields
\begin{align*}
\mathbb{E} \Big [
	\prod_{i=1}^{s_1} (W^H_{u_i})^{\mathbf{r_1}_i}
	\prod_{j=1}^{s_2} (W^H_{(l+v_j)2^{-p}}-W^H_{l2^{-p}})^{\mathbf{r_2}_j}
\Big ]
=
\sum_{P} \prod_{(i,j) \in P} \mathbb{E}\big[X_i(w) X_j(w)\big].
\end{align*}

Since there are finitely many $2$-partitions $P$, it is enough to prove that for a given partition of $\{ 1, \cdots, b_1+b_2 \}$, the mapping $\widetilde{E}_P$ defined by
\begin{align*}
\widetilde{E}_P(H)=
\int_{[0,1]^{s_1+s2}}
\prod_{(i,j) \in P} \mathbb{E}\big[X_i^H(w) X_j^H(w)\big]
dw
\end{align*}
is differentiable and its derivative is bounded uniformly over $p$ and $l$. Notice that each term $\mathbb{E}[X_i^H(w) X_j^H(w)]$ is bounded and has a derivative with respect to $H$ uniformly bounded over $w \in [0,1]^{b_1+b_2}$, as readily seen by using explicit computations similar to the case $b_1 = b_2 = 1$. Therefore the mapping $(H,w) \mapsto \prod_{(i,j) \in P} \mathbb{E}[X_i^H(w) X_j^H(w)]$ is differentiable with respect to $H$ and its derivative is given by
\begin{align*}
\prod_{(i_0,j_0) \in P} \partial_H \big(\mathbb{E}\big[X_{i_0}^H(w) X_{j_0}^H(w)\big] \big) \prod_{(i,j) \in P \backslash \{(i_0,j_0)\} } \mathbb{E}\big[X_i^H(w) X_j^H(w)\big],
\end{align*}
which is also uniformly bounded over $w \in [0,1]^{b_1+b_2}$ and we can conclude.

\subsection{Proof of Proposition \ref{prop:bound_energy}}
\label{subsec:energy:bound}

Let $J_0$ and $N$ be two arbitrary integers and let $r>0$. We have
\begin{align*}
&\mathbb{P}_{H,\eta}
\Big(
\inf_{J_0 \leq j \leq N-1} 2^{2jH} Q_{j,N-j-1} \leq r
\Big)
\leq 
\sum_{j=J_0}^{N-1}
\mathbb{P}_{H,\eta}
(
2^{2jH} Q_{j,N-j-1} \leq r
)
\\
&\;\;\leq 
\sum_{j=J_0}^{N-1}
\mathbb{P}_{H,\eta}
\Big(
	Q_{j,N-j-1} - 
	\sum_{a=1}^{S} \eta^{2a} 2^{-2aHj} \kappa_{N-j-1,a}(H)
	\leq \\
&\hspace{3cm}  	r2^{-2jH} - \sum_{a=1}^{S} \eta^{2a} 2^{-2aHj} \kappa_{N-j-1,a}(H) 
\Big).
\end{align*}
Notice then that
\begin{align*}
r2^{-2jH} - \sum_{a=1}^{S} \eta^{2a} 2^{-2aHj} \kappa_{N-j-1,a}(H) 
\leq 
r2^{-2jH} - (\eta_-^2\wedge 1) 2^{-2Hj} c_{-,1} + c_{\cdot, S} S2^{-4Hj}
\end{align*}
by Equations \eqref{eq:kappa1:bound} and \eqref{eq:kappap:bound}. We can take $r < r_0$ small enough and $J_0$ large enough so that $r2^{-2jH} - (\eta_-^2\wedge 1) 2^{-2Hj} c_{-,1} + c_{\cdot, S} 2^{-3H(N-j-1)} < -c_02^{-2jH}$ for some absolute constant $c_0 > 0$. Proposition \ref{prop:energy} and Markov's inequality finally give
\begin{align*}
\mathbb{P}_{H,\eta}
\Big(
\inf_{J_0 \leq j \leq N-1} 2^{2jH} Q_{j,N-j-1} \leq r
\Big)
\leq 
C \sum_{j=J_0}^{N-1} 2^{-j}
\leq C 2^{-J_0} \leq \varepsilon
\end{align*}
provided that $J_0$ is large enough. Similarly, we have
\begin{align*}
r2^{-2jH} - \sum_{a=1}^{S} \eta^{2a} 2^{-2aHj} \kappa_{N-j-1,a}(H) 
\geq 
r2^{-2jH} - (\eta_+^2\vee 1) 2^{-2Hj} c_{+,1} - Sc_{\cdot, S} 2^{-4Hj}
\end{align*}
so that we can conclude with the same argument that 
\begin{align*}
\mathbb{P}_{H,\eta}
\Big(
\sup_{J_0 \leq j \leq N-1} 2^{2jH} Q_{j,N-j-1} \geq r
\Big)
 \leq \varepsilon
\end{align*}
provided that $r$  and $J_0$ are large enough.

\subsection{Proof of Lemma \ref{lem:dist:B}}
\label{sec:proof:lem:dist:B}

By definition,$|
B_{j,p}^{(S)}(\eta_1, H_1) 
-
B_{j,p}^{(S)}(\eta_2, H_2) 
|$ is bounded by
\begin{align*}
\sum_{a=2}^{S} 
\Big(
\eta_1^{2a} 
|
2^{-2aH_1j} \kappa_{p,a}(H_1)
-
2^{-2aH_2j} \kappa_{p,a}(H_2)
|
+
2^{-2aH_2j} \kappa_{p,a}(H_2)
|
\eta_1^{2a} 
-
\eta_2^{2a}
|\Big).
\end{align*}

Moreover, $t \mapsto t^{2a}$ is differentiable with derivative $t \mapsto 2a t^{2a-1}$ which is uniformly bounded on $[\eta_-, \eta_+]$. Similarly, $t \mapsto 2^{-2atj} \kappa_{p,a}(t)$ is differentiable with derivative
\begin{align*}
t \mapsto 2^{-2atj} (-2aj\log(2) \kappa_{p,a}(t) + \kappa_{p,a}'(t)).
\end{align*}
Its absolute value is bounded by $C j 2^{-2atj}$ by Lemma \ref{lem:bound:kappap}. Since $2 \leq a \leq S$, we get
\begin{align*}
|
B_{j,p}^{(S)}(H_1, \eta_1) 
-
B_{j,p}^{(S)}(H_2, \eta_2) 
|
\leq
C
(
j 2^{-4(H_1 \wedge H_2)j} |H_1 - H_2|
+
2^{-4H_2j} 
|
\eta_1 
-
\eta_2
|
).
\end{align*}

\section{Proof of Theorem \ref{thm:construction}}
\label{sec:estimator:proof}

\subsection{Outline and completion of proof}

Our estimator is heavily based on the estimation of the energy levels through the quantities $\widehat{Q}_{j,p}$. The next result gives a bound on the error term $\widehat{Q}_{j,p} - Q_{j,p}$.

\begin{prop} 
\label{prop:estimation_energy} We have
\begin{align*}
\mathbb{E}_{H, \eta} \Big[ \big( \widehat{Q}_{j,p} - Q_{j,p} \big)^2 \Big] \leq C(2^{-2p-j} + 2^{-j(2H+1)-p}).
\end{align*}
\end{prop}

Theorem \ref{thm:construction} is a consequence of the next four propositions.

\begin{prop}
\label{prop:first_estimator_H}
Define $v_n^{(0)} = v_n^{(0)}(H) = n^{-2H} \vee n^{-1/(4H+2)} $. Then 
$\big((v_n^{(0)})^{-1} (\widehat{H}^{(0)}_n - H)\big)_n$ is tight, uniformly over $\mathcal{D}$.
\end{prop}

Recall that this first estimator is used to derive a new adaptive level choice $\widehat{j}_n$. Indeed, we have

\begin{cor}
\label{cor:behaviorjn}
\begin{align*}
\mathbb{P}_{H, \eta}
\Big(
\widehat{j}_n \in \big \{ \big\lfloor \tfrac{1}{2H+1} \log_2 n\big\rfloor - 1, \big\lfloor \tfrac{1}{2H+1} \log_2 n \big\rfloor \big\}
\Big)
\to 1
\end{align*}
uniformly over $\mathcal{D}$.
\end{cor}

\begin{proof}
Since $(v_n^{(0)})^{-1}
|\widehat{H}_n^{(0)} - H|$  is bounded in probability uniformly over $\mathcal{D}$, $(v_n^{(0)})^{-1}
|   1/({2\widehat{H}^{(0)}_n+1}) -   1/(2H+1)|$ is also bounded in probability uniformly over $\mathcal{D}$. Therefore
\begin{align*}
&\mathbb{P}_{H, \eta}
\Big(
\widehat{j}_n \in \big \{ \big\lfloor \tfrac{1}{2H+1} \log_2 n \big\rfloor - 1, \big\lfloor \tfrac{1}{2H+1} \log_2 n  \big\rfloor \big\}
\Big)
\\&\;\;\;\;\;\;\;\;\geq
\mathbb{P}_{H, \eta}
\Big(
\tfrac{1}{2H+1} \log_2 n - 1
<
 \tfrac{1}{2\widehat{H}^{(0)}_n+1}\log_2 n 
<
\tfrac{1}{2H+1} \log_2 n + 1
\Big)
\\&\;\;\;\;\;\;\;\;\geq
\mathbb{P}_{H, \eta}
\Big(
 (v_n^{(0)})^{-1} | \tfrac{1}{2\widehat{H}^{(0)}_n+1} - \tfrac{1}{2H+1} |
<
 (v_n^{(0)} \log_2 n)^{-1}
\Big) \to 1
\end{align*}
since $v_n^{(0)} \log_2 n \to 0$ and the convergence is uniform on $\mathcal{D}$.
\end{proof}

\begin{prop}
\label{prop:first_estimator_eta}
Define $w_n^{(0)} = w_n^{(0)}(H) = v_n^{(0)} \log n $. Then 
$(w_n^{(0)})^{-1}
|\widehat{\eta}_n - \eta|$  is bounded in probability uniformly over $\mathcal{D}$.
\end{prop}

\begin{prop}
\label{prop:refinement_H}
Suppose that we have estimators $\widehat{H}_n$ and $\widehat{\eta}_n$ such that $v_n^{-1} |\widehat{H}_n - H|$ and $w_n^{-1} |\widehat{\eta}_n - \eta|$ are bounded in probability uniformly over $\mathcal{D}$, with $w_n = v_n \log n \to 0$. Suppose also that $\widehat{\eta}_n \in [\eta_-, \eta_+]$ and $\widehat{H}_n \in [H_-, H_+]$. We write $
\widehat{H}_n^{c} = \widehat{H}_n^{c}(\widehat{H}_n,\widehat{\eta}_n)$
as defined in Equation \eqref{eq:def:bias_corrected_estimator_H}. 
Define also $v_n^c = v_n^c(H) = (v_n \log n n^{- 2H/(2H+1)}) \vee n^{-1/(4H+2)}$. Then $(v_n^c)^{-1} |\widehat{H}_n - H|$ is bounded in probability uniformly over $\mathcal{D}$.
\end{prop}

\begin{prop}
\label{prop:refinement_eta}
Suppose that we have estimators $\widehat{H}_n$ and $\widehat{\eta}_n$ such that $v_n^{-1} |\widehat{H}_n - H|$ and $w_n^{-1} |\widehat{\eta}_n - \eta|$ are bounded in probability uniformly over $\mathcal{D}$. Suppose also that $\widehat{\eta}_n \in [\eta_-, \eta_+]$ and $\widehat{H}_n \in [H_-, H_+]$. We write $
\widehat{\eta}_n^{c} = \widehat{\eta}_n^{c}(\widehat{H}_n,\widehat{\eta}_n)$
and 
$ w_n^c = \max \big(
v_n \log n, \;
n^{-1/(4H+2)} \log n, \;
w_n n^{-2H/(2H+1)}
\big)
$. Then $(w_n^c)^{-1} |\widehat{\eta}_n - \eta|$ is bounded in probability uniformly over $\mathcal{D}$.
\end{prop}

We are now ready to conclude the proof of Theorem \ref{thm:construction}. We define by induction the sequences $v_n^{(m)}$ and $w_n^{(m)}$ by
\begin{align*}
v_n^{(0)} = n^{-2H} \vee n^{-1/(4H+2)}
\;
\text{ and }
\;
w_n^{(0)} = v_n^{(0)} \log n,
\end{align*}
exactly as in Propositions \ref{prop:first_estimator_H} and \ref{prop:first_estimator_eta}, and then for $m>0$
\begin{align*}
v_n^{(m)} = (v_n^{(m-1)} \log n n^{- 2H/(2H+1)}) \vee n^{-1/(4H+2)}
\;
\text{ and }
\;
w_n^{(m)} = v_n^{(m)} \log n.
\end{align*}
By induction, we can see that
\begin{align*}
v_n^{(m)} = (\log^m(n) n^{-2H(1+m/(2H+1))}) \vee n^{-1/(4H+2)}
\end{align*}
and provided that $m > 1/(4H) - 2H - 1$, we can see that $v_n^{(m)}=n^{-1/(4H+2)}$ for $n$ large enough.\\

Propositions \ref{prop:first_estimator_H} and \ref{prop:first_estimator_eta} show that $(v_n^{(0)})^{-1} |\widehat{H}^{(0)}_n - H|$ and $(w_n^{(0)})^{-1}
|\widehat{\eta}_n^{(0)} - \eta|$ are bounded in probability uniformly over $\mathcal{D}$, while Propositions \ref{prop:refinement_H} and \ref{prop:refinement_eta} ensure that for any $m>0$, $(v_n^{(m)})^{-1} |\widehat{H}^{(m)}_n - H|$ and $(w_n^{(m)})^{-1}
|\widehat{\eta}_n^{(m)} - \eta|$ are bounded in probability uniformly over $\mathcal{D}$.\\

Since $m_{opt}> m > 1/(4H) - 2H - 1$ for any $H_- < H < H_+$, we can conclude the proof of Theorem \ref{thm:construction}.

\subsection{Proof of Proposition \ref{prop:estimation_energy}}

By definition,
\begin{align*}
\widehat{Q}_{j,p} = Q_{j,p} + 2 \sum_k d_{j,p,k}e_{j,k,p} +  \sum_k \big (e_{j,k,p}^2 -2^{-j-p+1} \Var(\log \xi^2) \big).
\end{align*}

Moreover, the random variables $e_{j,k,p}$ are centered and $e_{j,k_1,p}$ is independent of $e_{j,k_2,p}$ if $|k_1 - k_2| \geq 2$. We deduce that 
\begin{align*}
\mathbb{E}_{H, \eta} \Big[\Big( \sum_k \Big( e_{j,k,p}^2 -2^{-j-p+1} \Var(\log \xi^2)\Big)\Big)^2 \Big] 
&\leq 
C \, \mathbb{E}_{H, \eta} \Big[ \sum_k \Big( e_{j,k,p}^2 -2^{-j-p+1} \Var(\log \xi^2)\Big)^2 \Big] 
\\
&\leq C \sum_k \Big( \mathbb{E}_{H, \eta} [e_{j,k,p}^4] + 2^{-2j-2p+2} \Var(\log \xi^2) \Big).
\end{align*}
But $
e_{j,k,p} = 2^{-p-j/2}
\sum_{l=0}^{2 \cdot 2^p-1}
\chi_{l,p}
\log \Big( 
\xi_{j,p,k2^{p}+l}^2
\Big)
$ where $\chi_{l,p} = -1$ if $l < 2^p$ and $\chi_{l,p} = 1$ otherwise and the $\xi_{j,p,k2^{p}+l}^2$ are independent. Therefore
\begin{align*}
\mathbb{E}_{H, \eta} \big[e_{j,k,p}^4 \big] \leq C 2^{-4p-2j} \mathbb{E}_{H, \eta} \Big[ \Big(\sum_{l=0}^{2 \cdot 2^p-1}
\log \big( 
\xi_{j,p,k2^{p}+l}^2
\big)^2\Big)^2\Big]
\leq  C 2^{-2p-2j}.
\end{align*}
We finally obtain 
\begin{align*}
\mathbb{E}_{H, \eta} \Big[\big( \sum_k e_{j,k,p}^2 -2^{-j-p+1} \Var(\log \xi^2)\big)^2 \Big] 
\leq
2^{-2p-j}.
\end{align*}
We next focus on the term
$\mathbb{E}_{H, \eta}\Big[ \big(\sum_k d_{j,p,k}e_{j,k,p} \big)^2 \Big]$. Note that $(d_{j,k,p})_{j,k}$ and $(e_{j',k',p})_{j',k'}$ are independent. Moreover, the random variables $e_{j,k,p}$ are centered and $e_{j,k_1,p}$ is independent of $e_{j,k_2,p}$ if $|k_1 - k_2| \geq 2$. Therefore, working first conditionally on $(d_{j,k})_{j,k}$, we obtain 
\begin{align*}
\mathbb{E}_{H, \eta}\Big[ \Big(\sum_k d_{j,p,k}e_{j,k,p} \big)^2 \Big]
&\leq C
\sum_k \mathbb{E}_{H, \eta}\big[ d_{j,p,k}^2e_{j,k,p}^2\big]
\\
&\leq C
\sum_k \mathbb{E}_{H, \eta} \big[ d_{j,p,k}^2\big] \mathbb{E}_{H, \eta} \big[ e_{j,k,p}^2\big]
\\
&\leq
C 2^{j} 2^{-j(2H+1)}  2^{-j-p}
\leq 
C 2^{-j(2H+1)- p}
\end{align*}
by Lemma \ref{lem:m2_d}.

\subsection{Proof of Proposition \ref{prop:first_estimator_H}}

Since $H \in [H_-, H_+]$ and $t \mapsto 2^{-2t}$ is invertible on $(0,1)$ with an inverse that is uniformly Lipschitz on the compact sets of $(0,1)$, it is enough to prove that the sequence 
\begin{align*}
\Big((v_n^{(0)})^{-1} \big(\frac{Q_{J_n^* + 1, N-J_n^* - 1, n}}{Q_{J_n^*, N-J_n^* - 1, n}} - 2^{-2H}\big)\Big)_n
\end{align*}
is tight, uniformly over $\mathcal{D}$. We plan to use the following decomposition:
\begin{align*}
\Big | \frac{\widehat{Q}_{J_n^*+1,N-J_n^*-1,n} }{\widehat{Q}_{J_n^*,N-J_n^*-1,n} } - 2^{-2H} \Big |
\leq 
B_n^{(1)} + V_n^{(1)} + V_n^{(2)},
\end{align*}
where
\begin{align*}
B_n^{(1)} &= \Big | \frac{{Q}_{J_n^*+1,N-J_n^*-1} }{{Q}_{J_n^*,N-J_n^*-1} } - 2^{-2H} \Big |,
\\
V_n^{(1)} &= \Big | \frac{\widehat{Q}_{J_n^*+1,N-J_n^*-1,n} - {Q}_{J_n^*+1,N-J_n^*-1} }{\widehat{Q}_{J_n^*,N-J_n^*-1,n} } \Big |,
\\
V_n^{(2)} &= \Big | \frac{{Q}_{J_n^*+1,N-J_n^*-1} 
( \widehat{Q}_{J_n^*,N-J_n^*-1,n}-{Q}_{J_n^*,N-J_n^*-1}  )
}{\widehat{Q}_{J_n^*,N-J_n^*-1,n}{Q}_{J_n^*,N-J_n^*-1} } \Big |.
\end{align*}
It thus suffices to prove that $\big((v_n^{(0)})^{-1} B_n^{(1)}\big)_n$, $\big((v_n^{(0)})^{-1}V_n^{(1)}\big)_n $ and $\big((v_n^{(0)})^{-1} V_n^{(2)}\big)_n$ are tight uniformly over $\mathcal{D}$.

\subsubsection*{The behaviour of $J_n^*$}

For any $\varepsilon > 0$, let 
\begin{align*}
J_n^-(\varepsilon) = \max \big\{j: r_-(\varepsilon)2^{-2jH} \geq 2^jn^{-1} \big\},
\end{align*}
where $r_-(\varepsilon)$ is defined in Proposition \ref{prop:bound_energy}. Notice that $J_n^-(\varepsilon) = \max \big\{ j: r_-(\varepsilon)n \geq 2^{j(2H+1)} \big\}$ so that
\begin{align}
\label{eq:bounds_Jnm}
\tfrac{1}{2} (r_-(\varepsilon)n )^{1/(2H+1)} \leq 2^{J_n^-(\varepsilon)} \leq  (r_-(\varepsilon)n )^{1/(2H+1)}.
\end{align}

We will show that for $\varepsilon > 0$ fixed, there exists $L(\varepsilon) > 0$ and $\varphi_n(\varepsilon) \to 0$ such that
\begin{align}
\label{eq:bound_proba_jnm}
\sup_{H, \eta} \; \mathbb{P}_{H, \eta}(J_n^* < J_n^-(\varepsilon) - L(\varepsilon)) \leq \varepsilon + \varphi_n(\varepsilon).
\end{align}

Let $L=L(\varepsilon)$ to be chosen later. We write $r = r_-(\varepsilon)$ and $p=N-J_n^-(\varepsilon)+L-1$ when the context is clear. We also write $J_n^- = J_n^-(\varepsilon)$ for conciseness. We have
\begin{align*}
\mathbb{P}_{H, \eta}(J_n^* < J_n^- - L) 
&\leq 
\mathbb{P}_{H, \eta}(\widehat{Q}_{J_n^- - L, p, n} < 2^{J_n^- - L} n^{-1})
\\
&\leq 
\mathbb{P}_{H, \eta}(\widehat{Q}_{J_n^- - L, p, n} - Q_{J_n^- - L, p} < 2^{J_n^- - L} n^{-1} - Q_{J_n^- - L, p})
\\
&\leq
\mathbb{P}_{H, \eta}(\widehat{Q}_{J_n^- - L, p, n} - Q_{J_n^- - L} < 2^{J_n^- - L, p} n^{-1} - r2^{-2(J_n^- - L)H})
+ 
\varepsilon
\end{align*}
since $\mathbb{P}_{H, \eta}(\inf_{J_0 \leq j \leq N-1} 2^{2jH} Q_{j,N-j-1} \leq r ) \leq \varepsilon$ by Proposition \ref{prop:bound_energy}. Moreover, note that \eqref{eq:bounds_Jnm} yields
\begin{align*}
2^{J_n^- - L} n^{-1} - r2^{-2(J_n^- - L)H} 
&\leq 
(r_-(\varepsilon)n )^{1/(2H+1)} 2^{-L}n^{-1} - r_-(\varepsilon) 2^{2(L+1)H} (r_-(\varepsilon)n )^{-2H/(2H+1)}
\\
&=
r_-(\varepsilon)^{1/(2H+1)} n^{-2H/(2H+1)} (2^{-L} - 2^{2(L+1)H}).
\end{align*}

For $L$ large enough, $2^{-L} - 2^{2(L+1)H} \leq -1$ so that we get by Proposition \ref{prop:estimation_energy}, Markov's inequality and Equation \eqref{eq:bounds_Jnm}
\begin{align*}
&\mathbb{P}_{H, \eta}(\widehat{Q}_{J_n^- - L, p, n} - Q_{J_n^- - L, p} < 2^{J_n^- - L} n^{-1} - r2^{-2(J_n^- - L)H})
\\
&\;\;\;\;\leq 
C(n^{-2}2^{J_n^-} + 2^{-2HJ_n^-}n^{-1}) 
r_-(\varepsilon)^{-2/(2H+1)} n^{4H/(2H+1)} (2^{-L} - 2^{2(L+1)H})^{-2}
\\
&\;\;\;\;\leq 
C(n^{-2/(2H+1)}2^{J_n^-} + 2^{-2HJ_n^-}n^{(2H-1)/(2H+1)})
\\
&\;\;\;\;\leq 
C(n^{-1/(2H+1)} + n^{-1/(2H+1)}),
\end{align*} 
which proves \eqref{eq:bound_proba_jnm}.

\subsubsection*{The term $B_n^{(1)}$}

By \eqref{eq:bound_proba_jnm} and Proposition \ref{prop:bound_energy}, 
$\mathbb{P}_{H, \eta}\Big({(v_n^{(0)})^{-1} B_n^{(1)} \geq M}\Big)$
is bounded by
\begin{align*}
&\mathbb{P}_{H, \eta}\Big( (v_n^{(0)})^{-1} B_n^{(1)} \geq M, J_n^* \geq J_n^-(\varepsilon) - L(\varepsilon)\Big )
+
\mathbb{P}_{H,\eta}(J_n^* < J_n^-(\varepsilon) - L(\varepsilon) ) 
\\
&\;\;\;\;\leq 
\sum_{j=J_n^-(\varepsilon) - L(\varepsilon)}^{N-1}
\mathbb{P}_{H, \eta}\Big( (v_n^{(0)})^{-1} B_n^{(1)} \geq M, J_n^* = j, {Q}_{j,N-j-1} \geq 2^{-2Hj}r_-(\varepsilon)\Big )
+ 2\varepsilon + \varphi_n(\varepsilon).
\end{align*}

Using the definition of $B$, the probability appearing in the sum is bounded by
\begin{align*}
&
\mathbb{P}\Big(
\Big | {Q}_{j+1,N-j-1}  - 2^{-2H}{Q}_{j,N-j-1} \Big | \geq M 2^{-2Hj} r_-(\varepsilon) v_n^{(0)}
\Big)
\\
&\leq
M^{-2} 2^{4Hj} r_-(\varepsilon)^{-2} (v_n^{(0)})^{-2}
\mathbb{E}
\Big | {Q}_{j+1,N-j-1}  - 2^{-2H}{Q}_{j,N-j-1} \Big |^2 .
\end{align*}
But  $\mathbb{E}_{H, \eta}
\Big[ \big| {Q}_{j+1,N-j-1}  - 2^{-2H}{Q}_{j,N-j-1} \big |^2\Big]$
is bounded by a constant times
\begin{align*}
&
\mathbb{E}_{H, \eta} \Big [\big(
	Q_{j+1,N-j-1}
	-
	\sum_{a=1}^{S} \eta^{2a} 2^{-2aH(j+1)} \kappa_{N-j-1,a}(H)
\big)^2
\Big ]
\\&\;\;\;\;+2^{-4H}
\mathbb{E}_{H, \eta} \Big [\big(
	Q_{j,N-j-1}
	-
	\sum_{a=1}^{S} \eta^{2a} 2^{-2aHj} \kappa_{N-j-1,a}(H)
\big)^2
\Big ]
\\&\;\;\;\;+
\Big(
	\sum_{a=1}^{S} \eta^{2a} 2^{-2aH(j+1)} \kappa_{N-j-1,a}(H)
	-
	2^{-2H}\sum_{a=1}^{S} \eta^{2a} 2^{-2aHj} \kappa_{N-j-1,a}(H)
\Big)^2.
\end{align*}
The two first terms are bounded by $C 2^{-j(1+4H)}$ by Proposition \ref{prop:energy}, and the last term equals
\begin{align*}
\Big( \sum_{a=2}^{S} \eta^{2a} 
	\kappa_{N-j-1,a}(H)
	2^{-2aHj}
	(
		2^{-2aH} -2^{-2H} 
	)
\Big)^2,
\end{align*}
which is bounded by $C 2^{-6HN-2jH}$ by Equation \eqref{eq:kappap:bound}. Thus $\mathbb{P}_{H, \eta}\Big({(v_n^{(0)})^{-1} B_n^{(1)} \geq M}\Big)$ is bounded by 
\begin{align*}
&C M^{-2} r_-(\varepsilon)^{-2} (v_n^{(0)})^{-2}
\sum_{j=J_n^-(\varepsilon) - L(\varepsilon)}^{N-1}
\Big( 2^{4Hj} (2^{-j(1+4H)} + 2^{-6HN-2jH})\Big)
+ 2\varepsilon + \varphi_n(\varepsilon)
\\
&\;\;\;\;\leq 
C(\varepsilon) M^{-2} (v_n^{(0)})^{-2}
\Big(
2^{-J_n^-(\varepsilon)} + 2^{-4HN}
\Big)
+ 2\varepsilon + \varphi_n(\varepsilon),
\end{align*}
where $C(\varepsilon)$ denotes a constant $C$ that may depend on $\varepsilon$. But $2^{-J_n^-(\varepsilon)} \leq C(\varepsilon) n^{-1/(2H+1)}$ so we can conclude since $n^{-1/(2H+1)} (v_n^{(0)})^{-2}$ and $n^{-4H} (v_n^{(0)})^{-2}$ are bounded sequences  by definition of $v_n^{(0)}$.

\subsubsection*{The term $V^{(1)}_n$} We now turn to $V^{(1)}_n$. By definition, recall that
$\widehat{Q}_{J_n^*,N-J_n^*-1,n} \geq  2^{J_n^*}n^{-1}$, at least when $J_n^* \geq J_n^-(\varepsilon) - L(\varepsilon)$. Therefore $\mathbb{P}_{H,\eta}\Big({ (v_n^{(0)})^{-1} V_n^{(1)} \geq M}\Big)$ is bounded by
\begin{align*}
\sum_{j=J_n^-(\varepsilon) - L(\varepsilon) }^{N-1}\mathbb{P}_{H,\eta}\Big({ (v_n^{(0)})^{-1} | \widehat{Q}_{j+1,N-j-1,n} - {Q}_{j+1,N-j-1} | \geq M 2^{j}n^{-1} } \Big) 
+ \varepsilon + \varphi_n(\varepsilon)
\end{align*}
by Equation \eqref{eq:bound_proba_jnm} By Proposition \ref{prop:estimation_energy}, the probability in the sum is bounded from above by a constant times
\begin{align*}
M^{-2} & 2^{-2j} n^2 (v_n^{(0)})^{-2}
\Big (
 n^{-2} 2^{j} + n^{-1}2^{-2jH}
\Big )
\leq
M^{-2} (v_n^{(0)})^{-2}
\Big (
 2^{-j} + n2^{-2j(H+1)}
\Big ).
\end{align*}
Since $L(\varepsilon)$ is fixed, we deduce that
\begin{align*}
\mathbb{P}_{H,\eta}\Big({ (v_n^{(0)})^{-1} V_n^{(1)} \geq M}\Big)
&\leq
M^{-2} (v_n^{(0)})^{-2}
\Big (
 2^{-J_n^-(\varepsilon)} + n2^{-2J_n^-(\varepsilon)(H+1)}
\Big )
+ \varepsilon + \varphi_n(\varepsilon).
\end{align*}
But recall from Equation \eqref{eq:bounds_Jnm} that $2^{-J_n^-(\varepsilon)} \leq C n^{-1/(2H+1)}$ so that 
\begin{align*}
\mathbb{P}_{H,\eta}\Big({(v_n^{(0)})^{-1} V_n^{(1)} \geq M}\Big)
\leq 
C M^{-2}(v_n^{(0)})^{-2}n^{-1/(2H+1)} + \varepsilon + \varphi_n(\varepsilon)
\end{align*}
and we conclude using that $(v_n^{(0)})^{-2}n^{-1/(2H+1)}$ is bounded.

\subsubsection*{The term $V^{(2)}_n$}

By definition, we have
\begin{align*}
V_n^{(2)} &= 
\frac{{Q}_{J_n^*+1,N-J_n^*-1} 
}{{Q}_{J_n^*,N-J_n^*-1} } 
\times
\Big | \frac{
( \widehat{Q}_{J_n^*,N-J_n^*-1,n}-{Q}_{J_n^*,N-J_n^*-1}  )
}{\widehat{Q}_{J_n^*,N-J_n^*-1,n}} \Big |.
\end{align*}

We can show that $(v_n^{(0)})^{-1}\Big | \frac{
( \widehat{Q}_{J_n^*,N-J_n^*-1,n}-{Q}_{J_n^*,N-J_n^*-1}  )
}{\widehat{Q}_{J_n^*,N-J_n^*-1,n}} \Big |$ is bounded in probability uniformly over $\mathcal{D}$ by applying the same proof as for $V^{(1)}_n$. Therefore, it is enough to show that $\frac{{Q}_{J_n^*+1,N-J_n^*-1} 
}{{Q}_{J_n^*,N-J_n^*-1} } $ is bounded in probability uniformly over $\mathcal{D}$. Recall from \eqref{eq:bound_proba_jnm} that $J_n^* \geq J_n^-(\varepsilon)-L(\varepsilon)$ with probability at least $1 - \varepsilon - \varphi_n(\varepsilon)$. Therefore
\begin{align*}
\mathbb{P}_{H,\eta}
\Big(
&\frac{
	Q_{J_n^*+1,N-J_n^*-1} 
}{
	Q_{J_n^*,N-J_n^*-1} 
}  \geq M\Big)
=
\mathbb{P}_{H,\eta}
\Big(
	Q_{J_n^*+1,N-J_n^*-1} 
\geq M Q_{J_n^*,N-J_n^*-1}  \Big)
\\
&\leq
\mathbb{P}_{H,\eta}
\Big(
	Q_{J_n^*+1,N-J_n^*-1} \geq M Q_{J_n^*,N-J_n^*-1} ,\; J_n^* \geq J_n^-(\varepsilon)-L(\varepsilon) \Big) 
+ \varepsilon + \varphi_n(\varepsilon).
\end{align*}
By Proposition \ref{prop:bound_energy}, we have $Q_{j,N-j-1} \geq 2^{-2jH} r_-(\varepsilon)$ for all $J_0(\varepsilon) \leq j \leq N-1$ with probability at least $1-\varepsilon$. As $J_n^-(\varepsilon)-L(\varepsilon) \geq J_0$, we obtain
\begin{align*}
\mathbb{P}_{H,\eta}
\Big(
\frac{
	Q_{J_n^*+1,N-J_n^*-1} 
}{
	Q_{J_n^*,N-J_n^*-1} 
}  \geq M\Big)
&\leq
\mathbb{P}_{H,\eta}
\Big(
	Q_{J_n^*+1,N-J_n^*-1} \geq M 2^{-2J_n^*H} r_-(\varepsilon) ,\; J_n^* \geq J_0 \Big) 
+ 2\varepsilon + \varphi_n(\varepsilon)
\\
&\leq
\mathbb{P}_{H,\eta}
\Big(
	\sup_{J_0 \leq j \leq N-1} 2^{2jH} Q_{j+1,N-j-1} 
\geq M r_-(\varepsilon)  \Big) + 2\varepsilon + \varphi_n(\varepsilon).
\end{align*}
We can conclude since whenever $M r_-(\varepsilon) \geq r_+(\varepsilon)$, we have 
$$\mathbb{P}_{H,\eta}
\Big(
	\sup_{J_0 \leq j \leq N-1} 2^{2jH} Q_{j+1,N-j-1} 
\geq M r_-(\varepsilon)  \Big) \leq \varepsilon$$  by Proposition \ref{prop:bound_energy}.

\subsection{Proof of Proposition \ref{prop:first_estimator_eta}}

Since $\eta \in [\eta_-, \eta_+]$ and $t \mapsto t^2$ is invertible on $(0,\infty)$ with inverse uniformly Lipschitz on the compact sets of $[\epsilon,\infty)$, for any $\epsilon > 0$, it is enough to prove that $(\widehat{\eta}_n^{(0)})^2 - \eta^2$ is bounded in probability uniformly over $\mathcal{D}$. \\

First, notice that 
\begin{align*}
|(\widehat{\eta}_n^{(0)})^2 - \eta^2|
\leq
\kappa_{N-\widehat{j}_n, 1}^{-1}(\widehat{H}_n) \big(B_n^{(1)} + B_n^{(2)} + V_n^{(1)}+V_n^{(2)} \big) 2^{2 \widehat{j}_n(\widehat{H}_n-H)}
\end{align*}
where
\begin{align*}
B_n^{(1)} 
&=
2^{2 \widehat{j}_nH}
\bigg|
\sum_{a=2}^{S} \eta^{2a} 2^{-2aH\widehat{j}_n} \kappa_{N-\widehat{j}_n,a}(H)
\bigg|,
\\
B_n^{(2)} &=
2^{2 \widehat{j}_nH}
\bigg|
{Q}_{\widehat{j}_n,N-\widehat{j}_n} - 
\sum_{a=1}^{S} \eta^{2a} 2^{-2aH\widehat{j}_n} \kappa_{N-\widehat{j}_n,a}(H)
\bigg|,
\\
V_n^{(1)} 
&=
2^{2 \widehat{j}_nH}
\Big|
\eta^{2} 2^{-2H\widehat{j}_n} \kappa_{N-\widehat{j}_n,1}(H)
-
\eta^2 2^{-2\widehat{j}_n\widehat{H}_n} \kappa_{N-\widehat{j}_n, 1}(\widehat{H}_n)
\Big|,
\\
V_n^{(2)} &= 
2^{2 \widehat{j}_nH}
\big| \widehat{Q}_{\widehat{j}_n,N-\widehat{j}_n,n} - {Q}_{\widehat{j}_n,N-\widehat{j}_n}\big|.
\end{align*}

The term $\kappa_{N-\widehat{j}_n, 1}^{-1}(\widehat{H}_n)$ can be ignored because $\kappa_{p,1}$ is a continuous function bounded away from $0$ on $[H_-, H_+]$ uniformly in $p$, see \eqref{eq:kappa1:bound}. We now prove that the sequences $(2^{2 \widehat{j}_n(\widehat{H}_n-H)})_n$, $\big((w_n^{(0)})^{-1} B_n^{(1)}\big)_n$, $\big((w_n^{(0)})^{-1} B_n^{(2)}\big)_n$, $\big((w_n^{(0)})^{-1} V_n^{(1)}\big)_n$ and  $\big((w_n^{(0)})^{-1} V_n^{(2)}\big)_n$ are tight, uniformly over $\mathcal{D}$.

\subsubsection*{The term $2^{2 \widehat{j}_n(\widehat{H}_n-H)}$.} We prove that $2^{2 \widehat{j}_n(\widehat{H}_n-H)}$ is bounded in probability uniformly over $\mathcal{D}$. To that end, it is sufficient to prove the tightness for $\widehat{j}_n(\widehat{H}_n-H)$. But with probability converging to $1$, $\widehat{j}_n \geq \big\lfloor \tfrac{1}{2H+1} \log_2 n\big\rfloor - 1$ by \eqref{eq:bound_proba_jnm} so that
\begin{align*}
\mathbb{P}_{H,\eta}
\big(
\widehat{j}_n|\widehat{H}_n^{(0)}-H| \geq M
\big)
\leq
\mathbb{P}_{H,\eta}
\big(
\big(
\big\lfloor \tfrac{1}{2H+1} \log_2 n\big\rfloor - 1\big)|\widehat{H}_n^{(0)}-H| \geq M
\big) + \varepsilon.
\end{align*}

But $(v_n^{(0)})^{-1} |\widehat{H}_n^{(0)}-H|$ is bounded in probability and $\big\lfloor \tfrac{1}{2H+1} \log_2 n\big\rfloor v_n^{(0)} \to 0$ (and $v_n^{(0)} \to 0$) so we can conclude.

\subsubsection*{The term $B_n^{(1)}$}

By  \eqref{eq:kappap:bound}, we know that 
\begin{align*}
2^{2 \widehat{j}_nH}
|
\sum_{a=2}^{S} \eta^{2a} 2^{-2aH\widehat{j}_n} \kappa_{N-\widehat{j}_n,a}(H)
| \leq C 2^{-2H\widehat{j}_n-3H(N-\widehat{j}_n)}
\end{align*}
so we need to prove that $(w_n^{(0)})^{-1} 2^{H\widehat{j}_n} n^{-3H}$ is uniformly tight. Recall from Corollary \ref{cor:behaviorjn} that
\begin{align*}
\mathbb{P}_{H,\eta}
\Big(
\widehat{j}_n \in \big \{ \big\lfloor \tfrac{1}{2H+1} \log_2 n\big\rfloor - 1, \big\lfloor \tfrac{1}{2H+1} \log_2 n \big\rfloor \big\}
\Big) \to 1\end{align*}
uniformly on $\mathcal{D}$. Moreover 
\begin{align*}
(w_n^{(0)})^{-1} 2^{H\big\lfloor \tfrac{1}{2H+1} \log_2 n \big\rfloor} n^{-3H}
\leq 
C 
(w_n^{(0)})^{-1} 2^{\tfrac{H}{2H+1} \log_2 n} n^{-3H}
=
C
(w_n^{(0)} n^{3H -\tfrac{H}{2H+1} })^{-1},
\end{align*}
which is bounded since $w_n^{(0)} n^{3H -\tfrac{H}{2H+1} } = v_n^{(0)} n^{\tfrac{6H^2 +2H}{2H+1} } \log n$ is bounded from below (since $v_n^{(0)} \geq n^{-2H}$), thus proving that $(w_n^{(0)})^{-1} 2^{H\widehat{j}_n} n^{-3H}$ is uniformly tight.

\subsubsection*{The term $B_n^{(2)}$}

By Corollary \ref{cor:behaviorjn}, $\mathbb{P}_{H,\eta}
\big(
(w_n^{(0)})^{-1}B_n^{(2)} \geq M
\big)$ is equal to
\begin{align*}
&
\mathbb{P}_{H,\eta}
\Big(
\Big |
{Q}_{\widehat{j}_n,N-\widehat{j}_n} - 
\sum_{a=1}^{S} \eta^{2a} 2^{-2aH\widehat{j}_n} \kappa_{N-\widehat{j}_n,a}(H)
\Big | 
\geq M w_n^{(0)} 2^{- 2 \widehat{j}_nH}
\Big)
\\
&\leq
\sum_{
\substack{
j=\big\lfloor \tfrac{1}{2H+1} \log_2 n\big\rfloor - r,\\ r \in \{0, 1 \} 
}
}
\mathbb{P}_{H,\eta}
\big(
\big |
{Q}_{j,N-j} - 
\sum_{a=1}^{S} \eta^{2a} 2^{-2aHj} \kappa_{N-j,a}(H)
\big | 
\geq M w_n^{(0)} 2^{- 2 j H}
\big) 
+o(1)
\end{align*}
uniformly in $\mathcal{D}$.
Moreover, by Proposition \ref{prop:energy}, this last term is further bounded by 
\begin{align*}
C \sum_j
M^{-2} (w_n^{(0)})^{-2} 2^{-j}
+ o(1)
\leq
C
M^{-2} (w_n^{(0)})^{-2} n^{-1/(2H+1)}
+ o(1)
\end{align*}
and we conclude using that $(w_n^{(0)})^{-2} n^{-1/(2H+1)}$ is bounded.

\subsubsection*{The term $V^{(1)}_n$}

For each $j$ and $p$, the mapping 
$t \mapsto 2^{-2tj} \kappa_{p,1}(t) $ is differentiable with derivative 
\begin{align*}
t \mapsto 2^{-2tj} (-2j\log(2)  \kappa_{p,1}(t) + \kappa_{p,1}'(t)).
\end{align*}
By Lemma \ref{lem:bound:kappap}, its absolute value is bounded by $C(j+1)2^{-2tj} $ for some constant $C$. Therefore, Taylor's formula yields that 
\begin{align*}
V_n^{(1)} 
\leq 
C (\widehat{j}_n + 1) |\widehat{H}_n^{(0)} - H| \leq C \log n |\widehat{H}_n^{(0)} - H| .
\end{align*}

Moreover, $(v_n^{(0)})^{-1}|\widehat{H}_n^{(0)} - H|$ is bounded in probability uniformly over $\mathcal{D}$ so $\log n^{-1} (v_n^{(0)})^{-1} V_n^{(1)}$ is also bounded in probability uniformly over $\mathcal{D}$.

\subsubsection*{The term $V^{(2)}_n$}

By Corollary \ref{cor:behaviorjn},
\begin{align*}
&\mathbb{P}_{H,\eta}\Big({ {w_n^{(0)}}^{-1} V_n^{(2)} \geq M}\Big) \\
&\leq
\sum_{
\substack{
j=\big\lfloor \tfrac{1}{2H+1} \log_2 n\big\rfloor - r,\\ r \in \{0, 1 \} 
}
}
\mathbb{P}_{H,\eta}\Big(
| \widehat{Q}_{j,N-j,n} - {Q}_{j,N-j}| \geq M 2^{-2 j H} {w_n^{(0)}} \Big) + o(1)
\end{align*}
uniformly over $\mathcal{D}$. By Proposition \ref{prop:estimation_energy}, the probability in the sum is bounded from above by a constant times 
\begin{align*}
& M^{-2} (w_n^{(0)})^{-2}
 (
 n^{-2} 2^{j(1+4H)} + n^{-1}2^{2jH}
 )  \\
&\leq C M^{-2} (w_n^{(0)})^{-2}
 (
 n^{-2} n^{(1+4H)/(2H+1)} + n^{-1}n^{2H/(2H+1)}
 ) 
\\
&= C M^{-2} (w_n^{(0)})^{-2}
n^{-1/(2H+1)},
\end{align*}
which vanishes as $n$ grows.

\subsection{Proof of Proposition \ref{prop:refinement_H}}

In the same way as for Proposition \ref{prop:first_estimator_H}, we only need to show that 
\begin{align*}
(v_n^c(H))^{-1} \Big |
 \frac{\widehat{Q}^{(S)}_{J_n^{*c} + 1, N-J_n^{*c} - 1, n}(\widehat{H}_n, \widehat{\eta}_n)}{\widehat{Q}^{(S)}_{J_n^{*c}, N-J_n^{*c} - 1, n}(\widehat{H}_n, \widehat{\eta}_n)}
- 2^{-2H} \Big |
\end{align*}
 is bounded in $\mathbb{P}_{H,\eta}$-probability uniformly over $\mathcal{D}$ where we write $J_n^{*c} = J_n^{*c}(\widehat{H}_n, \widehat{\eta}_n) $ for simplicity. We will use the same strategy of proof as in Proposition \ref{prop:first_estimator_H} although additional care needs to be taken because of the use of the first-hand estimator $\widehat{H}_n$. Note that the following decomposition holds:
\begin{align*}
 \Big |
 \frac{\widehat{Q}^{(S)}_{J_n^{*c} + 1, N-J_n^{*c} - 1, n}(\widehat{H}_n, \widehat{\eta}_n)}{\widehat{Q}^{(S)}_{J_n^{*c}, N-J_n^{*c} - 1, n}(\widehat{H}_n, \widehat{\eta}_n)}
- 2^{-2H} \Big |
\leq 
B_n + V_n^{(1)} + V_n^{(2)},
\end{align*}
where
\begin{align*}
B_n &= \Big | \frac{{Q}_{J_n^{*c}+1,N-J_n^{*c}-1}^{(S)}(H,\eta) }{{Q}_{J_n^{*c},N-J_n^{*c}-1}^{(S)}(H,\eta) } - 2^{-2H} \Big |,
\\
V_n^{(1)} &= \Big | \frac{\widehat{Q}_{J_n^{*c}+1,N-J_n^{*c}-1,n}^{(S)}(\widehat{H}_n, \widehat{\eta}_n) - {Q}_{J_n^{*c}+1,N-J_n^{*c}-1}^{(S)}(H,\eta) }{\widehat{Q}_{J_n^{*c},N-J_n^{*c}-1,n}^{(S)}(\widehat{H}_n, \widehat{\eta}_n) } \Big |,
\\
V_n^{(2)} &= \Big | \frac{{Q}_{J_n^{*c}+1,N-J_n^{*c}-1}^{(S)}(H,\eta)
( \widehat{Q}_{J_n^{*c},N-J_n^{*c}-1,n}^{(S)}(\widehat{H}_n, \widehat{\eta}_n)-{Q}_{J_n^{*c},N-J_n^{*c}-1}^{(S)}(H,\eta)  )
}{\widehat{Q}_{J_n^{*c},N-J_n^{*c}-1,n}^{(S)}(\widehat{H}_n, \widehat{\eta}_n){Q}_{J_n^{*c},N-J_n^{*c}-1}^{(S)}(H,\eta) } \Big |.
\end{align*}
We now prove that $\big((v_n^{c})^{-1} B_n\big)_n$, $\big((v_n^{c})^{-1} V_n^{(1)} \big)_n$ and $\big((v_n^{c})^{-1} V_n^{(2)}\big)_n$ are bounded in probability uniformly over $\mathcal{D}$.

\subsubsection*{The behaviour of $J_n^{*c}$}

We fix $\varepsilon >0$ and we define
\begin{align*}
J_n^{-c}(\varepsilon) = \max \big\{j: r_-^{(S)}(\varepsilon)2^{-2jH} \geq 2^jn^{-1} \big\}
\end{align*}
where $r_-^{(S)}(\varepsilon)$ is defined in Proposition \ref{prop:bound_energy_corrected}. Notice that $J_n^{-c}(\varepsilon)$ is independent of $\eta$ and $H$ since in Proposition \ref{prop:bound_energy_corrected}, $r_-^{(S)}(\varepsilon)$ is defined uniformly for all $\eta$ and $H$. Similarly to Equation \eqref{eq:bounds_Jnm},
\begin{align}
\label{eq:bounds_Jnm_corrected}
\tfrac{1}{2} (r_-^{(S)}(\varepsilon)n )^{1/(2H+1)} \leq 2^{J_n^{-c}(\varepsilon)} \leq  (r_-^{(S)}(\varepsilon)n )^{1/(2H+1)}.
\end{align}
We show that there exist $L^c(\varepsilon) > 0$ and $\varphi_n^c(\varepsilon) \to 0$ such that 
\begin{align}
\label{eq:bound_proba_jnm_corrected}
\sup_{H, \eta} \; \mathbb{P}_{H,\eta}(J_n^{*c} < J_n^{-c}(\varepsilon) - L^c(\varepsilon)) \leq \varepsilon + \varphi_n^c(\varepsilon).
\end{align}

Let $L$ to be chosen later. We write $r = r_-^{(S)}(\varepsilon)$ and $p=N-j-1$ when the context is clear. We also write $J_n^{-c} = J_n^{-c}(\varepsilon)$ for conciseness. By definition, $\mathbb{P}_{H,\eta}(J_n^{*c} < J_n^{-c} - L) $ is bounded by
\begin{align*}
& \mathbb{P}_{H,\eta}(\widehat{Q}_{J_n^{-c} - L, N-J_n^{-c}+ L - 1, n}^{(S)}(\widehat{H}_n, \widehat{\eta}_n)
 - Q_{J_n^{-c} - L, N-J_n^{-c}+ L - 1}^{(S)}(H,\eta) \\
& < \frac{2^{J_n^{-c} - L}}{n} - Q_{J_n^{-c} - L, N-J_n^{-c}+ L - 1}^{(S)}(H,\eta)
).
\end{align*}
But $\mathbb{P}_{H,\eta}(\inf_{J_0 \leq j \leq N-1} 2^{2jH} Q_{j,N-j-1}^{(S)}(H,\eta)
 \leq r ) \leq \varepsilon$ by Proposition \ref{prop:bound_energy_corrected}, so it is also bounded by
\begin{align*}
& \mathbb{P}_{H,\eta}(\widehat{Q}_{J_n^{-c} - L, N-J_n^{-c}+ L, n}^{(S)}(\widehat{H}_n, \widehat{\eta}_n)
 - Q_{J_n^{-c} - L, N-J_n^{-c}+ L - 1}^{(S)}(H,\eta) \\
 &
 < 2^{J_n^{-c} - L} /n - r2^{-2(J_n^{-c} - L)H})
+ 
\varepsilon.
\end{align*}
Moreover, note that \eqref{eq:bounds_Jnm_corrected} yields
\begin{align*}
&2^{J_n^{-c} - L} n^{-1} - r2^{-2(J_n^{-c} - L)H} \\
&\leq 
(r_-^{(S)}(\varepsilon)n )^{1/(2H+1)} 2^{-L}n^{-1} - r_-^{(S)}(\varepsilon) 2^{2(L+1)H} (r_-^{(S)}(\varepsilon)n )^{-2H/(2H+1)}
\\
&=
r_-^{(S)}(\varepsilon)^{1/(2H+1)} n^{-2H/(2H+1)} (2^{-L} - 2^{2(L+1)H}).
\end{align*}
For $L$ large enough, $2^{-L} - 2^{2(L+1)H} \leq -1$ so that we only need to prove that both
\begin{align*}
&\mathbb{P}_{H,\eta}(
	\widehat{Q}_{J_n^{-c} - L, N-J_n^{-c}+ L - 1, n}
 - 
 	\widehat{Q}_{J_n^{-c} - L, N-J_n^{-c}+ L - 1, n}
 < 
 -
 r_-^{(S)}(\varepsilon)^{1/(2H+1)} n^{-2H/(2H+1)}
)
\end{align*}
and
\begin{align*}
&\mathbb{P}_{H,\eta}(
	B^{(S)}_{J_n^{-c} - L, N-J_n^{-c}+ L - 1}(\widehat{H}_n, \widehat{\eta}_n)
 	- 
 	B^{(S)}_{J_n^{-c} - L, N-J_n^{-c}+ L - 1}(H,\eta)
 < \\
 &\hspace{6cm}
 -
 r_-^{(S)}(\varepsilon)^{1/(2H+1)} n^{-2H/(2H+1)}
)
\end{align*}
converge to $0$ uniformly on $\mathcal{D}$. The first convergence is proven as in the preliminary of the proof of Proposition \ref{prop:first_estimator_H}. We deal with the second convergence using Lemma \ref{lem:dist:B}. The probability considered is bounded by
\begin{align*}
&\mathbb{P}_{H,\eta}(
c_B 2^{-4(\widehat{H} \wedge H)(J_n^{-c} - L)}
(
(J_n^{-c} - L) |\widehat{H}_n - H|
+
|
\widehat{\eta}_n
-
\eta
|
)
 \geq 
 r_-^{(S)}(\varepsilon)^{1/(2H+1)} n^{-2H/(2H+1)}
)
\\
&\leq
\mathbb{P}_{H,\eta}(
 2^{-4(\widehat{H} \wedge H)J_n^{-c}}
(
\log n |\widehat{H}_n - H|
+
|
\widehat{\eta}_n
-
\eta
|
)
 \geq 
  n^{-2H/(2H+1)} \widetilde{c_B}^{-1}
)
\end{align*}
where $\widetilde{c_B} = r_-^{(S)}(\varepsilon)^{-1/(2H+1)} c_B 2^{4 H_+ L}$. But for any $\widetilde{\varepsilon} > 0$, we have $\mathbb{P}_{H,\eta}(v_n^{-1}|\widehat{H}_n - H| \geq \widetilde{M}) \leq \widetilde{\varepsilon}$ for $\widetilde{M}$ large enough, so the last term is bounded by
\begin{align*}
\mathbb{P}_{H,\eta}(
 2^{-4(H + \widetilde{M}v_n)J_n^{-c}}
(
\log n |\widehat{H}_n - H|
+
|
\widehat{\eta}_n
-
\eta
|
)
 \geq 
  n^{-2H/(2H+1)} \widetilde{c_B}^{-1}
) + \widetilde{\varepsilon}.
\end{align*}
We conclude using firstly that $v_n J_n^{-c} \leq v_n \log n \to 0$ (hence   $2^{-4HJ_n^{-c}}$ is of the same order as $n^{-4H/(2H+1)}$) and secondly that $n^{-2H/(2H+1)} \log n  |\widehat{H}_n - H|$ and $n^{-2H/(2H+1)} |
\widehat{\eta}_n
-
\eta
|$ converge to $0$.

\subsubsection*{The term $B_n$}

We rewrite $B_n$ as
\begin{align*}
\frac{
\Big | 
	{Q}_{J_n^{*c}+1,N-J_n^{*c}-1} 
	- 
	2^{-2H}{Q}_{J_n^{*c},N-J_n^{*c}-1} 
	-
	{B}_{J_n^{*c}+1,N-J_n^{*c}-1}^{(S)}(H,\eta) 
	+ 
	2^{-2H}B_{J_n^{*c},N-J_n^{*c}-1}^{(S)}(H,\eta)
\Big |
}{{Q}_{J_n^{*c},N-J_n^{*c}-1}^{(S)}(H,\eta) }.
\end{align*}
Note that by definition of $B^{(S)}_{j,p}$, we have
\begin{align*}
2^{-2H} &B_{J_n^{*c},N-J_n^{*c}-1}^{(S)}(H,\eta)
-
{B}_{J_n^{*c}+1,N-J_n^{*c}-1}^{(S)}(H,\eta) 
\\
&=
 2^{-2H} \sum_{a=2}^{S} \eta^{2a} 2^{-2aHJ_n^{*c}} \kappa_{N-J_n^{*c}-1,a}(H)
-
\sum_{a=2}^{S} \eta^{2a} 2^{-2aH(J_n^{*c}+1)} \kappa_{N-J_n^{*c}-1,a}(H)
\\
&=
 2^{-2H} \sum_{a=1}^{S} \eta^{2a} 2^{-2aHJ_n^{*c}} \kappa_{N-J_n^{*c}-1,a}(H)
-
\sum_{a=1}^{S} \eta^{2a} 2^{-2aH(J_n^{*c}+1)} \kappa_{N-J_n^{*c}-1,a}(H)
\end{align*}
since the terms of the sums with $a=1$ are the same. Thus $B_n \leq B_n^{(1)} + 2^{-2H} B_n^{(2)}$ where
\begin{align*}
&B_n^{(1)} := \frac{
\Big | 
	{Q}_{J_n^{*c}+1,N-J_n^{*c}-1} 
	- 
	\sum_{a=1}^{S} \eta^{2a} 2^{-2aH(J_n^{*c}+1)} \kappa_{N-J_n^{*c}-1,a}(H)
\Big |
}{{Q}_{J_n^{*c},N-J_n^{*c}-1}^{(S)}(H,\eta) }
\end{align*}
and
\begin{align*}
&B_n^{(2)} :=
\frac{
\Big | 
	{Q}_{J_n^{*c},N-J_n^{*c}-1} 
	- 
	\sum_{a=1}^{S} \eta^{2a} 2^{-2aHJ_n^{*c}} \kappa_{N-J_n^{*c}-1,a}(H)
\Big |
}{{Q}_{J_n^{*c},N-J_n^{*c}-1}^{(S)}(H,\eta) }.
\end{align*}
Both terms are controlled identically so we only prove here that $\big((v_n^{c})^{-1} B_n^{(1)}\big)_n$ is bounded in probability uniformly over $\mathcal{D}$. 
By \eqref{eq:bound_proba_jnm_corrected} and Proposition \ref{prop:bound_energy_corrected}, $\mathbb{P}_{H,\eta}\big({(v_n^{c})^{-1} B_n^{(1)} \geq M}\big)
$ is bounded by
\begin{align*}
&
\mathbb{P}_{H,\eta}\Big( (v_n^{c})^{-1} B_n^{(1)} \geq M, J_n^{*c} \geq J_n^{-c}(\varepsilon) - L^c(\varepsilon)\Big )
+
\mathbb{P}_{H,\eta}^{n}(J_n^{*c} < J_n^{-c}(\varepsilon) - L^c(\varepsilon) ) 
\\
&\;\;\;\;\leq 
\sum_{j=J_n^{-c}(\varepsilon) - L^c(\varepsilon)}^{N-1}
\mathbb{P}_{H,\eta}\Big( (v_n^{c})^{-1} B_n^{(1)} \geq M, J_n^* = j, {Q}_{j,N-j-1}^{(S)}(H,\eta) \geq 2^{-2Hj}r_-^{(S)}(\varepsilon)\Big )
+ 2\varepsilon + \varphi_n^c(\varepsilon).
\end{align*}
Using the definition of $B^{(1)}_n$, each term within the sum is bounded by
\begin{align*}
\mathbb{P}_{H,\eta}\Big(
\Big | 
	{Q}_{j+1,N-j-1} 
	- 
	\sum_{a=1}^{S} \Big( \eta^{2a} 2^{-2aH(j+1)} \kappa_{N-j-1,a}(H) \Big)
\Big |
 \geq M 2^{-2Hj} r_-^{(S)}(\varepsilon) v_n^{c}
\Big).
\end{align*}
By Markov's inequality and Proposition \ref{prop:energy}, it is further bounded by $
C M^{-2} 2^{-j} r_-^{(S)}(\varepsilon)^{-2} (v_n^{c})^{-2}$ and
summing over $j$ yields
\begin{align*}
\mathbb{P}_{H,\eta}\big({(v_n^{c})^{-1} B_n^{(1)} \geq M}\big)
&\leq 
C M^{-2} 2^{-J_n^{-c}(\varepsilon) + L^c(\varepsilon)} r_-^{(S)}(\varepsilon)^{-2} (v_n^{c})^{-2}
+ 2\varepsilon + \varphi_n^c(\varepsilon).
\end{align*}
We conclude using that $(v_n^{c})^{-2} 2^{-J_n^{-c}(\varepsilon)}$ is bounded, see \eqref{eq:bounds_Jnm_corrected}.

\subsubsection*{The term $V_n^{(1)}$} We have
\begin{align*}
V^{(1)}_n
=&
\Big | \frac{\widehat{Q}_{J_n^{*c}+1,N-J_n^{*c}-1,n}^{(S)}(\widehat{H}_n, \widehat{\eta}_n) - {Q}_{J_n^{*c}+1,N-J_n^{*c}-1}^{(S)}(H,\eta) }{\widehat{Q}_{J_n^{*c},N-J_n^{*c}-1,n}^{(S)}(\widehat{H}_n, \widehat{\eta}_n) } \Big |
\\
\leq&
 \frac{
| 
		\widehat{Q}_{J_n^{*c}+1,N-J_n^{*c}-1,n}(\widehat{H}_n, \widehat{\eta}_n)
		- 
		{Q}_{J_n^{*c}+1,N-J_n^{*c}-1} (H,\eta) 
|
}{\widehat{Q}_{J_n^{*c},N-J_n^{*c}-1,n}^{(S)}(\widehat{H}_n, \widehat{\eta}_n) }
\\&\;\;\;\;
 + \frac{
| 
		B_{J_n^{*c}+1,N-J_n^{*c}-1}^{(S)}(\widehat{H}_n, \widehat{\eta}_n)
		- 
		B_{J_n^{*c}+1,N-J_n^{*c}-1}^{(S)}(H,\eta) 
|
}{\widehat{Q}_{J_n^{*c},N-J_n^{*c}-1,n}^{(S)}(\widehat{H}_n, \widehat{\eta}_n) } .
\end{align*}
Moreover, $\widehat{Q}_{J_n^{*c},N-J_n^{*c}-1,n}^{(S)}(\widehat{H}_n, \widehat{\eta}_n) \geq  2^{J_n^{*c}}n^{-1} $ and $J_n^{*c} \geq J_n^{-c}(\varepsilon) - L^c(\varepsilon)$ at least with probability $1 - \varepsilon -\varphi^c_n(\varepsilon)$ by Equation \eqref{eq:bound_proba_jnm_corrected}, so it is enough to prove that 
\begin{align}
\label{eq:ref:vn1:term1}
n (v_n^{c})^{-1} 2^{-J_n^{*c}} | 
		\widehat{Q}_{J_n^{*c}+1,N-J_n^{*c}-1,n}
		- 
		{Q}_{J_n^{*c}+1,N-J_n^{*c}-1,n}
|
\end{align} 
and 
\begin{align}
\label{eq:ref:vn1:term2}
n (v_n^{c})^{-1} 2^{-J_n^{*c}} | 
		B_{J_n^{*c}+1,N-J_n^*-1}(\widehat{H}_n, \widehat{\eta}_n)
		- 
		B_{J_n^{*c}+1,N-J_n^*-1}(H,\eta) 
|\end{align} are bounded in probability uniformly over $\mathcal{D}$, conditionally on $J_n^{*c} \geq J_n^{-c}(\varepsilon) - L^c(\varepsilon)$. The term \eqref{eq:ref:vn1:term1} is similar to $V_n^{(1)}$ appearing in the proof of Proposition \ref{prop:first_estimator_H}. Indeed, by Proposition \ref{prop:estimation_energy}, we have
\begin{align*}
&\mathbb{P}_{H,\eta}
\Big(
n (v_n^{c})^{-1} 2^{-J_n^{*c}} | 
		\widehat{Q}_{J_n^{*c}+1,N-J_n^*-1,n}
		- 
		{Q}_{J_n^{*c}+1,N-J_n^*-1,n} |
\geq M
,\;
J_n^{*c} \geq J_n^{-c}(\varepsilon) - L^c(\varepsilon)
\Big)
\\
&\;\;\;\;\leq
\sum_{j=J_n^{-c}(\varepsilon) - L^c(\varepsilon)}^{N-1}
\mathbb{P}_{H,\eta}
\Big(
 | 
		\widehat{Q}_{j+1,N-j-1,n}
		- 
		{Q}_{j+1,N-j-1,n} |
\geq M n^{-1} v_n^{c} 2^{j}
\Big)
\\
&\;\;\;\;\leq
C
\sum_{j=J_n^{-c}(\varepsilon) - L^c(\varepsilon)}^{N-1}
(
2^{j}n^{-2}
+
2^{-2Hj}n^{-1}
)
n^{2} (M v_n^{c})^{-2} 2^{-2j} 
\\
&\;\;\;\;\leq
C(\varepsilon) M^{-2} (v_n^{c})^{-2} (2^{-J_n^{-c}(\varepsilon)} + n 2^{-2(H+1)J_n^{-c}(\varepsilon)})
\end{align*}
and we conclude using \eqref{eq:bounds_Jnm_corrected} and the definition of $v_n^{c}$. We next focus on the term \eqref{eq:ref:vn1:term2}. By Lemma \ref{lem:dist:B} and using $J_n^{*c} \leq \log n$, it is bounded by
\begin{align*}
c_B
	2^{-4((\widehat{H}_n - H) \wedge 0)\log n}
	\times
n (v_n^{c})^{-1}
 2^{-(4H+1)J_n^{*c}}
(
\log n |\widehat{H}_n - H|
+
|
\widehat{\eta}_n
-
\eta
|
).
\end{align*}
First, notice that $2^{-4((\widehat{H}_n - H) \wedge 0)\log n}$ is bounded in probability uniformly over $\mathcal{D}$ because $v_n^{-1}|\widehat{H}_n - H|$ is uniformly tight and $v_n \log n \to 0$. Therefore we can focus on $n (v_n^{c})^{-1}
 2^{-(4H+1)J_n^{*c}}
(
\log n |\widehat{H}_n - H|
+
|
\widehat{\eta}_n
-
\eta
|
)$. Conditionally on $J_n^{*c} \geq J_n^{-c}(\varepsilon) - L^c(\varepsilon)$, it is bounded by
\begin{align*}
C
n v_n  \log n (v_n^{c})^{-1}
 2^{-(4H+1)J_n^{-c}(\varepsilon)}
(
 {v_n}^{-1} |\widehat{H}_n - H|
+
(v_n \log n)^{-1}
|
\widehat{\eta}_n
-
\eta
|
)
\end{align*}
and we conclude using \eqref{eq:bounds_Jnm_corrected}, the assumptions on $v_n$, $\widehat{H}_n$ and $\widehat{\eta}_n$ together with the definition of $v_n^{c}$.

\subsubsection*{The term $V^{(2)}_n$}

The same method as for the corresponding term in the proof of Proposition \ref{prop:first_estimator_H} applies, relying now on Proposition \ref{prop:bound_energy_corrected} and \eqref{eq:bound_proba_jnm_corrected} instead of Proposition \ref{prop:bound_energy} and  \eqref{eq:bound_proba_jnm}.

\subsection{Proof of Proposition \ref{prop:refinement_eta}}

As for Proposition \ref{prop:first_estimator_eta}, it is enough to prove that $(w_n^{c})^{-1} |
({\widehat{\eta}^{c}}_n)^2
 - \eta^2|$ is bounded in probability uniformly over $\mathcal{D}$. First, notice that 
\begin{align*}
|({\widehat{\eta}^{c}}_n)^2 - \eta^2|
\leq
\kappa_{N-\widehat{j}_n, 1}^{-1}(\widehat{H}_n) \big(B_n^{(1)} + B_n^{(2)} + V_n^{(1)}+V_n^{(2)} \big) 2^{2 \widehat{j}_n(\widehat{H}_n-H)}
\end{align*}
where
\begin{align*}
B_n^{(1)} 
&=
2^{2 \widehat{j}_nH}
\bigg|
\sum_{a=2}^{S} \eta^{2a} 2^{-2aH\widehat{j}_n} \kappa_{N-\widehat{j}_n,a}(H)
-
\sum_{a=2}^{S} (\widehat{\eta}_n)^{2a} 2^{-2a\widehat{H}_n \widehat{j}_n} \kappa_{N-\widehat{j}_n,a}(\widehat{H}_n)
\bigg|,
\\
B_n^{(2)} &=
2^{2 \widehat{j}_nH}
\bigg|
{Q}_{\widehat{j}_n,N-\widehat{j}_n} - 
\sum_{a=1}^{S} \eta^{2a} 2^{-2aH\widehat{j}_n} \kappa_{N-\widehat{j}_n,a}(H)
\bigg|,
\\
V_n^{(1)} 
&=
2^{2 \widehat{j}_nH}
\Big |
\eta^{2} 2^{-2H\widehat{j}_n} \kappa_{N-\widehat{j}_n,1}(H)
-
\eta^2 2^{-2\widehat{j}_n\widehat{H}_n} \kappa_{N-\widehat{j}_n, 1}(\widehat{H}_n)
\Big|,
\\
V_n^{(2)} &= 
2^{2 \widehat{j}_nH}
\big | \widehat{Q}_{\widehat{j}_n,N-\widehat{j}_n,n} - {Q}_{\widehat{j}_n,N-\widehat{j}_n}\big|.
\end{align*}
The term $\kappa_{N-\widehat{j}_n, 1}^{-1}(\widehat{H}_n)$ can be ignored because $\kappa_{p,1}$ is a continuous function bounded away from $0$ on $[H_-, H_+]$, see Equation \eqref{eq:kappa1:bound}. We prove that $2^{2 \widehat{j}_n(\widehat{H}_n-H)}$, $(w_n^{c})^{-1} B_n^{(1)}$, $(w_n^{c})^{-1} B_n^{(2)}$, $(w_n^{c})^{-1} V_n^{(1)}$ and  $(w_n^{c})^{-1} V_n^{(2)}$ are bounded in probability uniformly over $\mathcal{D}$. Indeed, all these terms except $B_n^{(1)}$ can be treated in the same  way as in the proof of Proposition \ref{prop:first_estimator_eta}. Therefore, we only need to focus on $B_n^{(1)}$. By Lemma \ref{lem:dist:B}, we have
\begin{align*}
B_n^{(1)}
&=
2^{2 \widehat{j}_nH}
\Big |
B^{(S)}_{\widehat{j}_n, N-\widehat{j}_n} (H,\eta)
-
B^{(S)}_{\widehat{j}_n, N-\widehat{j}_n} (\widehat{\eta}_n,\widehat{H}_n)
\Big |
\\
&\leq
c_B 2^{2 \widehat{j}_nH} 2^{-4(H \wedge \widehat{H}_n)\widehat{j}_n}
(
\widehat{j}_n |\widehat{H}_n - H |
+
|
\widehat{\eta}_n
-
\eta
|
)
\\
&\leq
c_B 2^{-4(0 \wedge (\widehat{H}_n - H))\widehat{j}_n}
2^{-2H\widehat{j}_n}
(
\widehat{j}_n |\widehat{H}_n - H |
+
|
\widehat{\eta}_n
-
\eta
|
).
\end{align*}
Since $v_n^{-1}|\widehat{H}_n - H|$ is bounded in probability uniformly on $\mathcal{D}$ and $v_n \widehat{j}_n \leq v_n \log n \to 0$, the sequence $2^{-4(0 \wedge (\widehat{H}_n - H))\widehat{j}_n}$ is uniformly tight. Thus it is enough to show that $(w_n^{c})^{-1}2^{-2H\widehat{j}_n}
(
\widehat{j}_n |\widehat{H}_n - H |
+
|
\widehat{\eta}_n
-
\eta
|)$ is bounded in probability uniformly on $\mathcal{D}$. Recall that
\begin{align*}
\mathbb{P}_{H,\eta}
\Big(
\widehat{j}_n \in \big \{ \big\lfloor \tfrac{1}{2H+1} \log_2 n\big\rfloor - 1, \big\lfloor \tfrac{1}{2H+1} \log_2 n \big\rfloor \big\}
\Big) \to 1,
\end{align*}
so it is enough to prove that $(w_n^{c})^{-1}n^{-2H/(2H+1)}
\log n |\widehat{H}_n - H |
$ and $(w_n^{c})^{-1}n^{-2H/(2H+1)}
|
\widehat{\eta}_n
-
\eta
|$ are uniformly tight. This is indeed the case since $v_n (w_n^{c})^{-1}n^{-2H/(2H+1)}
\log n$ and $w_n (w_n^{c})^{-1}n^{-2H/(2H+1)}$ are bounded and $v_n^{-1} |\widehat{H}_n - H |$ and $w_n^{-1} |
\widehat{\eta}_n
-
\eta
|$ are uniformly tight.

\section*{Acknowledgments}
Mathieu Rosenbaum and Gr\'egoire Szymanski gratefully acknowledge the financial support of the \'Ecole Polytechnique chairs {\it Deep Finance and Statistics} and {\it Machine Learning and Systematic Methods}. Yanghui Liu is supported by the PSC-CUNY Award 64353-00 52.  The authors would further like to thank the Associate Editor and two referees for their careful reading of the paper and for their constructive comments, which led to significant improvements of the paper.

\bibliographystyle{imsart-number}
\bibliography{library_submitted}

\appendix

\section{Asymptotic expansion of the integrated volatility} \label{sec: first app}

This Section is of independent interest from the rest of this paper. We first recall a few notations, introduced in Section \ref{subsec:model}.\\

Let $\mathcal{D}$ be a compact subset of $(0,1) \times (0, \infty)$ and let $T > 0$. Fix also some arbitrary integer constant $S>0$ and some constant $0 < H^* < \min_{(H,\eta) \in \mathcal{D}} H$. We consider a measurable space $(\Omega, \mathcal{A})$ on which is defined a process $(\sigma_t)_{t \leq T}$ such that $\sigma$ is given by
\begin{align*}
\sigma_t^2 = \exp ( \eta W^H_t )
\end{align*}
where $(W^H_t)$ is a fractional Brownian motion with Hurst index $H$ under the probability measure $\mathbb{P}_{H, \eta}$. We will write $\mathbb{E}_{H, \eta}$ for the expectation under the probability measure $\mathbb{P}_{H, \eta}$. \\

For any $\alpha > 0$ we also define the best $\alpha$-Hölder constant of a function $f:[0,T] \to \mathbb{R}$ by
\begin{align*}
\mathcal{H}_\alpha(f) := \sup_{0 \leq s \neq t \leq T} \frac{|f(t) - f(s)|}{|t-s|^\alpha}
\end{align*}
and we write $\mathcal{H}^H_\alpha := \mathcal{H}_\alpha(W^H)$ for notational simplicity.

\begin{prop}
\label{prop:development_iv}
There exists a random variable $Z_0$, bounded in $L^2(\mathbb{P}_{H,\eta})$ uniformly on $\mathcal{D}$, such that for any $\delta > 0$ and $i$ such that $(i+1) \delta \leq T$, we have
\begin{align*}
\log \Big( \delta^{-1} \int_{i\delta}^{(i+1) \delta} \sigma_u^2 du \Big)
&=
\sum_{b=2}^{S} 
\sum_{s=1}^{S} \frac{(-1)^{s-1}}{s}
\sum_{\substack{\mathbf{r} \in \{ 1, \dots, S \}^s \\ \sum_j \mathbf{r}_j = b}} \prod_{j=1}^{s} \frac{ \eta^{\mathbf{r}_j}  }{\mathbf{r}_j!} 
	\frac{1}{\delta}
\int_{i\delta}^{(i+1)\delta}	(W^H_u-W^H_{i\delta})^{\mathbf{r}_j} du
\\
&\;\;\;\;
+ 
\frac{1}{\delta}
\int_{i\delta}^{(i+1)\delta} \eta W^H_u du + Z(i, \delta) \cdot \delta^{H^*(S+1)},
\end{align*}
where the random variables $Z(i, \delta)$ satisfy
$| Z(i, \delta) | \leq Z_0
$ and where we use the convention that $\sum_{b=2}^{S} \cdots = 0$ when $S=1$.
\end{prop}

\begin{proof}

Recall that for any real numbers $x$ and $a$, Taylor's formula gives
\begin{align*}
e^x 
&= \sum_{r=0}^{S} \frac{ (x-a)^r e^{a} }{r!} + \frac{ (x-a)^{S+1} }{S!} \int_0^1 (1-z)^{S} e^{a+z(x-a)} dz
\\
&= e^a \Big( 1+ \sum_{r=1}^{S} \frac{ (x-a)^r }{r!} + \frac{ (x-a)^{S+1}}{S!} \int_0^1 (1-z)^{S} e^{z(x-a)} dz \Big).
\end{align*}

Applying this equality with $x = \eta W^H_u$ and $a = \eta W^H_{i\delta}$, we obtain
\begin{align*}
e^{\eta W^H_u}
=
e^{\eta W^H_{i\delta}} 
\Big( 
	1
	+
	\sum_{r=1}^{S} 
	\frac{ \eta^{r} (W^H_u-W^H_{i\delta})^{r} }{r!} 
	+ 
	\frac{\eta^{S+1} (W^H_u-W^H_{i\delta})^{S+1}}{S!} 
	\int_0^1 (1-z)^{S} e^{\eta z(W^H_u-W^H_{i\delta})} dz 
\Big).
\end{align*}

Notice that since $H > H^*$, the random variable $\mathcal{H}^H_{H^*}$ is almost surely finite. Then 
\begin{align*}
\Big |
\eta^{S+1} (W^H_u-W^H_{i\delta})^{S+1}
	\int_0^1 (1-z)^{S} e^{\eta z(W^H_u-W^H_{i\delta})} dz 
\Big |
	\leq 
\eta^{S+1}
(\mathcal{H}^H_{H^*} )^{S+1}
| u - i\delta |^{H^*(S+1)}
\frac{e^{2 \eta \| W^H \|_\infty}}{S+1}.
\end{align*}

Therefore, for $i\delta \leq u \leq (i+1) \delta$, we have 
\begin{align}
\label{eq:approx_sigma_2}
e^{\eta W^H_u}
=
e^{\eta W^H_{i\delta}} 
\Big( 
	1
	+
	\sum_{r=1}^{S} 
	\frac{ \eta^{r} (W^H_u-W^H_{i\delta})^{r} }{r!} 
	+ 
	R^{H,S}(u)
	\delta^{H^*(S+1)}
\Big),
\end{align}
where 
\begin{align*}
| R^{H,S}(u) | 
\leq 
\eta^{S+1}
(\mathcal{H}^H_{H^*} )^{S+1}
| u/\delta - i |^{H^*(S+1)}
\frac{e^{2 \eta || W^H ||_\infty}}{(S+1)!}
\leq 
(\eta\mathcal{H}^H_{H^*} )^{S+1}
\frac{e^{2 \eta \| W^H \|_\infty}}{(S+1)!} = R_0^{H,S}
\end{align*}
and $R_0^{H,S}$ is a random variable independent of $\delta$ and $u$.\\

Then we integrate both sides of \eqref{eq:approx_sigma_2} and we take the logarithm. This yields 
\begin{align}
\label{eq:approx_integration_sigma_2}
\log \Big( \frac{1}{\delta} \int_{i\delta}^{(i+1)\delta} \sigma_u^2 du \Big)
&=
\eta W^H_{i\delta}
+
\log \Big( 1
	+
	\sum_{r=1}^{S} 
	\frac{ \eta^{r}  }{r!} 
	\mathfrak{I}_{i,\delta}^{H,r}
	+ 
	\widetilde{R}^{H,S}(i,\delta)
	\delta^{H^*(S+1)}
\Big)
\end{align}
where
\begin{align*}
\mathfrak{I}_{i,\delta}^{H,r}
=
\frac{1}{\delta}
\int_{i\delta}^{(i+1)\delta}	(W^H_u-W^H_{i\delta})^{r} du
\end{align*}
and where
$
\widetilde{R}^{H,S}(i,\delta)
=
\frac{1}{\delta} \int_{i\delta}^{(i+1)\delta}	
	R^{H,S}(u)
	du
$ is still dominated by $R_0^{H,S}$. We will now expand the logarithm on the right-hand side of \eqref{eq:approx_integration_sigma_2}. Taylor expansion of the logarithm gives
\begin{align*}
\log (1+x) = \sum_{s=1}^{S} \frac{(-1)^{s-1}}{s} x^s
+
(-1)^S
\int_{0}^x \frac{(x-t)^S}{(1+t)^{S+1}}dt.
\end{align*}
We apply this last formula to 
\begin{align*}
x 
=
\frac{1}{\delta} \int_{i\delta}^{(i+1)\delta} \frac{\sigma_u^2}{\sigma^2_{i\delta}} du - 1 = \sum_{r=1}^{S} 
	\frac{ \eta^{r}  }{r!} \mathfrak{I}_{i,\delta}^{H,r}
	+ 
	\widetilde{R}^{H,S}(i,\delta)
	\delta^{H^*(S+1)}.
\end{align*}
This quantity is independent of $S$ so we also have $x = \sum_{r=1}^{S'}
	\frac{ \eta^{r}  }{r!} \mathfrak{I}_{i,\delta}^{H,r}
	+ 
	\widetilde{R}^{H,S'}(i,\delta)
	\delta^{(S'+1)H^*}
$ for any $S'$. In particular, with $S' = 0$, we have $x = \widetilde{R}^{H,0}(i,\delta)
	\delta^{H^*}$. We obtain
\begin{align}
\label{eq:approx_before_simplification}
&\log \Big( \frac{1}{\delta} \int_{i\delta}^{(i+1)\delta} \sigma_u^2 du \Big)\\
&=
\eta W^H_{i\delta}
+
\sum_{s=1}^{S} \bigg[ \frac{(-1)^{s-1}}{s}
\Big(
\sum_{r=1}^{S} 
	\frac{ \eta^{r}  }{r!} 
	\mathfrak{I}_{i,\delta}^{H,r}
	+ 
	\widetilde{R}^{H,S}(i,\delta)
	\delta^{H^*(S+1)}
\Big)^{s}
\bigg]
+
T^{H}(i, \delta) \nonumber
\end{align}
where 
$T^{H}(i, \delta) =
(-1)^S
\int_{0}^{\widetilde{R}^{H,0}(i,\delta)
	\delta^{H^*}} \frac{(\widetilde{R}^{H,0}(i,\delta)
	\delta^{H^*}-t)^S}{(1+t)^{S+1}}dt
$. Notice in addition that
\begin{align*}
\Big | 
(-1)^S
\int_{0}^x \frac{(x-t)^k}{(1+t)^{k+1}}dt
\Big | \leq 
\begin{cases}
\frac{|x|^{S+1}}{S+1} & \text {for } x \geq 0,
\\
\frac{|x|^{S+1}}{(1+x)^{S+1}(S+1)}
& \text {for } x \leq 0,
\end{cases}
\end{align*}
which translates here into
\begin{align*}
\Big | 
T^{H}(i, \delta)
\Big | \leq 
\begin{cases}
\frac{|\widetilde{R}^{H,0}(i,\delta)
	\delta^{H^*}|^{S+1}}{S+1} & \text {for } 1+x \geq 0,
\\
\frac{|\widetilde{R}^{H,0}(i,\delta)
	\delta^{H^*}|^{S+1}}{(\frac{1}{\delta} \int_{i\delta}^{(i+1)\delta} \sigma_u^2 du )^{S+1}(S+1)}\sigma^{2(S+1)}_{i\delta}
& \text {for } 1+x \leq 0.
\end{cases}
\end{align*}
Since $\exp( -\eta \| W^H \|_\infty) \leq \sigma_u^2 \leq \exp( \eta \| W^H \|_\infty)$ and $|\widetilde{R}^{H,0}(i,\delta)| \leq R_0^{H,0} = \eta \mathcal{H}^H_{H^*}
e^{2 \eta \| W^H \|_\infty}$, we deduce that 
\begin{align*}
\big | 
T^{H}(i, \delta)
\big | \leq 
\frac{|\widetilde{R}^{H,0}(i,\delta)
	\delta|^{S+1}}{S+1}
e^{2(S+1)\eta \| W^H \|_\infty}
\leq 
\frac{1}{S+1}\Big( \eta \mathcal{H}^H_{H^*}
e^{3\eta \| W^H \|_\infty} \delta
\Big)^{S+1}.
\end{align*}

By Kolmogorov's continuity criterion (see {\it e.g.} Theorem 2.1 in \cite{revuz1999continuous}), the random variable $\mathcal{H}^H_{H^*}$ has moments of all orders bounded independently of $H$. Moreover, $\| W^H \|_\infty$ has exponential moments of all orders bounded independently of $H$. Hence $T^{H}(i, \delta)$ satisfies the condition required for $Z$ in Proposition \ref{prop:development_iv}. We next focus on the expression 
$$
\Big(
\sum_{r=1}^{S} 
	\frac{ \eta^{r}  }{r!} 
	\mathfrak{I}_{i,\delta}^{H,r}
	+ 
	\widetilde{R}^{H,S}(i,\delta)
	\delta^{H^*(S+1)}
\Big)^{s}.$$
We expand the power $s$ and remove all the terms of order smaller than $\delta^{H^*(S+1)}$. Note that $| \mathfrak{I}_{i,\delta}^{H,r} | \leq \delta^{H^*r} (\mathcal{H}^H_{H^*})^r $. Thus,
\begin{align*}
\Big(
\sum_{r=1}^{S} 
	\frac{ \eta^{r}  }{r!} 
	\mathfrak{I}_{i,\delta}^{H,r}
	+ 
	\widetilde{R}^{H,S}(i,\delta)
	\delta^{H^*(S+1)}
\Big)^{s}
=
\sum_{\mathbf{r}} \prod_{j=1}^{s} X_{\mathbf{r}_j},
\end{align*}
where the sum is taken over all $\mathbf{r} = (\mathbf{r}_1, \dots, \mathbf{r}_{s})$
with $1 \leq \mathbf{r}_j \leq S+1$ and 
where we write $X_r = \frac{ \eta^{r}  }{r!} 
	\mathfrak{I}_{i,\delta}^{H,r}$ for $r \leq S$ and $X_{S+1} = \widetilde{R}^{H,S}(i,\delta)
	\delta^{H^*(S+1)}$. 
By the preceding remark, we have
\begin{align*}
\Big| \prod_{j=1}^{s} X_{\mathbf{r}_j} \Big| 
\leq
\mathfrak{C}^{|\{j: \, \mathbf{r}_j = S+1\}|}
\delta^{H^*\sum_j \mathbf{r}_j}
(\mathcal{H}^H_{H^*})^{\sum_j \mathbf{r}_j},
\end{align*}
where $|A|$ denotes the cardinality of the set $A$ and $\mathfrak{C}$ is a random variable with moments of all orders bounded independently of $H$. Proceeding as for $T^{H}(i, \delta)$, we can show that $\prod_{j=1}^{s} X_{\mathbf{r}_j}$ can be incorporated in the remainder term of Proposition \ref{prop:development_iv} whenever $\sum_j \mathbf{r}_j \geq S+1$. Therefore, we can restrict the sum $\sum_{\mathbf{r}} \prod_{j=1}^{s} X_{\mathbf{r}_j}
$ to indices $\mathbf{r}$ satisfying $\sum_j \mathbf{r}_j \leq S$. In that case, $\mathbf{r}_j \leq S$ and we get
\begin{align*}
\Big(
\sum_{r=1}^{S} 
	\frac{ \eta^{r}  }{r!} 
	\mathfrak{I}_{i,\delta}^{H,r}
	+ 
	\widetilde{R}^{H,S}(i,\delta)
	\delta^{H^*(S+1)}
\Big)^{s}
=
\sum_{\mathbf{r}} \prod_{j=1}^{s} \frac{ \eta^{\mathbf{r}_j}  }{\mathbf{r}_j!} 
	\mathfrak{I}_{i,\delta}^{H,\mathbf{r}_j}
+ \text{remainder of order } \delta^{H^*(S+1)}.
\end{align*}

Plugging this into \eqref{eq:approx_before_simplification} and using the symbol $\approx$ to indicate   the presence of a remainder term of order $ \delta^{H^*(S+1)}$ that can be incorporated in the term $Z$ of Proposition \ref{prop:development_iv}, we obtain
\begin{align*}
\log \Big( \frac{1}{\delta} \int_{i\delta}^{(i+1)\delta} \sigma_u^2 du \Big)
&\approx
\eta W^H_{i\delta}
+
\sum_{s=1}^{S} 
\bigg[
\frac{(-1)^{s-1}}{s}
\sum_{\substack{\mathbf{r} \\ \sum_j \mathbf{r}_j \leq S}} \prod_{j=1}^{s} \frac{ \eta^{\mathbf{r}_j}  }{\mathbf{r}_j!} 
	\mathfrak{I}_{i,\delta}^{H,\mathbf{r}_j}
	\bigg]
\\
&=
\eta W^H_{i\delta}
+
\sum_{b=1}^{S} 
\sum_{s=1}^{S} 
\bigg[
\frac{(-1)^{s-1}}{s}
\sum_{\substack{\mathbf{r} \\ \sum_j \mathbf{r}_j = b}} \prod_{j=1}^{s} \frac{ \eta^{\mathbf{r}_j}  }{\mathbf{r}_j!} 
	\mathfrak{I}_{i,\delta}^{H,\mathbf{r}_j}
\bigg]
\\
&=
\frac{1}{\delta}
\int_{i\delta}^{(i+1)\delta} \eta W^H_u du
+
\sum_{b=2}^{S} 
\sum_{s=1}^{S} 
\bigg[
\frac{(-1)^{s-1}}{s}
\sum_{\substack{\mathbf{r} \\ \sum_j \mathbf{r}_j = b}} \prod_{j=1}^{s} \frac{ \eta^{\mathbf{r}_j}  }{\mathbf{r}_j!} 
	\mathfrak{I}_{i,\delta}^{H,\mathbf{r}_j}
	\bigg].
\end{align*}

\end{proof}

\section{Some correlation estimates for Gaussian vectors}
 \label{sec: second app}
First, we recall   Isserlis' Theorem (see \cite{isserlis1918formula}), which allows us to compute the expectation of a product of zero-mean correlated normal random variables. 

\begin{thm}[Isserlis' Theorem]
\label{thm:gaussian_moments}
Suppose that $(X_1, \dots, X_{2n})$ is a centered Gaussian vector. Then we have
\begin{align*}
\mathbb{E}\big[\prod_i X_i\big] = \sum_{P} \prod_{(i,j) \in P} \mathbb{E}\big[X_i X_j\big]
\end{align*}
where the sum is over all the partitions $P$ of $\{1,\dots, 2n\}$ into subsets of exactly two elements.
\end{thm}

In particular, we have:
\begin{align}
\label{eq:cov_4_gaussian}
\Cov(X_1X_2, X_3X_4) = \mathbb{E}[X_1 X_3]\mathbb{E}[X_2 X_4] + \mathbb{E}[X_1 X_4] \mathbb{E}[X_2 X_3].
\end{align}

\begin{prop}
\label{prop:covariance_gaussian_product}
Suppose that $(X_1, \dots, X_{n+m})$ is a centered Gaussian vector, where $n$ and $m$ are two integers such that $n+m$ is even. Suppose in addition that for any $i \leq n$ and $j \geq n+1$, we have
\begin{align*}
\big | \mathbb{E}{[X_iX_j]} \big | \leq \rho \sigma^2
\end{align*}
for some $0 \leq \rho \leq 1$ and $\sigma \geq 0$, and suppose that $\mathbb{E} [X_i^2 ]\leq \sigma^2$ for any $i \geq 1$. Then 
\begin{align*}
\big | \Cov(\prod_{i=1}^n X_i, \prod_{j=n+1}^{n+m} X_j) \big | \leq C \rho^{\alpha} \sigma^{n+m},
\end{align*}
where $\alpha = 1$ if $n$ is odd and $\alpha = 2$ if $n$ is even and $C$ is a constant depending only on $n$ and $m$.
\end{prop}

\begin{proof}
Denote by $\mathcal{P}_2(E)$ the set of all partitions of the set $E$ in subsets of exactly $2$ elements. Then we have by Theorem \ref{thm:gaussian_moments} that
\begin{align*}
\Cov(\prod_{i=1}^n X_i, \prod_{j=n+1}^{n+m} X_j)
&=
\sum_{P \in \mathcal{P}_2(\{ 1, \dots, n+m\})} \prod_{(i,j) \in P} \mathbb{E}[X_i X_j] \\ &\;\;\;\;
- 
\bigg[ \sum_{P \in \mathcal{P}_2(\{ 1, \dots, n\}) } \prod_{(i,j) \in P} \mathbb{E}[X_i X_j] \bigg]
\bigg[ \sum_{P \in \mathcal{P}_2(\{ n+1, \dots, n+m\})} \prod_{(i,j) \in P} \mathbb{E}[X_i X_j]\bigg].
\end{align*}

Moreover, the mapping
\begin{align*}
\mathcal{P}_2(\{ 1, \dots, n\})
\times
\mathcal{P}_2(\{ n+1, \dots, n+m\})
&\to
\mathcal{P}_2(\{ 1, \dots, n+m\})
\\
(P,Q) &\mapsto P \cup Q
\end{align*}
is injective and its image $\mathcal{Q}(n,m)$ is exactly the set of the partitions $P$ of $\{ 1, \dots, n+m\}$  such that if $(i,j) \in P$ with $i\leq n$, then $j \leq n$ as well. Thus, 
\begin{align*}
\Cov(\prod_{i=1}^n X_i, \prod_{j=n+1}^{n+m} X_j)
&=
\sum_{P \in \mathcal{P}_2(\{ 1, \dots, n+m\}) \backslash \mathcal{Q}(n,m)} \prod_{(i,j) \in P} \mathbb{E}[X_i X_j] .
\end{align*}
Since there are finitely many partitions of $\{ 1, \dots, n+m\}$, the proof is complete once we can prove that for any $P \in \mathcal{P}_2(\{ 1, \dots, n+m\}) \backslash \mathcal{Q}(n,m)$, there exists a constant $C$ depending only on $n$ and $m$ such that 
\begin{align*}
\Big|
\prod_{(i,j) \in P} \mathbb{E}[X_i X_j] 
\Big|
\leq 
\rho^{\alpha} \sigma^{n+m}.
\end{align*}
Consider such a partition $P$. Then there is at least one pair $(i_0,j_0) \in P$ such that $i_0 \leq n$ and $j_0 \geq n+1$. Thus 
\begin{align*}
\Big|
\prod_{(i,j) \in P} \mathbb{E}[X_i X_j] 
\Big|
\leq 
\Big|
\mathbb{E}[X_{i_0} X_{j_0}] 
\prod_{(i,j) \in P\backslash \{(i_0,j_0)\}} \mathbb{E}[X_i X_j] 
\Big|
\leq 
\rho \sigma^{2} \prod_{(i,j) \in P\backslash \{(i_0,j_0)\}} \sigma^{2},
\end{align*}
which concludes the case where $n$ is odd since $\# \big(P\backslash \{(i_0,j_0)\}\big) = (n+m-2)/2$.\\

Suppose now that $n$ is even. Then there must be another pair $(i_1,j_1) \in P$ such that $i_1 \leq n$ and $j_1 \geq n+1$, $(i_1,j_1) \neq (i_0,j_0)$ because there is no partition of  $\{1,\dots,n\} \backslash \{i_0\}$ into subsets of $2$ elements. Then 
\begin{align*}
\bigg|
\prod_{(i,j) \in P} \mathbb{E}[X_i X_j] 
\bigg|
\leq 
\bigg|
\mathbb{E}[X_{i_0} X_{j_0}] 
\mathbb{E}[X_{i_1} X_{j_1}] 
\prod_{(i,j) \in P\backslash \{(i_0,j_0), (i_1,j_1)\}} \mathbb{E}[X_i X_j] 
\bigg|
\leq 
\rho^2 \sigma^{n+m}.
\end{align*}
\end{proof}

\section{Log-moments of \texorpdfstring{$\chi^2$}{chi2}-variables}
\label{sec:log_mom_chi2}
We write $\Gamma$ for the usual gamma function, defined by
\begin{align}
\label{eq:def:gamma}
\Gamma(t) = \int_0^\infty t^{x-1}e^{-t} dt.
\end{align}

We also introduce the polygamma function $\psi^{(k)}$, which is the $k$-th logarithmic derivative of the gamma function. Thus, $\psi^{(0)} = \tfrac{\Gamma'}{\Gamma}$ and explicit computations also give
\begin{align*}
\psi^{(1)} &= \tfrac{\Gamma''}{\Gamma} - (\psi^{(0)})^2,
\\
\psi^{(2)} &= \tfrac{\Gamma^{(3)}}{\Gamma}
-
(\psi^{(0)})^3
-
3 \psi^{(0)} \psi^{(1)},
\\
\psi^{(3)} &= \tfrac{\Gamma^{(4)}}{\Gamma}
- (\psi^{(0)})^4
- 6 (\psi^{(0)})^2 \psi^{(1)}
- 4 \psi^{(0)} \psi^{(2)}
- 3 (\psi^{(1)})^2.
\end{align*}
Note that we can explicitly express ratios $\tfrac{\Gamma^{(k)}}{\Gamma}$ with $k\leq 4$ in terms of polygamma functions from these equations.

\begin{lem}
\label{lem:log_mom_chi2}
Suppose that for $m\geq 1$, $X_m$ is a random variable following a $\chi^2$-distribution with $m$ degrees of freedom. We write $Y_m = \log(m^{-1}X_m)$. Then there exists $C > 0$ such that for any $m \geq 1$,
\begin{align*}
\Var (Y_m) &= \psi^{(1)}(\tfrac{m}{2}) \leq C m^{-1},
\\
\mathbb{E}  \big[
Y_m^4
\big]
&\leq C m^{-2}.
\end{align*}
\end{lem}

\begin{proof}
We know that
\begin{align*}
\mathbb{E}  \big[
\exp(t Y_{m})
\big]
=
(\tfrac{2}{m})^t \frac{\Gamma(\tfrac{m}{2}+t)}{\Gamma(m)}.
\end{align*}
Moments of $Y_m$ can be derived through the classical formula 
\begin{align*}
\mathbb{E}  \big[
Y_{m}^k
\big]
&=
\frac{d^k}{dt^k}\Big\rvert_{t=0} \Big( (\tfrac{2}{m})^t \frac{\Gamma(\tfrac{m}{2}+t)}{\Gamma(m)} \Big).
\end{align*}
Thus we get 
\begin{align*}
\mathbb{E}  \big[
Y_{m}
\big]&=
-\log(\tfrac{m}{2})+\frac{\Gamma'(\tfrac{m}{2})}{\Gamma(\tfrac{m}{2})},
\\
\mathbb{E}  \big[
Y_{m}^2
\big]&=
\log^2(\tfrac{m}{2})-2\log(\tfrac{m}{2}) \frac{\Gamma'(\tfrac{m}{2})}{\Gamma(\tfrac{m}{2})}+\frac{\Gamma''(\tfrac{m}{2})}{\Gamma(\tfrac{m}{2})},
\\
\mathbb{E}  \big[
Y_{m}^4
\big]&=
\log^4(\tfrac{m}{2})
- 4 \log^3(\tfrac{m}{2}) \frac{\Gamma'(\tfrac{m}{2})}{\Gamma(\tfrac{m}{2})} + 6 \log^2(\tfrac{m}{2})\frac{\Gamma''(\tfrac{m}{2})}{\Gamma(\tfrac{m}{2})}
- 4 \log(\tfrac{m}{2}) \frac{\Gamma^{(3)}(\tfrac{m}{2})}{\Gamma(\tfrac{m}{2})} 
+ \frac{\Gamma^{(4)}(\tfrac{m}{2})}{\Gamma(\tfrac{m}{2})}.
\end{align*}
We then rewrite $\Var (Y_m)$ and $\mathbb{E} \big[
Y_{m}^4
\big]$ in terms of these polygamma functions as follows:
\begin{align*}
\Var (Y_m) 
&= \psi^{(1)}(\tfrac{m}{2}),
\\
\mathbb{E} \big[
Y_{m}^4
\big]&=
\log^4(\tfrac{m}{2})
- 4 \log^3(\tfrac{m}{2}) \psi^{(0)}(\tfrac{m}{2}) 
+ 6 \log^2(\tfrac{m}{2}) (\psi^{(1)}(\tfrac{m}{2})  + \psi^{(0)}(\tfrac{m}{2})^2)
\\
&\;\;\;\;- 4 \log(\tfrac{m}{2}) (\psi^{(2)}(\tfrac{m}{2})  + \psi^{(0)}(\tfrac{m}{2})^3 + 3 \psi^{(0)}(\tfrac{m}{2})  \psi^{(1)}(\tfrac{m}{2}) )
\\
&\;\;\;\;+ \psi^{(3)}(\tfrac{m}{2})  + \psi^{(0)}(\tfrac{m}{2})^4
+ 6 \psi^{(0)}(\tfrac{m}{2})^2 \psi^{(1)}(\tfrac{m}{2}) 
+ 4 \psi^{(0)}(\tfrac{m}{2})  \psi^{(2)}(\tfrac{m}{2}) 
+ 3 \psi^{(1)}(\tfrac{m}{2})^2.
\end{align*}
An asymptotic expansion of the polygamma functions $\psi^{(k)}(x)$ when $x \to \infty$ is given by Equation 
5.15.9	in the internet appendix of \cite{olver2010nist}:
\begin{align*}
\psi^{(0)}(x) &= \log(x) - \tfrac{1}{2}x^{-1} - \tfrac{1}{12}x^{-2} + O(x^{-3}),
\\
\psi^{(1)}(x) &= x^{-1} + \tfrac{1}{2}x^{-2} + O(x^{-3}),
\\
\psi^{(2)}(x) &= -x^{-2} + O(x^{-3}),
\\
\psi^{(3)}(x) &= O(x^{-3}).
\end{align*}
Plugging these asymptotic expansions into the explicit expressions of $\Var (Y_m) $ and $\mathbb{E} \big[
Y_m^4
\big]$, we get
\begin{align*}
\Var (Y_m) &= \psi^{(1)}(\tfrac{m}{2}) = \tfrac{2}{m} + O(m^{-2}),
\\
\mathbb{E}  \big[
Y_m^4
\big]
&= 12m^{-2} + o(m^{-2}),
\end{align*}
which concludes the proof.
\end{proof}

\end{document}